\documentclass[reqno,11pt]{amsart}
\usepackage{amsmath,amssymb,amsthm,graphicx,mathrsfs,bbm,url}
\usepackage{amsthm}
\usepackage{wrapfig}
\usepackage{soul}
\usepackage{enumitem}
\usepackage{mathtools}
\usepackage[usenames,dvipsnames]{xcolor}
\usepackage[colorlinks=true,linkcolor=Red,citecolor=G	reen]{hyperref}
\usepackage[super]{nth}
\usepackage[open, openlevel=2, depth=3, atend]{bookmark}
\hypersetup{pdfstartview=XYZ}
\usepackage[font=footnotesize]{caption}
\usepackage{subcaption}
\usepackage{caption}
\usepackage{a4wide}
\usepackage{tikz}
\usepackage{setspace}
\usetikzlibrary{decorations.pathreplacing}
\usetikzlibrary{arrows}
\usetikzlibrary{shapes.misc}
\usetikzlibrary{shapes.symbols}
\usetikzlibrary{patterns}
\usepackage{soul}
\usepackage{stmaryrd}

\usepackage{epstopdf}
 
\usepackage{hyperref}

\newcommand{\C}{\mathbb{C}}
\newcommand{\R}{\mathbb{R}}

\newcommand{\Z}{\mathbb{Z}}

\newcommand{\M}{\mathcal{M}}
\newcommand{\N}{\mathbb{N}}

\newcommand{\V}{\mathbb{V}}
\newcommand{\X}{\mathbf{X}}

\newcommand{\HH}{\mathbb{H}}
\newcommand{\Ss}{\mathbb{S}}
\newcommand{\eps}{\varepsilon}

\newcommand{\mc}{\mathcal}

\newcommand{\End}{\mathrm{End}}

\newcommand{\norm}[1]{\left\lVert#1\right\rVert}

\DeclareMathOperator{\Tr}{Tr}

\DeclareMathOperator{\re}{Re}
\DeclareMathOperator{\rk}{rank}

\DeclareMathOperator{\dd}{d}

\DeclareMathOperator{\id}{Id}
\DeclareMathOperator{\E}{\mathcal{E}}

\DeclareMathOperator{\Hom}{Hom}
\DeclareMathOperator{\Op}{Op}

\DeclareMathOperator{\ran}{ran}

\DeclareMathOperator{\e}{\mathbf{e}}
\DeclareMathOperator{\RR}{\mathbf{R}}

\theoremstyle{plain}
\newtheorem{theorem}{Theorem}[section]
\newtheorem{definition}[theorem]{Definition}
\newtheorem{lemma}[theorem]{Lemma}
\newtheorem{remark}[theorem]{Remark}
\newtheorem{prop}[theorem]{Proposition}

\newtheorem{corollary}[theorem]{Corollary}

\numberwithin{equation}{section}

\begin{document}

\begin{abstract}
Let $(M,g)$ be a smooth Anosov Riemannian manifold and $\mc{C}^\sharp$ the set of its primitive closed geodesics. Given a Hermitian vector bundle $\E$ equipped with a unitary connection $\nabla^{\E}$, we define $\mc{T}^\sharp(\E, \nabla^{\E})$ as the sequence of traces of holonomies of $\nabla^{\E}$ along elements of $\mc{C}^\sharp$. This descends to a homomorphism on the additive moduli space $\mathbb{A}$ of connections up to gauge $\mc{T}^\sharp: (\mathbb{A}, \oplus) \to \ell^\infty(\mc{C}^\sharp)$, which we call the \emph{primitive trace map}. It is the restriction of the well-known \emph{Wilson loop} operator to primitive closed geodesics.

The main theorem of this paper shows that the primitive trace map $\mc{T}^\sharp$ is locally injective near generic points of $\mathbb{A}$ when $\dim(M) \geq 3$. We obtain global results in some particular cases: flat bundles, direct sums of line bundles, and general bundles in negative curvature under a spectral assumption which is satisfied in particular for connections with small curvature. As a consequence of the main theorem, we also derive a spectral rigidity result for the connection Laplacian.

The proofs are based on two new ingredients: a Liv\v{s}ic-type theorem in hyperbolic dynamical systems showing that the cohomology class of a unitary cocycle is determined by its trace along closed primitive orbits, and a theorem relating the local geometry of $\mathbb{A}$ to the Pollicott-Ruelle resonance near zero of a certain natural transport operator.

\end{abstract}

\title{The holonomy inverse problem}
\author[M. Ceki\'{c}]{Mihajlo Ceki\'{c}}
\date{\today}
\address{Laboratoire de Math\'{e}matiques d'Orsay, CNRS, Universit\'{e} Paris-Saclay, 91405 Orsay, France}
\curraddr{Institut f\"ur Mathematik, Universit\"at Z\"urich, Winterthurerstrasse 190, CH-8057 Z\"urich, Switzerland}
\email{mihajlo.cekic@math.uzh.ch}

\author[T. Lefeuvre]{Thibault Lefeuvre}

\address{Université de Paris and Sorbonne Université, CNRS, IMJ-PRG, F-75006 Paris, France.}

\email{tlefeuvre@imj-prg.fr}

\maketitle

\setcounter{tocdepth}{1}
\tableofcontents

\newpage

\section{Introduction}

\subsection{Primitive trace map, local injectivity}

Let $(M,g)$ be a smooth closed Riemannian Anosov manifold such as a manifold of negative sectional curvature \cite{Anosov-67}. Recall that this means that there exists a continuous flow-invariant splitting of the tangent bundle to the unit tangent bundle $\M := SM$:
\[
T\M = \R X \oplus E_s \oplus E_u,
\]
such that:
\begin{equation}
\label{equation:anosov}
\begin{array}{l}
\forall t \geq 0, \forall v \in E_s, ~~ |\dd\varphi_t(v)| \leq Ce^{-t\theta}|v|, \\
\forall t \leq 0, \forall v \in E_u, ~~ |\dd\varphi_t(v)| \leq Ce^{-|t|\theta}|v|,
\end{array}
\end{equation}
where $(\varphi_t)_{t \in \R}$ is the geodesic flow on $\M$ generated by the vector field $X$, and the constants $C,\theta > 0$ are uniform and the metric $|\cdot|$ is arbitrary. 

Let $\E \rightarrow M$ be a smooth Hermitian vector bundle. We denote by $\mc{A}_{\E}$ the affine space of smooth unitary connections on $\E$ and $\mathbb{A}_{\E}$ the moduli space of connections up to gauge-equivalence, namely a point $\mathfrak{a} \in \mathbb{A}_{\E}$ is an orbit $\mathfrak{a} := \left\{p^*\nabla^{\E} ~|~p \in C^\infty(M,\mathrm{U}(\E))\right\}$ of gauge-equivalent connections, where $\nabla^{\E} \in \mathfrak{a}$ is arbitrary and $p^*\nabla^{\E}(\bullet) := p^{-1} \nabla^{\E} (p \bullet)$ is the pullback connection. We let $\mc{C} = \left\{c_1,c_2,...\right\}$ be the set of free homotopy classes of loops on $M$ which is known to be in one-to-one correspondence with closed geodesics \cite{Klingenberg-74}. More precisely, given $c \in \mc{C}$, there exists a unique closed geodesic $\gamma_g(c) \subset M$ in the class $c$. It will be important to make a difference between \emph{primitive} and \emph{non-primitive} homotopy classes (resp. closed geodesics): a free loop is said to be primitive if it cannot be homotoped to a certain power (greater or equal than $2$) of another free loop. The set of primitive classes defines a subset $\mc{C}^\sharp = \left\{c_1^\sharp,c_2^\sharp,...\right\} \subset \mc{C}$.

Given a class $\mathfrak{a} \in \mathbb{A}_{\E}$, a unitary connection $\nabla^{\E} \in \mathfrak{a}$ and an arbitrary point $x_{c^\sharp} \in \gamma_g(c^\sharp)$ (for some $c^\sharp \in \mc{C}^\sharp$), the parallel transport $\mathrm{Hol}_{\nabla^{\E}}(c^\sharp) \in \mathrm{U}(\E_{x_{c^\sharp}})$, starting at $x_{c^\sharp}$, with respect to $\nabla^{\E}$ and along $\gamma_g({c^\sharp})$ depends on the choice of representative $\nabla^{\E} \in \mathfrak{a}$ since two gauge-equivalent connections have conjugate holonomies. However, the trace does not depend on a choice of $\nabla^{\E} \in \mathfrak{a}$ and the \emph{primitive trace map}:
\begin{equation}
\label{equation:trace}
\mc{T}^\sharp : \mathbb{A}_{\E} \ni \mathfrak{a} \mapsto \left(\Tr\left(\mathrm{Hol}_{\nabla^{\E}}(c^\sharp_1)\right), \Tr\left(\mathrm{Hol}_{\nabla^{\E}}(c^\sharp_2)\right), ...\right) \in \ell^\infty(\mc{C}^\sharp),
\end{equation}
is therefore well-defined. Observe that the data of the primitive trace map is a rather weak information: in particular, it is \emph{not} (\emph{a priori}) equivalent to the data of the conjugacy class of the holonomy along each closed geodesic (and the latter is the same as the non-primitive trace map, where one considers \emph{all} closed geodesics). One of the main results of this paper is the following:

\begin{theorem}
\label{theorem:injectivity}
Let $(M,g)$ be a smooth Anosov Riemannian manifold of dimension $\geq 3$ and let $\E \rightarrow M$ be a smooth Hermitian vector bundle. Let $\mathfrak{a}_0 \in \mathbb{A}_{\E}$ be a \emph{generic} point. Then, the primitive trace map is \emph{locally injective} near $\mathfrak{a}_0$.
\end{theorem}

By \emph{local injectivity}, we mean the following: there exists $N \in \N$ (independent of $\mathfrak{a}_0$) such that $\mc{T}^\sharp$ is locally injective in the $C^N$-quotient topology on $\mathbb{A}_{\E}$. In other words, for any element $\nabla^{\E}_0 \in \mathfrak{a}_0$, there exists $\eps > 0$ such that the following holds; if $\nabla_{1,2}^{\E}$ are two smooth unitary connections such that $\|p_i^*\nabla_{i}^{\E} - \nabla^{\E}_0\|_{C^N} < \eps$ for some $p_i \in C^\infty(M,\mathrm{U}(\E))$, and $\mc{T}^\sharp(\nabla_1^{\E}) = \mc{T}^\sharp(\nabla_2^{\E})$, then $\nabla_1^{\E}$ and $\nabla_2^{\E}$ are gauge-equivalent. 

We say that a point $\mathfrak{a}$ is \emph{generic} if it enjoys the following two features: 

\begin{itemize}\label{def:generic}
\item[\textbf{(A)}]  $\mathfrak{a}$ is \textbf{opaque}. By definition (see \cite[Section 5]{Cekic-Lefeuvre-20}), this means that for all $\nabla^{\E} \in \mathfrak{a}$, the parallel transport map along geodesics does not preserve any non-trivial subbundle $\mc{F} \subset \E$ (i.e. $\mc{F}$ is preserved by parallel transport along geodesics if and only if $\mc{F} = \left\{0\right\}$ or $\mc{F} = \E$). This was proved to be equivalent to the fact that the Pollicott-Ruelle resonance at $z=0$ of the operator $\X := \pi^* \nabla^{\End}_X$ has multiplicity equal to $1$, with resonant space $\C \cdot \mathbbm{1}_{\E}$ (here $\pi : SM \rightarrow M$ is the projection; $\nabla^{\End}$ is the induced connection on the endomorphism bundle, see \S\ref{ssection:connections} for further details);

\item[\textbf{(B)}] $\mathfrak{a}$ has \textbf{solenoidally injective generalized X-ray transform} $\Pi^{\End(\E)}_1$ on twisted $1$-forms with values in $\End(\E)$. This last assumption is less easy to describe in simple geometric terms: roughly speaking, the X-ray transform is an operator of integration of symmetric $m$-tensors along closed geodesics. For vector-valued symmetric $m$-tensors, this might not be well-defined, and one needs a more general (hence, more abstract) definition involving the residue at $z=0$ of the meromorphic extension of the family $\C \ni z \mapsto (-\X-z)^{-1}$, see \S\ref{section:twisted}. 
\end{itemize}
It was shown in previous articles \cite{Cekic-Lefeuvre-20,Cekic-Lefeuvre-21-2} that in dimension $n \geq 3$, properties \textbf{(A)} and \textbf{(B)} are satisfied on an open dense subset $\omega \subset \mathbb{A}_{\E}$ with respect to the $C^N$-quotient topology.\footnote{More precisely, there exists $N \in \N$ and a subset $\Omega \subset \mc{A}_{\E}$ of the (affine) Fr\'echet space of smooth affine connections on $\E$ such that $\omega = \pi_{\E}(\Omega)$ (where $\pi_{\E} : \mc{A}_{\E} \rightarrow \mathbb{A}_{\E}$ is the projection) and
\begin{itemize}
\item $\Omega$ is invariant by the action of the gauge-group, namely $p^*\Omega = \Omega$ for all $p \in C^\infty(M,\mathrm{U}(\E))$;
\item $\Omega$ is open, namely for all $\nabla^{\E}_0 \in \Omega$, there exists $\eps > 0$ such that if $\nabla^{\E} \in \mc{A}_{\E}$ and $\|\nabla^{\E}-\nabla^{\E}_0\|_{C^N} < \eps$, then $\nabla^{\E} \in \Omega$;
\item $\Omega$ is dense, namely for all $\nabla^{\E}_0 \in \mc{A}_{\E}$, for all $\eps > 0$, there exists $\nabla^{\E} \in \Omega$ such that $\|\nabla^{\E}-\nabla^{\E}_0\|_{C^N} < \eps$;
\item Connections in $\Omega$ satisfy properties \textbf{(A)} and \textbf{(B)}.
\end{itemize}} When the reference connection $\mathfrak{a}$ satisfies only the property \textbf{(A)} (this is the case for the product connection on the trivial bundle for instance), we are able to show a \emph{weak local injectivity} result, see Theorem \ref{thm:weaklocal}. 

We note that the gauge class of a connection is uniquely determined from the holonomies along \emph{all} closed loops \cite{Barrett-91, Kobayashi-54} and that in mathematical physics our primitive trace map $\mc{T}^\sharp$ is known as the \emph{Wilson loop} operator \cite{Beasley-13, Giles-81, Loll-93, Wilson-74}. In stark contrast, our Theorem \ref{theorem:injectivity} says that the \emph{restriction to closed geodesics} of this operator already determines (locally) the gauge class of the connection.

\subsection{Global injectivity}

\label{ssection:global-inj}

We now mention some global injectivity results. We let $\mathbb{A}_r := \bigsqcup_{\E_r \in \mathrm{Vect}_r(M)} \mathbb{A}_{\E_r}$, where the disjoint union runs over all Hermitian vector bundles $\E_r \in \mathrm{Vect}_r(M)$ of rank $r$ over $M$ up to isomorphisms, and we set:
\[
\mathbb{A} := \bigsqcup_{r \geq 0} \mathbb{A}_r,
\]
and $\mathrm{Vect}(M) = \bigsqcup_{r \geq 0} \mathrm{Vect}_r(M)$ be the space of all topological vector bundles up to isomorphisms. A point $\mathrm{x} \in \mathbb{A}$ corresponds to a pair $([\mc{E}],\mathfrak{a})$, where $[\mc{E}] \in \mathrm{Vect}(M)$ is an equivalence class of Hermitian vector bundles and $\mathfrak{a}$ a class of gauge-equivalent unitary connections.\footnote{Note that if two smooth Hermitian vector bundles $\E_1$ and $\E_2$ are isomorphic as topological vector bundles (i.e. there exists an invertible $p \in C^\infty(M,\mathrm{Hom}(\E_1,\E_2))$), then they are also isomorphic as Hermitian vector bundles, that is $p$ can be taken unitary; the choice of Hermitian structure is therefore irrelevant.} 

The space $\mathbb{A}$ has a natural monoid structure given by the $\oplus$-operator of taking direct sums (both for the vector bundle part and the connection part). The primitive trace map can then be seen as a \emph{global} (monoid) homomorphism: 
\begin{equation}
\label{equation:trace-total}
\mc{T}^\sharp : \mathbb{A} \longrightarrow \ell^\infty(\mc{C}^\sharp),
\end{equation}
where $\ell^\infty(\mc{C}^\sharp)$ is endowed with the obvious additive structure. We actually conjecture that the generic assumption of Theorem \ref{theorem:injectivity} is unnecessary and that the primitive trace map \eqref{equation:trace-total} should be globally injective if $\dim(M) \geq 3$ and $\dim (M)$ is odd. Let us discuss a few partial results supporting the validity of this conjecture:

\begin{enumerate}
	\item In \S\ref{sssection:line}, we show that the primitive trace map is injective when restricted to \emph{direct sums of line bundles} when $\dim(M) \geq 3$, see Theorem \ref{theorem:sum}. Note that it was proved by Paternain \cite{Paternain-09} that the primitive trace map restricted to line bundles $\mc{T}_1^\sharp : \mathbb{A}_1 \longrightarrow \ell^\infty(\mc{C}^\sharp)$ is injective when $\dim(M) \geq 3$. 

	\item In \S\ref{sssection:flat}, under the restriction of $\mc{T}^\sharp$ to \emph{flat} connections, we show that the primitive trace map $\mc{T}^\sharp$ is globally injective, see Proposition \ref{proposition:flat}.
	
	\item In \S\ref{sssection:negative}, we also obtain a global result in negative curvature under an extra \emph{spectral condition}, see Proposition \ref{proposition:negative}. This condition is generic (see Appendix \ref{appendix:ckts}) and is also satisfied by connections with \emph{small curvature}, i.e. whose curvature is controlled by a constant depending only on the dimension and an upper bound on the sectional curvature of $(M,g)$ (see Lemma \ref{lemma:small-curvature}).

	\item In \S\ref{sssection:topology}, as a consequence of Corollary \ref{corollary:iso} below, we have that the primitive trace map $\mc{T}^{\sharp}([\mc{E}],\mathfrak{a})$ allows to recover the isomorphism class $\pi^*[\E]$. In particular if $\dim M$ is odd, this suffices to recover $[\E]$, see Proposition \ref{proposition:topology}.
	
\end{enumerate}

Theorem \ref{theorem:injectivity} is inspired by earlier work on the subject, see \cite{Paternain-09,Paternain-12,Paternain-13,Paternain-lecture-notes,Guillarmou-Paternain-Salo-Uhlmann-16} for instance. Nevertheless, it goes beyond the aforementioned literature thanks to an \emph{exact Liv\v{s}ic cocycle Theorem} (see Theorem \ref{theorem:weak-intro}), explained in the next paragraph \S\ref{ssection:exact}. It also belongs to a more general family of \emph{geometric inverse results} which has become a very active field of research in the past twenty years, both on closed manifolds and on manifolds with boundary, see \cite{Pestov-Uhlmann-05, Stefanov-Uhlmann-04,Paternain-Salo-Uhlmann-13, Uhlmann-Vasy-16,Stefanov-Uhlmann-Vasy-17,Guillarmou-17-2} among other references.

Theorem \ref{theorem:injectivity} can also be compared to a similar problem called the \emph{marked length spectrum} (MLS) rigidity conjecture, also known as the Burns-Katok \cite{Burns-Katok-85} conjecture. The latter asserts that if $(M,g)$ is Anosov, then the marked length spectrum
\begin{equation}
\label{equation:mls}
L_g : \mc{C} \rightarrow \R_+, ~~~L_g(c) := \ell_g(\gamma_g(c)),
\end{equation}
(where $\ell_g(\gamma)$ denotes the Riemannian length of the curve $\gamma \subset M$ computed with respect to the metric $g$), namely the length of all closed geodesics marked by the free homotopy classes of $M$, should determine the metric up to isometry. Despite some partial answers \cite{Katok-88, Croke-90,Otal-90,Besson-Courtois-Gallot-95,Hamenstadt-99,Guillarmou-Lefeuvre-18}, this conjecture is still widely open. Recently, Guillarmou and the second author proved a local version of the Burns-Katok conjecture \cite{Guillarmou-Lefeuvre-18} using techniques from microlocal analysis and the theory of Pollicott-Ruelle resonances. 

\subsection{Inverse Spectral problem}

The \emph{length spectrum} of the Riemannian manifold $(M,g)$ is the collection of lengths of closed geodesics \emph{counted with multiplicities}. It is said to be \emph{simple} if all closed geodesics have distinct lengths and this is known to be a generic condition (with respect to the metric), see \cite{Abraham-70,Anosov-82} (even in the non-Anosov case). Given $\nabla^{\E} \in \mathfrak{a}$, one can form the \emph{connection Laplacian} $\Delta_{\nabla^{\E}} := (\nabla^{\E})^*\nabla^{\E}$ (also known as the Bochner Laplacian) which is a differential operator of order $2$, non-negative, formally self-adjoint and elliptic, acting on $C^\infty(M,\E)$. While $\Delta_{\nabla^{\E}}$ depends on a choice of representative $\nabla^{\E}$ in the class $\mathfrak{a}$, its spectrum does not and there is a well-defined \emph{spectrum map}:
\begin{equation}\label{eq:spectrummap}
\mc{S} : \mathbb{A}_{\E} \ni \mathfrak{a} \mapsto \mathrm{spec}(\Delta_{\mathfrak{a}}),
\end{equation}
where $\mathrm{spec}(\Delta_{\mathfrak{a}}) = \left\{0 \leq \lambda_0(\mathfrak{a}) \leq \lambda_1(\mathfrak{a}) \leq ...\right\}$ is the spectrum counted with multiplicities. Note that more generally, the spectrum map \eqref{eq:spectrummap} can be defined on the whole moduli space $\mathbb{A}$ (just as the primitive trace map \eqref{equation:trace}). The trace formula of Duistermaat-Guillemin \cite{Duistermaat-Guillemin-75, Guillemin-73} applied to $\Delta_{\mathfrak{a}}$ reads (when the length spectrum is simple):
\begin{equation}
\label{equation:trace-formula}
\lim_{t \to \ell(\gamma_g(c))} \left(t-\ell(\gamma_g(c))\right) \sum_{j \geq 0} e^{-i \sqrt{\lambda_j}(\mathfrak{a})t} = \dfrac{\ell(\gamma_g(c^\sharp)) \Tr\left(\mathrm{Hol}_{\nabla^{\E}}(c)\right)}{2\pi |\det(\mathbbm{1}-P_{\gamma_g(c)})|^{1/2}} ,
\end{equation}
where $\sharp : \mc{C} \rightarrow \mc{C}^\sharp$ is the operator giving the primitive orbit associated to an orbit; $P_\gamma$ is the Poincar\'e map associated to the orbit $\gamma$ and $\ell(\gamma)$ its length. Theorem \ref{theorem:injectivity} therefore has the following straightforward consequence:

\begin{corollary}
\label{corollary:spectral}
Let $(M,g)$ be a smooth Anosov Riemannian manifold of dimension $\geq 3$ \emph{with simple length spectrum}. Then:
\begin{itemize}
\item Let $\E \rightarrow M$ be a smooth Hermitian vector bundle and $\mathfrak{a}_0 \in \mathbb{A}_{\E}$ be a generic point. Then, the spectrum map $\mc{S}$ is locally injective near $\mathfrak{a}_0$.
\item The spectrum map $\mc{S}$ is also globally injective when restricted to the cases \emph{(1)-(4)} of the previous paragraph \S\ref{ssection:global-inj}.
\end{itemize}

\end{corollary}

This corollary simply follows from Theorem \ref{theorem:injectivity} by observing that under the simple length spectrum assumption, the primitive trace map can be recovered from the equality \eqref{equation:trace-formula}. Corollary \ref{corollary:spectral} is analogous to the Guillemin-Kazhdan \cite{Guillemin-Kazhdan-80,Guillemin-Kazhdan-80-2} rigidity result in which a potential $q \in C^\infty(M)$ is recovered from the knowledge of the spectrum of $-\Delta_g + q$ (see also \cite{Croke-Sharafutdinov-98,Paternain-Salo-Uhlmann-14-1}). As far as the connection Laplacian is concerned, it seems that Corollary \ref{corollary:spectral} is the first positive result in this direction. Counter-examples were constructed by Kuwabara \cite{Kuwabara-90} using the Sunada method \cite{Sunada-85} but on coverings of a given Riemannian manifolds; hence the simple length spectrum condition is clearly violated. Up to our knowledge, it is also the first positive general result in an inverse spectral problem on a closed manifold of dimension $> 1$ with an \emph{infinite} gauge-group. 

This gives hope that similar methods could be used in the classical problem of recovering the isometry class of a metric from the spectrum of its Laplace-Beltrami operator \emph{locally} (similarly to a conjecture of Sarnak for planar domains \cite{Sarnak-90}). Such a result was already obtained in a neighbourhood of negatively-curved locally symmetric spaces by Sharafutdinov \cite{Sharafutdinov-09}. See also \cite{Croke-Sharafutdinov-98} for the weaker deformational spectral rigidity results or \cite{DeSimoi-Kaloshin-Wei-17, Hezari-Zelditch-19} for recent results in the plane.

\subsection{Exact Liv\v{s}ic cocycle theorem}

\label{ssection:exact}

The main ingredient in the proof of Theorem \ref{theorem:injectivity} is the following Liv\v{s}ic-type result in hyperbolic dynamical systems, which may be of independent interest. It shows that the cohomology class of a unitary cocycle over a transitive Anosov flow is determined by its trace along primitive periodic orbits. We phrase it in a somewhat more general context where we allow non-trivial vector bundles.

\begin{theorem}
\label{theorem:weak-intro}
Let $\M$ be a smooth manifold endowed with a smooth transitive Anosov flow $(\varphi_t)_{t \in \R}$. For $i\in \left\{1,2\right\}$, let $\E_i \rightarrow \M$ be a Hermitian vector bundle over $\M$ equipped with a unitary connection $\nabla^{\mc{E}_i}$, and denote by $C_i(x,t) : (\E_i)_x \rightarrow (\E_i)_{\varphi_t(x)}$ the parallel transport along the flowlines with respect to $\nabla^{\E_i}$. If the connections have \emph{trace-equivalent holonomies} in the sense that for all \emph{primitive} periodic orbits $\gamma$, one has
\begin{equation}
\label{equation:trace-intro}
\Tr\left(C_1(x_\gamma,\ell(\gamma))\right) = \Tr\left(C_2(x_\gamma,\ell(\gamma))\right),
\end{equation}
where $x_\gamma \in \gamma$ is arbitrary and $\ell(\gamma)$ is the period of $\gamma$, then the following holds: there exists $p \in C^\infty(\M,\mathrm{U}(\E_2,\E_1))$ such that for all $x \in \M, t \in \R$,
\begin{equation}
\label{equation:cohomologous}
C_1(x,t) = p(\varphi_t x) C_2(x,t) p(x)^{-1}.
\end{equation}
\end{theorem}

In the vocabulary of dynamical systems, note that every unitary cocycle is given by parallel transport along some unitary connection and \eqref{equation:cohomologous} says that the cocycles induced by parallel transport are \emph{cohomologous}. In particular, in the case of the trivial principal bundle $\mathrm{U}(r) \times \M \rightarrow \M$ our theorem can be restated just in terms of $\mathrm{U}(r)$-cocycles. Note that the bundles $\E_1$ and $\E_2$ could be \emph{a priori} distinct (and have different ranks) but Theorem \ref{theorem:weak-intro} shows that they are actually isomorphic:

\begin{corollary}
\label{corollary:iso}
Let $\E_1,\E_2 \rightarrow \M$ be two Hermitian vector bundles equipped with respective unitary connection $\nabla^{\mc{E}_1}$ and $\nabla^{\mc{E}_2}$. If the traces of the holonomy maps agree as in \eqref{equation:trace-intro}, then $\E_1 \simeq \E_2$ are isomorphic.

\end{corollary}

Theorem \ref{theorem:weak-intro} has other geometric consequences which are further detailed in \S\ref{ssection:intro}. Liv\v{s}ic-type theorems have a long history in hyperbolic dynamical systems going back to the seminal paper of Liv\v{s}ic \cite{Livsic-72} and appear in various settings. They were both developed in the Abelian case i.e. for functions (see \cite{Livsic-72,DeLaLlave-Marco-Moryon-86, Lopes-Thieullen-05,Guillarmou-17-1,Gouezel-Lefeuvre-19} for instance) and in the cocycle case.

Surprisingly, we could not locate any result such as Theorem \ref{theorem:weak-intro} in the literature. The closest works (in the discrete-time case) are that of Parry \cite{Parry-99} and Schmidt \cite{Schmidt-99} which mainly inspired the proof of Theorem \ref{theorem:weak-intro}. Nevertheless, when considering compact Lie groups, Parry's and Schmidt's results seem to be weaker as they need to assume that the conjugacy classes of the cocycles agree (and not only the traces) and that a certain additional cocycle is transitive in order to derive the same conclusion. The literature is mostly concerned with the discrete-time case, namely hyperbolic diffeomorphisms: in that case, a lot of articles are devoted to studying cocycles with values in a non-compact Lie group (and sometimes satisfying a ``slow-growth" assumption), see \cite{DeLaLlave-Windsor-10,Kalinin-11, Sadovskaya-13, Sadvoskaya-17, Avila-Kocsard-Liu-18}. One can also wonder if Theorem \ref{theorem:weak-intro} could be proved in the non-unitary setting. Other articles such as \cite{Nitica-Torok-95,Nitica-Torok-98,Nicol-Pollicott-99,Walkden-00,Pollicott-Walkden-01} seem to have been concerned with regularity issues on the map $p$, namely bootstrapping its regularity under some weak \emph{a priori} assumption (such as measurability only). Let us also point out at this stage that some regularity issues will appear while proving Theorem \ref{theorem:weak-intro} but this will be bypassed by the use of a recent regularity statement \cite[Theorem 4.1]{Bonthonneau-Lefeuvre-20} in hyperbolic dynamics.

\subsection{Strategy of proof}

We now briefly discuss the strategy of proof for Theorem \ref{theorem:injectivity}. Fix a generic unitary connection $\nabla^{\E}_0$ (namely, satisfying assumptions {\bf (A)} and {\bf (B)}) and pick two nearby connections $\nabla^{\E}_{1,2}$ such that
\begin{equation}
\label{equation:equality-traces}
\mc{T}^\sharp(\nabla^{\E}_1) = \mc{T}^\sharp(\nabla^{\E}_2).
\end{equation}
Our aim is to prove that there exists an isometry $p \in C^\infty(M,\mathrm{U}(\E))$ such that $\nabla^{\E}_2 = p^*\nabla^{\E}_1 = p^{-1}\nabla^{\E}_1(p \bullet)$. Equivalently, this is the same as having
\begin{equation}
\label{equation:wewant}
\nabla^{\mathrm{Hom}(\nabla^{\E}_2,\nabla^{\E}_1)} p = 0,
\end{equation}
where $\nabla^{\mathrm{Hom}(\nabla^{\E}_2,\nabla^{\E}_1)}$ is the \emph{mixed connection} (see \S\ref{sssection:connection-induced}), namely, the natural connection induced by $\nabla^{\E}_{1,2}$ on the endomorphism bundle $\End(\E)$ over $M$.

\begin{enumerate}[itemsep=5pt]

\item \textbf{Non-Abelian Liv\v sic theory:} The first step is to lift the connections to the unit tangent bundle $\pi : SM \to M$, namely, to consider the pullback bundle $\pi^* \E \to SM$ and the two pullback connections $\pi^* \nabla^{\E}_{1,2}$. Taking the parallel transport with respect to the connection $\pi^* \nabla^{\E}_i$ along a geodesic flowline $(\varphi_s(x,v))_{s \in [0,t]}$ (for $(x,v) \in SM$) yields a natural cocycle map $C_i((x,v),t) : \E_x \to \E_{\pi(\varphi_t(x,v))}$ as in \S\ref{ssection:exact}. It is then straightforward to verify that the equality of the traces \eqref{equation:equality-traces} translates into the fact that the cocycles $C_1$ and $C_2$ are trace-equivalent in the sense of \eqref{equation:trace-intro}. As a consequence, Theorem \ref{theorem:weak-intro} implies the existence of $p \in C^\infty(SM,\mathrm{U}(\pi^*\E))$ such that
\[
	C_1((x,v),t) = p(\varphi_t (x,v)) C_2((x,v),t) p(x,v)^{-1}, \qquad \forall (x,v) \in SM, \forall t \in \R.
\]
Differentiating the previous equality with respect to time $t$ and taking $t=0$, one finds that this is actually equivalent to
\begin{equation}
\label{equation:sm}
	(\pi^* \nabla)^{\mathrm{Hom}(\pi^*\nabla^{\E}_2,\pi^*\nabla^{\E}_2)}_X p = 0,
\end{equation}
where $X$ is the geodesic vector field. The relation \eqref{equation:sm} is similar to \eqref{equation:wewant} but it holds on the unit tangent bundle $SM$ and not on the base manifold $M$. The main difficulty is now to show that \eqref{equation:sm} implies \eqref{equation:wewant}; equivalently, this is the same as showing that the isometry $p$ does not depend on the velocity variable $v$ in $SM$ (only on the base variable $x$).

\item \textbf{Convexity of the leading resonance:} The idea is to use the well-established theory of Pollicott-Ruelle resonances -- which allows to define a natural spectrum for Anosov flows (namely, flows satisfying \eqref{equation:anosov}), and more generally for transport operators over Anosov flows -- and to translate \eqref{equation:sm} into the fact that the first order differential operator $\X := (\pi^* \nabla)^{\mathrm{Hom}(\pi^*\nabla^{\E}_2,\pi^*\nabla^{\E}_1)}_X$, acting on $C^\infty(SM,\pi^*\End(\E))$, has a Pollicott-Ruelle resonance at $z=0$. More precisely, we can write $\nabla^{\E}_i = \nabla^{\E}_0 + A_i$, where $A_i \in C^\infty(M,T^*M \otimes \End_{\mathrm{sk}}(\E))$ is small, with values in skew-Hermitian endomorphisms. Then, under the generic assumptions {\bf (A)} and {\bf (B)}, the operator $\X$ admits a simple leading resonance $\lambda_{A_1, A_2}$ which is real and nonpositive. Moreover, the generic assumption also ensures that for $A_1 = A_2 = 0$, the leading resonant space is spanned by the identity section $\mathbbm{1}_{\E} \in C^\infty(SM, \pi^*\End(\E))$. The key idea is to show by a convexity argument that $|\lambda_{A_1,A_2}|$ controls \emph{quantitatively} the distance between the connections $\nabla^{\E}_{1}$ and $\nabla^{\E}_{2}$ in the moduli space $\mathbb{A}_{\E}$. In particular, \eqref{equation:sm} means that the two connections satisfy $\lambda_{A_1,A_2} = 0$, and thus the connections must be gauge-equivalent.

\end{enumerate}

As pointed out by the referee, the idea to use the strict convexity of the dominant Pollicott-Ruelle resonance is reminiscent of the work of Katsuda-Sunada \cite{Katsuda-Sunada-88} and Pollicott \cite{Pollicott-91}. The main difference here is that our moduli space of unitary connections $\mathbb{A}_{\E}$ is infinite-dimensional and the quantitative strict convexity of $\lambda_{A_1,A_2}$ is given in the end by some elliptic theory.

\subsection{Organization of the paper}

The paper is divided in three parts:

\begin{itemize}[itemsep=5pt]

\item First of all, we prove in Section \S\ref{section:livsic} the \emph{exact Liv\v{s}ic cocycle} Theorem \ref{theorem:weak-intro} for general Anosov flows showing that the cohomology class of a unitary cocycle is determined by its trace along closed orbits. The proof is based on the introduction of a new tool which we call \emph{Parry's free monoid}, denoted by $\mathbf{G}$, and is formally generated by orbits homoclinic to a fixed closed orbit. We show that any unitary connection induces a unitary representation of the monoid $\mathbf{G}$ and that trace-equivalent connections have the same character; we can then apply tools from representation theory to conclude. 

\item In subsequent sections, we develop a microlocal framework, based on the theory of Pollicott-Ruelle resonances. We define a notion of \emph{generalized X-ray transform with values in a vector bundle} which is mainly inspired by \cite{Guillarmou-17-1,Guillarmou-Lefeuvre-18,Gouezel-Lefeuvre-19}. In \S\ref{section:geometry}, we relate the geometry of the moduli space of gauge-equivalent connections to the leading Pollicott-Ruelle resonance of a certain natural operator, the mixed connection.

\item Eventually, the main results such as Theorem \ref{theorem:injectivity} are proved in \S\ref{section:proofs}, where we also deduce the global properties of $\mc{T}^\sharp$ involving line bundles, flat bundles, negatively curved base manifolds, and the topology of bundles.

\end{itemize}
Some technical preliminaries are provided in Section \S\ref{section:tools}. 

\subsection{Perspectives}

We intend to pursue this work in different directions:

\begin{itemize}[itemsep=5pt]

\item Since the first release of the present article on arXiv (May 2021), the notion of \emph{Parry's free monoid} introduced in \S\ref{section:weak-strong} has proved to be extremely powerful. In particular, it was used in the companion papers \cite{Cekic-Lefeuvre-Moroianu-Semmelmann-21, Lefeuvre-21} in order to show that the frame flow of nearly $0.25$-pinched negatively-curved Riemannian manifolds is ergodic, thus almost answering a long-standing conjecture of Brin \cite[Conjecture 2.6]{Brin-82}.

\item Furthermore, in \cite[Theorem 4.5]{Cekic-Lefeuvre-23-polynomial} we were strikingly able to show \emph{global} injectivity of the primitive trace map under a suitable low-rank assumption, by exhibiting a relation with real algebraic geometry.

\item Moreover, in \cite[Theorem 1.1]{Cekic-Lefeuvre-23-stability} we showed a stability estimate version of Theorem \ref{theorem:injectivity}, namely
that the distance between connections up to gauge equivalence is controlled (at least
locally) by the distance between their images under the primitive trace map $\mathcal{T}^\sharp$.

\item Eventually, the arguments developed in \S\ref{section:livsic} mainly rely on the use of homoclinic orbits; to these orbits, we will associate a notion of \emph{length} which is well-defined as an element of $\R/T_\star\Z$, for some real number $T_\star > 0$. We believe that, similarly to the set of periodic orbits where one defines the \emph{Ruelle zeta function}
\[
\zeta(s) := \prod_{\gamma^\sharp \in \mc{G}^\sharp} \left(1 - e^{-s \ell(\gamma^\sharp)}\right), 
\]
where the product runs over all primitive periodic orbits $\gamma^\sharp \in \mc{G}^\sharp$ (and $\ell(\gamma^\sharp)$ denotes the orbit period) and shows that this extends meromorphically from $\left\{\Re(s) \gg 0 \right\}$ to $\C$ (see \cite{Giulietti-Liverani-Pollicott-13,Dyatlov-Zworski-16}), one could also define a complex function for homoclinic orbits by means of a Poincar\'e series (rather than a product). It should be a consequence of \cite[Theorem 4.15]{Dang-Riviere-20} that this function extends meromorphically to $\C$. It might then be interesting to compute its value at $0$; the latter might be independent of the choice of representatives for the length of homoclinic orbits (two representatives differ by $mT_\star$, for some $m \in \Z$) and could be (at least in some particular cases) an interesting topological invariant as for the Ruelle zeta function on surfaces, see \cite{Dyatlov-Zworski-17}.  \\

\end{itemize} 

\noindent \textbf{Acknowledgement:} M.C. has received funding from the European Research Council (ERC) under the European Union’s Horizon 2020 research and innovation programme (grant agreement No. 725967), and from an Ambizione grant (project number 201806) from the Swiss National Science Foundation. We warmly thank Yannick Guedes Bonthonneau, Yann Chaubet, Colin Guillarmou, Julien March\'e, Gabriel Paternain, Steve Zelditch for fruitful discussions. We also thank S\'ebastien Gou\"ezel, Boris Kalinin, Mark Pollicott, Klaus Schmidt for answering our questions on Liv\v{s}ic theory, and Nikhil Savale for providing us with the reference \cite{Guillemin-73}. Special thanks to Danica Kosanovi\'c for helping out with the topology part. We thank the referee for many suggestions and comments, hopefully improving the presentation.

\section{Setting up the tools}

\label{section:tools}

\subsection{Microlocal calculus and functional analysis} 

Let $\M$ be a smooth closed manifold. Given a smooth vector bundle $\E \rightarrow \M$, we denote by $\Psi^m(\M, \E)$ the space of pseudodifferential operators of order $m$ acting on $\E$. When $\E$ is the trivial line bundle, such an operator $P \in \Psi(\M)$ can be written (up to a smoothing remainder) in local coordinates as
\begin{equation}
\label{equation:quantization}
Pf(x) = \int_{\R^{n}} \int_{\R^{n}} e^{i\xi \cdot(x-y)}p(x,\xi)f(y)dyd\xi,
\end{equation}
where $f$ is compactly supported in the local patch and $p \in S^m(T^*\R^{n})$ is a \emph{symbol}, i.e. it satisfies the following estimates in local coordinates:
\begin{equation}
\label{equation:bounds}
\sup_{|\alpha'| \leq \alpha, |\beta'|\leq \beta} \sup_{(x,\xi) \in T^*\R^{n}}  \langle \xi \rangle^{m - |\alpha'|} |\partial_\xi^{\alpha'} \partial_x^{\beta'} p(x,\xi)| < \infty,
\end{equation}
for all $\alpha, \beta \in \N^n$, with $\langle \xi \rangle = \sqrt{1+|\xi|^2}$. When $\E$ is not the trivial line bundle, the symbol $p$ is an $\End(\E)$-valued symbol, which in local coordinates and local trivializations is identified with a matrix function. Given $p \in S^m(T^*\M)$, one can define a (non-canonical) \emph{quantization procedure} $\Op : S^m(T^*\M) \rightarrow \Psi^m(\M)$ thanks to \eqref{equation:quantization} in coordinates patches. This also works more generally with a vector bundle $\E \rightarrow \M$ and one has a quantization map $\Op : S^m(T^*\M, \End(\E)) \rightarrow \Psi^m(\M,\E)$ (note that the symbol is then a section of the pullback bundle $\End(\E) \rightarrow T^*\M$ satisfying the bounds \eqref{equation:bounds} in local coordinates and local trivializations). There is a well-defined (partial) inverse map 
\[\sigma_{\mathrm{princ}} : \Psi(\M,\E) \rightarrow S^m(T^*\M, \End(\E))/S^{m-1}(T^*\M, \End(\E))\] called the \emph{principal symbol} and satisfying $\sigma_{\mathrm{princ}}(\Op(p)) = [p]$ (the equivalence class as an element of $S^m(T^*\M, \End(\E))/S^{m-1}(T^*\M, \End(\E))$).

We denote by $H^s(\M, \E)$ the space of Sobolev sections of order $s \in \R$ and by $C^s_*(\M, \E)$ the space of H\"older-Zygmund sections of order $s \in \R$. It is well-known that for $s \in \R_+ \setminus \N$, $C^s_*$ coincide with the space of H\"older-continuous sections $C^s$ of order $s$. Recall that $C^s_*$ is an algebra as long as $s > 0$ and $H^s$ is an algebra for $s > n/2$. 
If $P \in \Psi^m(\M, \E)$ is a pseudodifferential operator of order $m \in \R$, then $P : X^{s+m}(\M, \E) \rightarrow X^s(\M, \E)$, where $X = H$, $C_*$, is bounded. We refer to \cite{Shubin-01,Taylor-91} for further details.

\subsection{Connections on vector bundles}

\label{ssection:connections}

We refer the reader to \cite[Chapter 2]{Donaldson-Kronheimer-90} for the background on connections on vector bundles.

\subsubsection{Mixed connection on the homomorphism bundle}

\label{sssection:connection-induced}

In this paragraph, we consider two Hermitian vector bundles $\E_1, \E_2 \rightarrow \M$ equipped with respective unitary connections $\nabla_1 = \nabla^{\E_1}$ and $\nabla_2 = \nabla^{\E_2}$ which can be written in some local patch $U \subset \R^n$ of coordinates and in local trivialisations of the bundles as $\nabla^{\E_i} = d +\Gamma_i$, for some $\Gamma_i \in C^\infty(U,T^*U \otimes \End_{\mathrm{sk}}(\E_i))$. Let $\Hom(\E_1,\E_2)$ be the vector bundle of homomorphisms from $\E_1$ to $\E_2$, endowed with the natural Hermitian structure.

\begin{definition}\label{definition:mixed}
We define the (unitary) \emph{homomorphism} or \emph{mixed} connection $\nabla^{\Hom(\nabla^{\E_1}, \nabla^{\E_2})}$ on $\Hom(\E_1, \E_2)$, induced by $\nabla^{\E_1}$ and $\nabla^{\E_2}$, by the Leibnitz property:
\[
\forall u \in C^\infty(\M, \Hom(\E_1, \E_2)), \forall s \in C^\infty(\M, \E_1), \quad \nabla^{\E_2}(us) = (\nabla^{\Hom(\nabla^{\E_1}, \nabla^{\E_2})}u) \cdot s + u\cdot  (\nabla^{\E_1}s).
\]
\end{definition}
Equivalently, it is straightforward to check that this is the canonical tensor product connection induced on $\Hom(\E_1, \E_2) \cong \E_2 \otimes \E_1^*$ and that in local coordinates and local trivializations we have
\begin{equation}\label{eq:localhom}
	\nabla^{\Hom(\nabla^{\E_1}, \nabla^{\E_2})} u := du + \Gamma_2 (\bullet) u - u \Gamma_1(\bullet).
\end{equation}
Note that this definition does not require the bundles to have same rank; we insist on the fact that the mixed connection $\nabla^{\Hom(\nabla_1, \nabla_2)}$ \emph{depends on a choice} of connections $\nabla_1$ and $\nabla_2$. In the particular case when $\E_1 = \E_2 = \E$ and $\nabla_1^{\E} = \nabla_2^{\E} = \nabla^{\E}$ we will write $\nabla^{\End(\nabla^{\E})} = \nabla^{\Hom(\nabla^{\E}, \nabla^{\E})}$ for the induced \emph{endomorphism} connection on $\End(\E)$. When clear from the context, we will also write $\nabla^{\End(\E)}$ for the endomorphism connection induced by $\nabla^{\E}$. We note that the homomorphism and endomorphism connections will play a central role in the upcoming sections.

Given a flow $(\varphi_t)_{t \in \R}$, we will denote by $P(x,t) : \Hom({\E_1}_x, {\E_2}_x) \rightarrow \Hom({\E_1}_{\varphi_t(x)}, {\E_2}_{\varphi_t(x)})$ the parallel transport with respect to the mixed connection along the flowlines of $(\varphi_t)_{t \in \R}$. This parallel transport has a clear geometric meaning, in fact we observe that for $u \in \Hom({\E_1}_x, {\E_2}_x)$, we have:
\begin{equation}
\label{equation:tp-mixed}
P(x,t) u = C_2(x,t) u C_1(x,t)^{-1},
\end{equation}
where $C_i(x,t) : {\E_i}_x \rightarrow {\E_i}_{\varphi_t(x)}$ is the parallel transport with respect to $\nabla^{\E_i}$ along the flowlines of $(\varphi_t)_{t \in \R}$.

Recall that the curvature tensor $F_\nabla \in C^\infty(\M, \Lambda^2(T^*M) \otimes \End(\E))$ of $\nabla = \nabla^{\E}$ is defined as, for any vector fields $X, Y$ on $\M$ and sections $S$ of $\E$
\begin{equation*}
	F_\nabla(X, Y)S = \nabla_X \nabla_Y S - \nabla_Y \nabla_X S - \nabla_{[X, Y]}S.
\end{equation*} 
Then a quick computation using \eqref{eq:localhom} reveals that:
\begin{equation}
\label{equation:induced-curvature}
F_{\nabla^{\mathrm{Hom}(\nabla_1,\nabla_2)}} u = F_{\nabla_2} \cdot u - u \cdot F_{\nabla_1}.
\end{equation}
Eventually, using the Leibnitz property, we observe that if $p_i \in C^\infty(\M,\mathrm{U}(\E_i',\E_i))$ is an isomorphism, then:
\begin{align}
\label{equation:lien}
\begin{split}
\nabla^{\Hom(\nabla_1, p_2^*\nabla_2)} u &= (p_2)^{-1} \nabla^{\Hom(\nabla_1, \nabla_2)} (p_2 u), \quad \forall u \in C^\infty(\M, \Hom(\E_1, \E_2')),\\
\nabla^{\Hom(p_1^*\nabla_1, \nabla_2)} u &= \nabla^{\Hom(\nabla_1, \nabla_2)} (u p_1^{-1}) \cdot p_1,\quad\,\, \forall u \in C^\infty(\M, \Hom(\E_1', \E_2)).
\end{split}
\end{align}

\subsubsection{Ambrose-Singer formula}

Recall that the celebrated Ambrose-Singer formula (see eg. \cite[Theorem 8.1]{Kobayashi-Nomizu-69}) determines the tangent space at the identity of the holonomy group with respect to an arbitrary connection, in terms of its curvature tensor. Here we give an integral version of this fact. We start with a Hermitian vector bundle $\E$ over the Riemannian manifold $(\M, g)$. Equip $\E$ with unitary connection $\nabla = \nabla^{\E}$.

Consider a smooth homotopy $\Gamma: [0, 1]^2 \to \M$ such that $\Gamma(0, 0) = p$.
The ``vertical'' map $C_{\uparrow}(s, t): \E_p \to \E_{\Gamma(s, t)}$ is obtained by parallel transporting with respect to $\nabla$ from $\E_p$ to $\E_{\Gamma(0, 1)}$, then $\E_{\Gamma(0, 1)}$ to $\E_{\Gamma(s, 1)}$ and $\E_{\Gamma(s, 1)}$ to $\E_{\Gamma(s, t)}$, along $\Gamma(0, \bullet)$, $\Gamma(\bullet, 1)$ and $\Gamma(s, \bullet)$, respectively. Next, define the ``horizontal'' map $C_{\rightarrow}(s, t): \E_p \to \E_{\Gamma(s, t)}$ by parallel transport with respect to $\nabla$ from $\E_p$ to $\E_{\Gamma(s, 0)}$ and $\E_{\Gamma(s, 0)}$ to $\E_{\Gamma(s, t)}$, along $\Gamma(\bullet, 0)$ and $\Gamma(s, \bullet)$, respectively. For a better understanding, see Figure \ref{fig:AS1}.

We are ready to prove the formula:

\begin{lemma}\label{lemma:ambrosesinger}
	The following formula holds
	\begin{equation}
		C_{\uparrow}^{-1}(1, 1) C_{\rightarrow}(1, 1) - \mathbbm{1}_{\E_{p}} = \int_0^1 \int_0^1 C_{\uparrow}(s, t)^{-1} F_\nabla(\partial_t, \partial_s) C_{\rightarrow}(s, t) \, dt \, ds.
	\end{equation}	
\end{lemma}

\begin{proof}
	Let $w_1, w_2 \in \E_p$; formally, we will identify the connection $\nabla$ with its pullback $\Gamma^*\nabla$ on the pullback bundle $\Gamma^*\E$ over $[0, 1]^2$, as well as the curvature $F_\nabla$ with $\Gamma^*F_\nabla$. Then we have the following chain of equalities:
	\begin{align*}
		&\langle{w_1, (C_{\uparrow}(1, 1)^{-1} C_{\rightarrow}(1, 1) - \mathbbm{1}_{\E_p}) w_2}\rangle = \langle{C_{\uparrow}(1, 1)w_1, C_{\rightarrow}(1, 1)w_2}\rangle - \langle{C_{\uparrow}(0, 1)w_1, C_{\rightarrow}(0, 1)w_2}\rangle\\
		 &= \int_0^1 \partial_s \langle{C_{\uparrow}(s, 1) w_1, C_{\rightarrow}(s, 1) w_2}\rangle ds\\
		&= \int_0^1 \langle{C_{\uparrow}(s, 1) w_1, \nabla_{\partial_s} C_{\rightarrow}(s, 1) w_2}\rangle ds\\
		&= \int_0^1 \Big[\int_0^1 \Big(\partial_t \langle{C_{\uparrow}(s, t) w_1, \nabla_{\partial_s} C_{\rightarrow}(s, t)w_2}\rangle + \langle{C_{\uparrow}(s, 0) w_1, \nabla_{\partial_s} C_{\rightarrow}(s, 0) w_2}\rangle\Big)dt\Big] ds\\
		&= \int_0^1 \int_0^1 \langle{C_{\uparrow}(s, t) w_1, \nabla_{\partial_t} \nabla_{\partial_s} C_{\rightarrow}(s, t) w_2}\rangle \,ds\, dt\\
		&= \int_0^1 \int_0^1 \langle{w_1, C_{\uparrow}(s, t)^{-1}\underbrace{(\nabla_{\partial_t} \nabla_{\partial_s} - \nabla_{\partial_s} \nabla_{\partial_t})}_{=F_\nabla(\partial_t, \partial_s)} C_{\rightarrow}(s, t) w_2}\rangle \,ds\, dt,
	\end{align*}
	as the Lie bracket $[\partial_s, \partial_t] = 0$ and we used the unitarity of $\nabla$ throughout. This completes the proof, since $w_1$ and $w_2$ were arbitrary.
\end{proof}

\begin{figure}
             \centering
\begin{tikzpicture}[scale = 0.8, everynode/.style={scale=0.5}]
\tikzset{cross/.style={cross out, draw=black, minimum size=2*(#1-\pgflinewidth), inner sep=0pt, outer sep=0pt},
cross/.default={1pt}}

      	\draw[thick, ->] (0, 0) -- (6,0) node[right] {\small $s$};
		\draw[thick, ->] (0, 0) -- (0,6) node[left] {\small $t$};
		
		\draw[thick] (3, 5) -- (5, 5) -- (5, 0);
		
		\draw[thick, blue] (0, 0) -- (0, 5) -- (3, 5) -- (3, 3);
		
		\draw[thick, red] (0, 0) -- (3, 0) -- (3, 3);
		
		\draw[thick, ->] (3, 1.5)--(3, 1.501);
		\draw[thick, ->] (1.5, 0)--(1.501, 0);
		
		\draw[thick, ->] (0, 2.5)--(0, 2.501);
		\draw[thick, ->] (1.5, 5)--(1.501, 5);
		\draw[thick, ->] (3, 4)--(3, 3.999);
		
		\fill (0, 0) node[below left] {\tiny $p$} circle (1.5pt);
		\fill (5, 0) node[below] {\tiny $\Gamma(1, 0)$}  circle (1.5pt);
		\fill (0, 5) node[left] {\tiny $\Gamma(0, 1)$}  circle (1.5pt);
		\fill (5, 5) node[right] {\tiny $\Gamma(1, 1)$}  circle (1.5pt);
		
		\fill (3, 3) node[right] {\small $\Gamma(s, t)$}  circle (1.5pt);
		\fill (3, 0) node[below] {\small $\Gamma(s, 0)$}  circle (1.5pt);
		\fill (3, 5) node[above] {\small $\Gamma(s, 1)$}  circle (1.5pt);
\end{tikzpicture}
             \caption{\small The homotopy $\Gamma$ in Lemma \ref{lemma:ambrosesinger} with the corresponding points in $\M$: in blue and red are the trajectories along which the parallel transport maps $C_{\uparrow}$ (vertical) and $C_{\rightarrow}$ (horizontal) are taken, respectively.}
             \label{fig:AS1}
\end{figure}
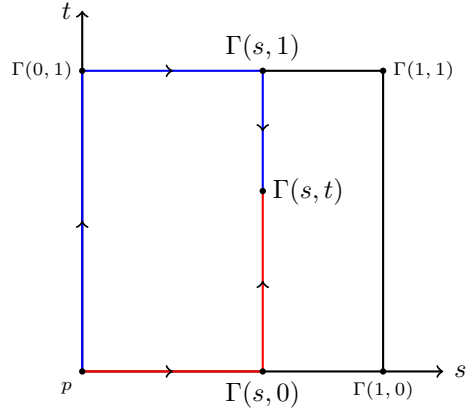

We have two applications in mind for this lemma: one if $\gamma$ is in a neighbourhood of $p$ and we use the radial homotopy via geodesics emanating from $p$, and the second one for the ``thin rectangle'' obtained by shadowing a piece of the flow orbit, see Lemma \ref{lemma:ASgeometry}.

\subsection{Fourier analysis in the fibers} In this paragraph, we recall some elements of Fourier analysis in the fibers and refer to \cite{Guillemin-Kazhdan-80, Guillemin-Kazhdan-80-2, Paternain-Salo-Uhlmann-14-2, Paternain-Salo-Uhlmann-15, Guillarmou-Paternain-Salo-Uhlmann-16} for further details.

\subsubsection{Analysis on the trivial line bundle}

\label{sssection:line-bundle}

Let $(M,g)$ be a smooth Riemannian manifold of arbitrary dimension $n \geq 2$. The unit tangent bundle is endowed with the natural Sasaki metric and we let $\pi : SM \rightarrow M$ be the projection on the base. There is a canonical splitting of the tangent bundle to $SM$ as:
\[
T(SM) = \HH \oplus \V \oplus \R X,
\] 
where $X$ is the geodesic vector field, $\V := \ker \dd \pi$ is the vertical space and $\HH$ is the horizontal space: it can be defined as the orthogonal to $\V \oplus \R X$ with respect to the Sasaki metric (see \cite[Chapter 1]{Paternain-99}). Any vector $Z \in T(SM)$ can be decomposed according to the splitting
\[
Z = \alpha(Z) X + Z_{\HH} + Z_{\V},
\]
where $\alpha$ is the Liouville $1$-form, $Z_{\HH} \in \HH, Z_{\V} \in \V$. If $f \in C^\infty(SM)$, its gradient computed with respect to the Sasaki metric can be written as:
\[
\nabla_{\mathrm{Sas}}f = (Xf) X + \nabla_{\HH} f + \nabla_{\V} f,
\]
where $\nabla_{\HH} f \in \HH$ is the horizontal gradient, $\nabla_{\V} f \in \V$ is the vertical gradient. We also let $\mc{N} \to SM$ be the \emph{normal bundle} whose fiber over each $(x,v) \in SM$ is given by $(\R \cdot v)^\bot$. The bundles $\mathbb{H}$ and $\mathbb{V}$ may be naturally identified with the bundle $\mc{N}$ (see \cite[Section 1]{Paternain-99}).

For every $x \in M$, the sphere $S_xM = \left\{ v \in T_xM ~|~ |v|^2_x = 1\right\} \subset SM$ endowed with the Sasaki metric is isometric to the canonical sphere $(\Ss^{n-1},g_{\mathrm{can}})$. We denote by $\Delta_{\V}$ the vertical Laplacian obtained for $f \in C^\infty(SM)$ as $\Delta_{\V} f(x,v) = \Delta_{g_{\mathrm{can}}}(f|_{S_xM})(v)$, where $\Delta_{g_{\mathrm{can}}}$ is the spherical Laplacian. For $m \geq 0$, we denote by $\Omega_m$ the (finite-dimensional) vector space of spherical harmonics of degree $m$ for the spherical Laplacian $\Delta_{g_{\mathrm{can}}}$: they are defined as the elements of $\ker(\Delta_{g_{\mathrm{can}}} + m(m+n-2))$. We will use the convention that $\Omega_m = \left\{0\right\}$ if $m< 0$. If $f \in C^\infty(SM)$, it can then be decomposed as $f = \sum_{m \geq 0} f_m$, where $f_m \in C^\infty(M,\Omega_m)$ is the $L^2$-orthogonal projection of $f$ onto the spherical harmonics of degree $m$.

There is a one-to-one correspondence between trace-free symmetric tensors of degree $m$ and spherical harmonics of degree $m$. More precisely, the map
\[
\pi_m^* : C^\infty(M,\otimes^m_S T^*M|_{0-\Tr}) \rightarrow C^\infty(M,\Omega_m),
\]
given by $\pi_m^*f(x,v) = f_x(v,...,v)$ is an isomorphism. Here, the index $0-\Tr$ denotes the space of trace-free symmetric tensors, namely tensors such that, if $(\e_1,...,\e_n)$ denotes a local orthonormal frame of $TM$:
\[
\Tr(f) := \sum_{i=1}^n f(\e_i,\e_i,\cdot, ...,\cdot) = 0.
\]
We will denote by ${\pi_m}_* : C^\infty(M,\Omega_m) \rightarrow C^\infty(M,\otimes^m_S T^*M|_{0-\Tr})$ the adjoint of this map. More generally, the mapping
\begin{equation}\label{eq:pi_muntwisted}
\pi_m^* : C^\infty(M,\otimes^m_S T^*M) \rightarrow \oplus_{k \geq 0}  C^\infty(M,\Omega_{m-2k})
\end{equation}
is an isomorphism. We refer to \cite[Section 2]{Cekic-Lefeuvre-20} for further details.

The geodesic vector field acts as $X : C^\infty(M,\Omega_m) \rightarrow C^\infty(M,\Omega_{m-1}) \oplus C^\infty(M,\Omega_{m+1})$ (see \cite{Guillemin-Kazhdan-80-2, Paternain-Salo-Uhlmann-15}). We define $X_+$ as the $L^2$-orthogonal projection of $X$ on the higher modes $\Omega_{m+1}$, namely if $u \in C^\infty(M,\Omega_m)$, then $X_+u := (Xu)_{m+1}$ and $X_-$ as the $L^2$-orthogonal projection of $X$ on the lower modes $\Omega_{m-1}$. For $m \geq 0$, the operator $X_+ : C^\infty(M,\Omega_m) \rightarrow C^\infty(M,\Omega_{m+1})$ is elliptic and thus has a finite dimensional kernel (see \cite{Dairbekov-Sharafutdinov-10}). The operator $X_- : C^\infty(M,\Omega_m) \rightarrow C^\infty(M,\Omega_{m-1})$ is of divergence type. The elements in the kernel of $X_+$ are called \emph{Conformal Killing Tensors (CKTs)}, associated to the trivial line bundle. For $m=0$, the kernel of $X_+$ on $C^\infty(M,\Omega_0)$ always contains the constant functions. We call \emph{non trivial CKTs} elements in $\ker X_+$ which are not constant functions on $SM$. The kernel of $X_+$ is invariant by changing the metric by a conformal factor (see \cite[Section 3.6]{Guillarmou-Paternain-Salo-Uhlmann-16}). It is known (see \cite{Paternain-Salo-Uhlmann-15}) that there are no non trivial CKTs in negative curvature and for Anosov surfaces but the question remains open for general Anosov manifolds. We provide a positive answer to this question \emph{generically} as a byproduct of our work \cite{Cekic-Lefeuvre-21-2}.

\subsubsection{Twisted Fourier analysis}

\label{sssection:twisted-fourier}

We now consider a Hermitian vector bundle with a unitary connection $(\mc{E},\nabla^{\mc{E}})$ over $(M,g)$ and define the operator $\X := (\pi^*\nabla^{\mc{E}})_X$ acting on $C^\infty(SM,\pi^*\mc{E})$, where $\pi : SM \rightarrow M$ is the projection. For the sake of simplicity, we will drop the $\pi^*$ in the following. If $f \in C^\infty(SM,\mc{E})$, then $\nabla^{\mc{E}}f \in C^\infty(SM,T^*(SM) \otimes \mc{E})$ and we can write
\[
\nabla^{\mc{E}}f = (\X f, \nabla_{\HH} f, \nabla_{\V} f),
\]
where $\nabla_{\HH} f \in C^\infty(SM, \HH^* \otimes \mc{E}), \nabla_{\V} f \in C^\infty(SM, \V^* \otimes \mc{E})$. For future reference, we introduce a bundle endomorphism map $R$ on $\mc{N} \otimes \E$, derived from the Riemann curvature tensor via the formula $R(x, v)(w \otimes e) = (R_x(w, v)v) \otimes e$.

If $(e_1,...,e_r)$ is a local orthonormal frame of $\mc{E}$, then we define the vertical Laplacian as
\[
\Delta_{\V}^{\mc{E}}(\sum_{k=1}^r u_k e_k) := \sum_{k=1}^r (\Delta_{\V} u_k) e_k.
\]
Any section $f \in C^\infty(SM,\mc{E})$ can be decomposed according to $f = \sum_{m \geq 0} f_m$, where $f_m \in \ker(\Delta_{\V}^{\mc{E}} + m(m+n-2))$ and we define $C^\infty(M,\Omega_m \otimes \mc{E}) := \ker(\Delta_{\V}^{\mc{E}} + m(m+n-2)) \cap C^\infty(SM,\mc{E})$.

Here again, the operator $\X : C^\infty(M,\Omega_m \otimes \E) \rightarrow C^\infty(M,\Omega_{m-1} \otimes \E) \oplus C^\infty(M,\Omega_{m+1} \otimes \E)$ can be split into the corresponding sum $\X = \X_- + \X_-$. For every $m \geq 0$, the operator $\X_+$ is elliptic and has finite dimensional kernel, whereas $\X_-$ is of divergence type. The kernel of $\X_+$ is invariant by a conformal change of the metric (see \cite[Section 3.6]{Guillarmou-Paternain-Salo-Uhlmann-16}) and elements in its kernel are called \emph{twisted Conformal Killing Tensors} (twisted CKTs). For simplicity we will often drop the word twisted and refer to the latter as CKTs. There are examples of vector bundles with CKTs on manifolds of arbitrary dimension. We proved in a companion paper \cite{Cekic-Lefeuvre-20} that the non existence of CKTs is a generic condition, no matter the curvature of the manifold (generic with respect to the connection, i.e. there is a residual set of the space of all unitary connections with regularity $C^k$, $k \geq 2$, which has no CKTs).

It is also known by \cite{Guillarmou-Paternain-Salo-Uhlmann-16}, that in negative curvature, there is always a \emph{finite number of degrees} with CKTs (and this number can be estimated thanks to a lower bound on the curvature of the manifold and the curvature of the vector bundle). In other words, $\ker \X_+|_{C^\infty(SM,\mc{E})}$ is finite-dimensional. The proof relies on an energy identity called the Pestov identity. This is also known for Anosov surfaces since any Anosov surface is conformally equivalent to a negatively-curved surface and CKTs are conformally invariant. Nevertheless, and to the best of our knowledge, it is still an open question to show that for Anosov manifolds of dimension $n \geq 3$, there is at most a finite number of CKTs.

\subsection{Twisted symmetric tensors}\label{section:twisted}

Given a section $u \in C^\infty(M,\otimes^m_S T^*M \otimes \mc{E})$, the connection $\nabla^{\mc{E}}$ produces an element $\nabla^{\mc{E}}u \in C^\infty(M, T^*M \otimes (\otimes^m_S T^*M) \otimes \mc{E})$. In coordinates, if $(e_1, ..., e_r)$ is a local orthonormal frame for $\mc{E}$ and $\nabla^{\mc{E}} = d + \Gamma$, for some one-form with values in skew-Hermitian matrices $\Gamma$, such that  $\nabla^{\E}e_k = \sum_{i = 1}^n\sum_{l = 1}^r \Gamma_{ik}^{l} dx_i \otimes e_l$, we have:
\begin{equation}
\label{equation:nabla-e}
\begin{split}
\nabla^{\mc{E}}(\sum_{k=1}^r u_k \otimes e_k) & = \sum_{k=1}^r \big(\nabla u_k \otimes e_k + u_k \otimes \nabla^{\mc{E}} e_k\big)\\
& = \sum_{k=1}^r \left(\nabla u_k +   \sum_{l=1}^r \sum_{i=1}^n \Gamma_{il}^k u_l \otimes dx_i \right) \otimes e_k,
\end{split}
\end{equation}
where $u_k \in C^\infty(M,\otimes^m_S T^*M)$ and $\nabla$ is the Levi-Civita connection. The symmetrization operator $\mc{S}^{\mc{E}} : C^\infty(M,\otimes^m T^*M \otimes \mc{E}) \rightarrow C^\infty(M,\otimes^m_S T^*M \otimes \mc{E})$ is defined by:
\[
\mc{S}^{\mc{E}}\left(\sum_{k=1}^r u_k \otimes e_k\right) = \sum_{k=1}^r \mc{S}(u_k) \otimes e_k,
\]
where $u_k \in C^\infty(M,\otimes^m_S T^*M)$ and in coordinates, writing $u_k = \sum_{i_1, ..., i_m=1}^n u_{i_1...i_m}^{(k)} dx_{i_1} \otimes ... \otimes dx_{i_m}$, we have
\[
\mc{S}(dx_{i_1} \otimes ... \otimes dx_{i_m}) = \dfrac{1}{m!} \sum_{\pi \in \mathfrak{S}_m} dx_{\pi(i_1)} \otimes ... \otimes dx_{\pi(i_m)},
\]
where $\mathfrak{S}_m$ denotes the group of permutations of order $m$. For the sake of simplicity, we will write $\mc{S}$ instead of $\mc{S}^{\mc{E}}$. We can symmetrize \eqref{equation:nabla-e} to produce an element $D^{\E} := \mc{S} \nabla^{\mc{E}}u \in C^\infty(M, \otimes^{m+1}_S T^*M \otimes \mc{E})$ given in coordinates by:
\begin{equation}
\label{equation:formula-de}
D^{\E} \left(\sum_{k=1}^r u_k \otimes e_k\right) = \sum_{k=1}^r \left( Du_k + \sum_{l=1}^r  \sum_{i=1}^n \Gamma_{il}^k \sigma(u_l \otimes dx_i) \right) \otimes e_k,
\end{equation}
where $D := \mc{S} \nabla$ is the usual symmetric derivative of symmetric tensors\footnote{Beware of the notation: $\nabla^{\mc{E}}$ is for the connection, $D^{\E}$ for the symmetric derivative of tensors and $\nabla^{\mathrm{End}(\E)}$ is the connection induced by $\nabla^{\mc{E}}$ on the endomorphism bundle.}. Elements of the form $Du \in C^\infty(M,\otimes^{m+1}_S T^*M)$ are called \emph{potential tensors}. By comparison, we will call elements of the form $D^{\E}f \in C^\infty(M,\otimes^{m+1}_S T^*M \otimes \mc{E})$ \emph{twisted potential tensors}. The operator $D^{\E}$ is a first order differential operator and its expression can be read off from \eqref{equation:formula-de}, namely:
\[
\begin{split}
\sigma_{\mathrm{princ}}(D^{\E})(x,\xi) \cdot \left(\sum_{k=1}^r u_k(x) \otimes e_k(x) \right) & = \sum_{k=1}^r \left(\sigma_{\mathrm{princ}}(D)(x,\xi) \cdot u_k(x)\right) \otimes e_k(x) \\
& = i \sum_{k=1}^r \sigma(\xi \otimes u_k(x)) \otimes e_k(x),
\end{split}
\]
where $e_k(x) \in \mc{E}_x, u_k(x) \in \otimes^m_S T^*_xM$ and the basis $(e_1(x),...,e_r(x))$ is assumed to be orthonormal for the metric $h$ on $\mc{E}$. One can check that this is an injective map, which means that $D^{\E}$ is a left-elliptic operator and can be inverted on the left modulo a smoothing remainder. Its kernel is finite-dimensional and consists of smooth elements.

Before that, we introduce for $m \in \N$, the operator
\[
\pi_m^* : C^\infty(M,\otimes^m_S T^*M \otimes \mc{E}) \rightarrow C^\infty(SM,\pi^*\mc{E}), 
\]
defined by
\[
\pi_m^*\left(\sum_{k=1}^r u_k \otimes e_k\right) (x,v) := \sum_{k=1}^r ({u_k})_x(v,...,v) e_k(x).
\]
Similarly to \eqref{eq:pi_muntwisted}, the following mappings are isomorphisms (see \cite[Section 2]{Cekic-Lefeuvre-20}):

\begin{align*}
\pi_m^* &: C^\infty(M,\otimes^m_S T^*M \otimes \mc{E}) \rightarrow \oplus_{k \geq 0} C^\infty(M,\Omega_{m-2k} \otimes \E),\\
\pi_m^* &: C^\infty(M,\otimes^m_S T^*M|_{0-\Tr} \otimes \mc{E}) \rightarrow C^\infty(M,\Omega_m \otimes \E).
\end{align*}
We recall the notation $(\pi^* \nabla^{\mc{E}})_X := \X$. The following remarkable commutation property holds (see \cite[Section 2]{Cekic-Lefeuvre-20}):
\begin{equation}\label{eq:pullback}
	\forall m \in \mathbb{Z}_{\geq 0}, \quad \pi_{m+1}^* D^{\E} = \X \pi_m^*.
\end{equation}

The vector bundle $\otimes^m_S T^*M \otimes \mc{E}$ is naturally endowed with a canonical fiberwise metric induced by the metrics $g$ and $h$ which allows to define a natural $L^2$ scalar product. The $L^2$ formal adjoint $(D^{\E})^*$ of $D^{\E}$ is of divergence type (in the sense that its principal symbol is surjective for every $(x,\xi) \in T^*M \setminus \left\{ 0 \right\}$, see \cite[Definition 3.1]{Cekic-Lefeuvre-20} for further details). We call \emph{twisted solenoidal tensors} the elements in its kernel.

By ellipticity of $D^{\E}$, for any twisted $m$-tensor $f$ there exists a unique $p \in (\ker D^{\E})^\perp \cap C^\infty(M,\otimes^{m-1}_S T^*M \otimes \mc{E}), h \in C^\infty(M,\otimes^{m}_S T^*M \otimes \mc{E})$ such that:
\begin{equation}\label{eq:decomposition-tt}
f = D^{\E}p + h, \quad (D^{\E})^*h = 0.
\end{equation}
This decomposition bears resemblance with the Hodge decomposition of differential forms; we also note that \eqref{eq:decomposition-tt} could be extended to other regularities. We define $\pi_{\ker (D^{\E})^*}f := h$ as the $L^2$-orthogonal projection on twisted solenoidal tensors. This can be expressed as:
\begin{equation}
\label{equation:projection}
\pi_{\ker (D^{\E})^*} = \mathbbm{1} - D^{\E} [(D^{\E})^*D^{\E}]^{-1} (D^{\E})^*,
\end{equation}
where $[(D^{\E})^*D^{\E}]^{-1}$ is the resolvent of the operator $(D^{\E})^*D^{\E}$ (defined as follows: $[(D^{\E})^*D^{\E}]^{-1} = 0$ on $\ker (D^{\E})^*D^{\E}$, and on the $L^2$-orthogonal of $\ker (D^{\E})^*D^{\E}$ it is genuinely given by the inverse of $(D^{\E})^*D^{\E}$, well defined by Fredholm theory of elliptic operators).

\subsection{Pollicott-Ruelle resonances}

We explain the link between the widely studied notion of Pollicott-Ruelle resonances (see for instance \cite{Liverani-04, Gouezel-Liverani-06,Butterley-Liverani-07,Faure-Roy-Sjostrand-08,Faure-Sjostrand-11,Faure-Tsuji-13,Dyatlov-Zworski-16}) and the notion of (twisted) Conformal Killing Tensors introduced in the last paragraph. We also refer to \cite{Cekic-Lefeuvre-20} for an extensive discussion about this.

\subsubsection{Definition of the resolvents}

\label{ssection:resonances}

Let $\M$ be a smooth manifold endowed with a vector field $X \in C^\infty(\M,T\M)$ generating an Anosov flow in the sense of \eqref{equation:anosov}. Throughout this paragraph, we will always assume that the flow is volume-preserving. It will be important to consider the dual decomposition to \eqref{equation:anosov}, namely
\[
T^*(\M) = \R E_0^* \oplus E_s^* \oplus E_u^*,
\]
where $E_0^*(E_s \oplus E_u) = 0, E_s^*(E_s \oplus \R X) = 0, E_u^*(E_u \oplus \R X) = 0$. As before, we consider a vector bundle $\mc{E} \rightarrow \M$ equipped with a unitary connection $\nabla^{\E}$ and set $\X := \nabla^{\mc{E}}_X$. Since $X$ preserves a smooth measure $\dd \mu$ and $\nabla^{\mc{E}}$ is unitary, the operator $\X$ is skew-adjoint on $L^2(\M,\mc{E};\dd\mu)$, with dense domain
\begin{equation}
\label{equation:domaine-p}
\mc{D}_{L^2} := \left\{ u \in L^2(\M,\mc{E};\dd\mu) ~|~ \X u \in L^2(\M,\mc{E};\dd\mu)\right\}.
\end{equation}
Its $L^2$-spectrum consists of absolutely continuous spectrum on $i\R$ and of embedded eigenvalues. We introduce the resolvents
\begin{equation}
\label{equation:resolvent}
\begin{split}
&\RR_+(z) := (-\X-z)^{-1} = - \int_0^{+\infty} e^{-t z} e^{-t\X} \dd t, \\
&\RR_-(z) := (\X-z)^{-1} = - \int_{-\infty}^0 e^{z t} e^{-t\X} \dd t,
\end{split}
\end{equation}
initially defined for $\Re(z) > 0$. (Let us stress on the conventions here: $-\X$ is associated to the positive resolvent $\RR_+(z)$ whereas $\X$ is associated to the negative one $\RR_-(z)$.) Here $e^{-t\X}$ denotes the propagator of $\X$, namely the parallel transport by $\nabla^{\mc{E}}$ along the flowlines of $X$. If $\X = X$ is simply the vector field acting on functions (i.e. $\mc{E}$ is the trivial line bundle), then $e^{-tX}f(x) = f(\varphi_{-t}(x))$ is nothing but the composition with the flow.

There exists a family $\mc{H}^s_\pm$ of Hilbert spaces called \emph{anisotropic Sobolev spaces}, indexed by $s > 0$, such that the resolvents can be meromorphically extended to the whole complex plane by making $\X$ act on $\mc{H}^s_\pm$. The poles of the resolvents are called the \emph{Pollicott-Ruelle resonances} and have been widely studied in the aforementioned literature \cite{Liverani-04, Gouezel-Liverani-06,Butterley-Liverani-07,Faure-Roy-Sjostrand-08,Faure-Sjostrand-11,Faure-Tsuji-13,Dyatlov-Zworski-16}. Note that the resonances and the resonant states associated to them are intrinsic to the flow and do not depend on any choice of construction of the anisotropic Sobolev spaces. More precisely, there exists a constant $c > 0$ such that $\RR_\pm(z) \in \mc{L}(\mc{H}_\pm^s)$ are meromorphic in $\left\{\Re(z) > -cs\right\}$. For $\RR_+(z)$ (resp. $\RR_-(z)$), the space $\mc{H}^s_+$ (resp. $\mc{H}^s_-$) consists of distributions which are microlocally $H^s$ in a neighborhood of $E_s^*$ (resp. microlocally $H^{s}$ in a neighborhood of $E_u^*$) and microlocally $H^{-s}$ in a neighborhood of $E_u^*$ (resp. microlocally $H^{-s}$ in a neighborhood of $E_s^*$), see \cite{Faure-Sjostrand-11,Dyatlov-Zworski-16}. These spaces also satisfy $(\mc{H}^s_+)' = \mc{H}^s_-$ (where one identifies the spaces using the $L^2$-pairing).

From now on, we will assume that $s$ is fixed and small enough, and set $\mc{H}_\pm := \mc{H}_\pm^s$. We have
\begin{equation}
\label{equation:betaan}
H^{s} \subset \mc{H}_{\pm} \subset H^{-s}.
\end{equation}
and there is a certain strip $\left\{\Re(z) > -\eps_{\mathrm{strip}}\right\}$ (for some $\eps_{\mathrm{strip}} > 0$) on which $z \mapsto \RR_{\pm}(z) \in \mc{L}(\mc{H}_{\pm})$ is meromorphic (and the same holds for small perturbations of $\X$).

These resolvents satisfy the following equalities on $\mc{H}_\pm$, for $z$ not a resonance:
\begin{equation}
\label{equation:resolvent-identity}
\RR_\pm(z)(\mp\X- z) = (\mp\X- z) \RR_\pm(z) = \mathbbm{1}_{\mc{E}}.
\end{equation}
Given $z \in \C$ which not a resonance, we have:
\begin{equation}
\label{equation:adjoint}
\RR_+(z)^* = \RR_-(\overline{z}),
\end{equation}
where this is understood in the following way: given $f_1, f_2 \in C^\infty(\M,\mc{E})$, we have
\[
\langle \RR_+(z) f_1, f_2 \rangle_{L^2} = \langle f_1,  \RR_-(\overline{z}) f_2 \rangle_{L^2}.
\]
We will always use this convention for the definition of the adjoint. 

Since the connections are unitary and the flow preserves a smooth measure, the propagators $e^{-t \mathbf{X}}$ preserve the norm in $L^2(\mc{M}, \mc{E}; \dd \mu)$. As a consequence, the formulas \eqref{equation:resolvent} converge when $\re z > 0$ and thus we obtain the following statement that we record for future purposes:
\begin{equation}\label{eq:spectrum-left-half-plane}
	\mathrm{the\,\,resonance\,\,spectrum\,\,of\,\,\pm\X\,\,is\,\,contained\,\,in\,\,\,\,} \{z \in \mathbb{C} \mid \re(z) \leq 0\}.
\end{equation}

A point $z_0 \in \C$ is a resonance for $-\X$ (resp. $\X$) i.e. is a pole of $z \mapsto \RR_+(z)$ (resp. $\RR_-(z)$) if and only if there exists a non-zero $u \in \mc{H}^s_+$ (resp. $\mc{H}^s_-$) for some $s > 0$ such that $-\X u = z_0 u$ (resp. $\X u = z_0 u$). If $\gamma$ is a small counter clock-wise oriented circle around $z_0$, then the spectral projector onto the resonant states is
\[
\Pi_{z_0}^{\pm} = - \dfrac{1}{2\pi i} \int_{\gamma} \RR_{\pm}(z) \dd z =  \dfrac{1}{2\pi i} \int_{\gamma} (z \pm \X)^{-1} \dd z,
\]
where we use the abuse of notation that $-(\X+z)^{-1}$ (resp. $(\X-z)^{-1}$) to denote the meromorphic extension of $\RR_+(z)$ (resp. $\RR_-(z)$).

\subsubsection{Resonances at $z=0$}

\label{sssection:resonances-zero}

By the previous paragraph, we can write in a neighborhood of $z=0$ the following Laurent expansion (beware the conventions):
\[
\RR_+(z) = - \RR_0^+ - \dfrac{\Pi_0^+}{z} + \mc{O}(z).
\]
(Or in other words, using our abuse of notations, $(\X+z)^{-1} = \RR_0^+ + \Pi_0^+/z + \mc{O}(z)$.) And:
\[
\RR_-(z) = -\RR_0^- - \dfrac{\Pi_0^-}{z} + \mc{O}(z).
\]
(Or in other words, $(z - \X)^{-1} = \RR_0^- + \Pi_0^-/z + \mc{O}(z)$.) As a consequence, these equalities define the two operators $\RR_0^{\pm}$ as the holomorphic part (at $z=0$) of the resolvents $-\RR_\pm(z)$. We introduce:
\begin{equation}
\label{equation:pi}
\Pi := \RR_0^+ + \RR_0^-.
\end{equation}
We have:

\begin{lemma}
\label{lemma:relations-resolvent}
We have $(\RR_0^+)^* = \RR_0^{-}, (\Pi_0^+)^* = \Pi_0^- = \Pi_0^+$. Thus $\Pi$ is formally self-adjoint. Moreover, it is nonnegative in the sense that for all $f \in C^\infty(\M,\mc{E})$, $\langle \Pi f, f \rangle_{L^2} = \langle f, \Pi f \rangle_{L^2} \geq 0$. Also, $\langle \Pi f, f \rangle = 0$ if and only if $\Pi f = 0$ if and only if $f = \X u + v $, for some $u \in C^\infty(\M,\E), v \in \ker(\X|_{\mc{H}_\pm})$.
\end{lemma}

\begin{proof} See \cite[Lemma 5.1]{Cekic-Lefeuvre-20}. \end{proof}

We also record here for the sake of clarity the following identities:
\begin{equation}\label{eq:resolventrelations}
\begin{split}
& \Pi_0^{\pm} \RR_0^+ = \RR_0^+ \Pi_0^{\pm} = 0, \Pi_0^{\pm} \RR_0^- = \RR_0^- \Pi_0^{\pm} = 0,\\
& \X \Pi_0^\pm = \Pi_0^\pm \X = 0, \X \RR_0^+ = \RR_0^+ \X = \mathbbm{1} - \Pi_0^+, -\X \RR_0^- = -\RR_0^- \X = \mathbbm{1} - \Pi_0^-.
\end{split}
\end{equation}

\subsubsection{Perturbation theory of resonances} 
We will need to apply the framework of Pollicott-Ruelle resonances for connections with finite regularity. Consider $\nabla_0^{\E}$, an arbitrary unitary connection of regularity $C^s_*$ (with $s > 1$) on a smooth Hermitian vector bundle $\E \rightarrow \M$ and define the first order differential operator $\X_0 := (\nabla_0)^{\E}_X$ acting on sections of $\E$. 

\begin{lemma}\label{lemma:mero-finite-reg}
	There exists a constant $C > 0$, depending only on the vector field $X$, and anisotropic Sobolev spaces $\mc{H}_\pm$, such that the resolvents $z \mapsto \RR_\pm(z) = (\mp\X_0-z)^{-1} \in \mc{L}(\mc{H}_\pm)$ admit a meromorphic extension from $\left\{\re z \gg 0\right\}$ to $\left\{\re z \geq - Cs\right\}$.  \end{lemma}
	
	For a proof, we refer to the article by Guedes Bonthonneau-de Poyferr\'e-Guillarmou, see \cite[Theorem 3]{GuedesBonthonneau-Guillarmou-dePoyferre-21} (we note however that less precise statements were obtained by microlocal methods also by Dyatlov-Zworski, see for instance \cite[Remark (i) on page 4]{Dyatlov-Zworski-16}). It is also immediate to extend the perturbation theory of Pollicott-Ruelle resonances of Bonthonneau \cite[Corollary 1.2]{Bonthonneau-19} to finite regularity (in fact, our case is easier to handle because the perturbations we consider are by order zero terms):
	
\begin{lemma}
\label{lemma:perturbation-pr}
	Let $C > 0$ and $\mc{H}_\pm$ be as in Lemma \ref{lemma:mero-finite-reg}. Let $z_0 \in \mathbb{C}$ with $\re(z_0) > -Cs$ be a Pollicott-Ruelle resonance of $-\X_0$ and $\gamma$ be a small contour around $z_0$ enclosing no other resonances of $-\X_0$. Then, there exists an $\varepsilon > 0$ and $s_{0} \gg 1$, such that for any $s > s_{0}$, and any connection $\nabla^{\mc{E}} = \nabla_0^{\mc{E}} + A$ for some $A \in  C_*^s(\mc{M}; T^*\mc{M} \otimes \End(\E))$ such that $\|A\|_{C_*^s} < \varepsilon$, the projector
	\[\Pi_A := \frac{1}{2\pi i} \oint_{\gamma} (-\X - z)^{-1} dz\]
	is well-defined, and the map $A \mapsto \Pi_A \in \mc{L}(\mc{H}_+)$ is $C^\infty$-regular with locally constant rank; here $\X_A := \X_0 + V_A = \nabla_X^{\E}$ and $V_A := A(X)$. Moreover, the map $A \mapsto \lambda_A$ associating to $A$ the sum of the resonances of $-\X_A$ enclosed by $\gamma$ (with multiplicities) is smooth near $A = 0$.
\end{lemma}

\begin{proof}
	We sketch the proof for the convenience of the reader. Note that for any $s > s_1$, the multiplication map 
	\begin{equation}\label{eq:multiplication}
		C_*^s(\mc{M}; \End(\E)) \times \mc{H}_\pm \ni (V, u) \mapsto Vu \in \mc{H}_\pm,
	\end{equation}
	is continuous; this follows from \cite[Section 2]{GuedesBonthonneau-Guillarmou-dePoyferre-21}.

	Note that $V_A = A(X) \in C_*^s(\mc{M}; \End(\E))$. Then for any $z$ in a small neighbourhood of $\gamma$
	\[\X_0 + V_A + z = (\X_0 + z)(\mathbbm{1} + (\X_0 + z)^{-1}V_A)\]
	and so by \eqref{eq:multiplication}, the map $\|(\X_0 + z)^{-1}V\|_{\mc{H}_+ \to \mc{H}_+}$ has norm smaller then $1$ if $\varepsilon$ is small enough, so the operator $\mathbbm{1} + (\X_0 + z)^{-1}V_A$ is invertible on the domain $\mc{D}_+ = \{u \in \mc{H}_+ \mid \X_0 u \in \mc{H}_+\}$ (note that $\X_0 V_A = [\X_0, V_A] + V_A\X_0$ where $[\X_0, V_A] \in C_*^{s - 1}(\mc{M}; \End(\E))$, so the domain $\mc{D}_+$ is invariant under multiplication with $V$ when $s > s_1 + 1$). This implies that $\X_0 + V_A + z: \mc{D}_+ \to \mc{H}_+$ is invertible for $z$ in a neighbourhood of $\gamma$, with inverse bounded by some uniform constant.
	
	Therefore $\Pi_A$ is well-defined and we may differentiate to get, in the direction of some $B \in C_*^s(\mc{M}; T^*\mc{M} \otimes \End(\E))$:
	\[D_B \Pi_A = \frac{1}{2\pi i} \oint_{\gamma} (-\X_A - z)^{-1} V_B (-\X_A - z)^{-1} dz.\]
	It follows from \eqref{eq:multiplication} that $\|D_B \Pi_A\|_{\mc{H}_+ \to \mc{H}_+} \leq C_1 \|V_B\|_{\mc{H}_+ \to \mc{H}_+} \leq C_2 \|B\|_{C_*^s}$ for some uniform constants $C_1, C_2 > 0$. This shows that $A \mapsto \Pi_A \in \mc{L}(\mc{H}_+)$ is $C^1$, and iterating this argument shows that in fact this map is smooth.
	
	To show that the rank of $\Pi_A$ is locally constant we refer to \cite[Section 4]{Bonthonneau-19} (see also \cite[Section 6]{Cekic-Paternain-20}). Finally, consider a basis of (generalized) resonant states $(u_i)_{i = 1}^r \in \mc{H}_+$ of $-\X_0$, where $r$ is the rank of $\Pi_0$. Then the map $A \mapsto (\Pi_A u_i)_{i = 1}^r \in \mc{H}_+$ is smooth and so for small enough $\varepsilon$, the sequence $(\Pi_A u_i)_{i = 1}^r$ is a basis of the range of $\Pi_A$ (of generalized resonant states of $-\X_A$). The map $-\X_A$ acts on the range of $\Pi_A$ by definition and so since $\lambda_A$ equals the trace of $-\X_A$ in the constructed smooth basis, this gives the smoothness of $\lambda_A$.
\end{proof}

\subsection{Generalized X-ray transform}
\label{ssection:generalized-x-ray}

The discussion is carried out here in the closed case, but could also be generalized to the case of a manifold with boundary. We introduce the operator
\[
\Pi := \RR_0^+  + \RR_0^-,
\]
where $\RR_0^+$ (resp. $\RR_0^-$) denotes the holomorphic part at $0$ of $-\RR_+(z)$ (resp. $-\RR_-(z)$) and $\Pi^+_0$ is the $L^2$-orthogonal projection on the (smooth) resonant states at $0$. Such an operator was first introduced in the non-twisted case by Guillarmou \cite{Guillarmou-17-1}. The operator $\Pi + \Pi^+_0$ is the derivative of the (total) $L^2$-spectral measure at $0$ of the skew-adjoint operator $\X$.

\begin{definition}[Generalized X-ray transform of twisted symmetric tensors]
\label{definition:generalized-xray}
We define the generalized X-ray transform of twisted symmetric tensors as the operator:
\[
\Pi_m := {\pi_m}_* \left(\Pi + \Pi^+_0\right) \pi_m^*.
\]
\end{definition}

In what follows, we will mostly use this operator with $m=1$. In this case, the operator $\Pi_1$ takes a one-form valued in some bundle $\E$, pulls it back on the unit tangent bundle to a spherical harmonic of degree $1$ twisted by some bundle ($\pi_1^*$-operator), then ``averages'' this spherical harmonic along the geodesic flowlines ($(\Pi + \Pi^+_0)$-operator) and then selects the first spherical harmonic of this distribution in order to give a twisted one-form on the base manifold $M$ (${\pi_1}_*$-operator). We remark that when we want to emphasize the dependence of $\Pi_m$ on a connection $\nabla^{\mc{E}}$, we will write $\Pi_m^{\nabla^{\mc{E}}}$ (this will appear in particular in \S \ref{ssection:pollicott-ruelle-dea}).

\begin{remark}\rm
	This also allows to define a generalized (twisted) X-ray transform $\Pi_m$ for an arbitrary unitary connection $\nabla^{\mc{E}}$ on $\mc{E}$. Indeed, it is not clear a priori if one sticks to the usual definition of the X-ray transform that one can find a ``natural'' candidate for the X-ray transform on twisted tensors. For instance, one could consider the map
\[
\mc{C} \ni \gamma \mapsto I_m^{\nabla^{\E}}f (\gamma) := \dfrac{1}{\ell(\gamma)} \int_0^{\ell(\gamma)} (e^{-t \X} f) (x_\gamma, v_\gamma) \dd t \in \mc{E}_{x_\gamma},
\]
where $\gamma \in \mc{C}$ is a closed geodesic and $(x_\gamma,v_\gamma) \in \gamma$. However, this definition does depend on the choice of base point $(x_\gamma,v_\gamma) \in \gamma$ and it would no longer be true that $I^{\nabla^{\mc{E}}}_m (D^{\E}p) (\gamma) = 0$ unless the connection is transparent.
\end{remark}

By \eqref{eq:pullback} and \eqref{eq:resolventrelations}, we have the equalities:
\begin{equation}\label{eq:Pi_mproperties}
(D^{\E})^* \Pi_m = 0 = \Pi_m D^{\E},
\end{equation}
showing that $\Pi_m$ maps the set of twisted solenoidal tensors to itself. We say that the generalised $X$-ray transform is \emph{solenoidally injective} ($s$-injective) on $m$-tensors, if for all $u \in C^\infty(SM, \E)$ and $f \in C^\infty(M, \otimes_S^mT^*M \otimes \E)$
\begin{equation}\label{eq:cohomeqn}
	\X u = \pi_m^* f \implies \exists p \in C^\infty(M, \otimes_S^{m-1}T^*M \otimes \E)\,\, \mathrm{such\,\, that}\,\,f = D^{\E} p.
\end{equation}
We have the following:

\begin{lemma}\label{lemma:x-ray}
	The generalised $X$-ray transform is $s$-injective on $m$-tensors if and only if $\Pi_m$ is injective on solenoidal tensors (if this holds, we say $\Pi_m$ is \emph{$s$-injective}).
\end{lemma}
\begin{proof}
	Assume that $\Pi_m f = 0$ and $f$ is a twisted solenoidal $m$-tensor. Then
\[
\langle \Pi_m f, f \rangle_{L^2} = \langle \Pi \pi_m^*f , \pi_m^*f \rangle_{L^2} +  \langle \Pi^+_0 \pi_m^*f, \pi_m^* f \rangle_{L^2} = 0.
\]
Both terms on the right hand side are non-negative by Lemma \ref{lemma:relations-resolvent}, hence both of them vanish, and the same Lemma implies that $\Pi \pi_m^* f = 0$ and $\Pi^+_0 \pi_m^*f = 0$. Thus $\X u = \pi_m^*f$ for some smooth $u$, so by the $s$-injectivity of generalised $X$-ray transform we obtain $f$ is potential, which implies $f = 0$.

The other direction is obvious by \eqref{eq:resolventrelations}.
\end{proof}

Next, we show $\Pi_m$ enjoys good analytical properties:

\begin{lemma}
\label{lemma:generalized-xray}
The operator $\Pi_m : C^\infty(M,\otimes^m_S T^*M \otimes \mc{E}) \rightarrow C^\infty(M,\otimes^m_S T^*M \otimes \mc{E})$ is:
\begin{enumerate}
\item A pseudodifferential operator of order $-1$,
\item Formally self-adjoint and elliptic on twisted solenoidal tensors (its Fredholm index is thus equal to $0$ and its kernel/cokernel are finite-dimensional),
\item Under the assumption that $\Pi_m$ is $s$-injective, the following stability estimates hold: 
\[
\forall s \in \mathbb{R},\,\, \forall f \in H^s(M,\otimes^m_S T^*M \otimes \mc{E}), ~~~ \|\pi_{\ker (D^{\E})^*}f\|_{H^s} \leq C_s \|\Pi_m f \|_{H^{s+1}},
\]
for some $C_s > 0$ and for some $C > 0$:
\[
\forall f \in H^{-1/2}(M,\otimes^m_S T^*M \otimes \mc{E}), ~~~ \langle \Pi_m f, f \rangle_{L^2} \geq C \|\pi_{\ker (D^{\E})^*} f\|^2_{H^{-1/2}}.
\]
In particular, these estimates hold if $(M,g)$ has negative curvature and $\nabla^{\E}$ has no twisted CKTs.
\end{enumerate}
\end{lemma}

Point (3) is a quantitative improvement of the statement: $\Pi_m f = 0, f \in \ker (D^{\E})^* \implies f = 0$, i.e. it provides a stability estimate for the X-ray transform (see Lemma \ref{lemma:x-ray} for the relation between $\Pi_m$ and the X-ray transform).

\begin{proof}
The proof of the first two points follows from a rather straightforward adaptation of the proof of \cite[Theorem 2.5.1]{Lefeuvre-thesis} (see also \cite{Guillarmou-17-1} for the original arguments); we omit it. It remains to prove the third point. 

The first estimate follows from $(2)$, the elliptic estimate and the fact that $\Pi_m$ is $s$-injective. The last estimate in the non-twisted case follows from \cite[Lemma 2.1]{Guillarmou-Knieper-Lefeuvre-19} (or \cite[Theorem 2.5.6]{Lefeuvre-thesis}) and subsequent remarks; the twisted case follows by minor adaptations.

If $(M,g)$ has negative curvature and $\nabla^{\E}$ has no twisted CKTs, using Lemma \ref{lemma:x-ray} and by \cite[Sections 4, 5]{Guillarmou-Paternain-Salo-Uhlmann-16} we get that $\Pi_m$ is $s$-injective, proving the claim.
\end{proof}

\section{Exact Liv\v{s}ic cocycle theory}

We phrase this section in a very general context which is that of a transitive Anosov flow on a smooth manifold. It is of independent interest to the rest of the article. 

\label{section:livsic}

\subsection{Statement of the results}

\label{ssection:intro}

\subsubsection{A weak exact Liv\v{s}ic cocycle theorem}

Let $\M$ be a smooth closed manifold endowed with a flow $(\varphi_t)_{t \in \R}$ with infinitesimal generator $X \in C^\infty(\M,T\M)$. We assume that the flow is Anosov in the sense of \eqref{equation:anosov} and that it is \emph{transitive}, namely it admits a dense orbit\footnote{Note that there are examples of non-transitive Anosov flows, see \cite{Franks-Williamas-80}.}. We denote by $\mc{G}$ the set of all periodic orbits for the flow and by $\mc{G}^\sharp$ the set of all \emph{primitive} orbits, namely orbits which cannot be written as a shorter orbit to some positive power greater or equal than $2$.

Let $(\mc{E},\nabla^{\mc{E}})$ be a smooth Hermitian vector bundle of rank $r$ equipped with a unitary connection $\nabla^{\mc{E}}$. We will denote by 
\[
C(x,t) : \mc{E}_x \rightarrow \mc{E}_{\varphi_t(x)},
\]
 the parallel transport along the flowlines of $(\varphi_t)_{t \in \R}$ with respect to the connection $\nabla^{\mc{E}}$. In the more general setting, we may consider $\E_1, \E_2 \rightarrow \M$, two Hermitian vector bundles, equipped with two respective unitary connections $\nabla^{\mc{E}_1}$ and $\nabla^{\mc{E}_2}$. 
Recall that if $\nabla^{\E_2} = p^*\nabla^{\E_1}$, for some unitary map $p \in C^\infty(\M,\mathrm{U}(\E_2, \E_1))$\footnote{Here, we denote by $\mathrm{U}(\E_2, \E_1) \rightarrow \M$ the bundle of unitary maps from $\E_2 \rightarrow \E_1$. Of course, it may be empty if the bundles are not isomorphic.}, i.e. the connections are gauge-equivalent, then parallel transport along the flowlines of $(\varphi_t)_{t \in \R}$ satisfies the commutation relation:
\[
C_1(x,t)= p(\varphi_t x) C_2(x,t) p(x)^{-1}.
\]
We say that such cocycles are \emph{cohomologous}. In particular, given a closed orbit $\gamma = (\varphi_t x_0)_{t \in [0,T]}$ of the flow, one has
\[
C_1(x_0,T) = p(x_0) C_2(x_0,T) p(x_0)^{-1},
\]
i.e. the parallel transport map are conjugate.

\begin{definition}
\label{definition:equivalence}
We say that the connections $\nabla^{\E_{1,2}}$ have 
\emph{trace-equivalent holonomies} if for all primitive closed orbits $\gamma \in \mc{G}^\sharp$, we have:
\begin{equation}
\label{equation:trace-equivalence}
\Tr(C_1(x_\gamma,\ell(\gamma))) = \Tr(C_2(x_\gamma,\ell(\gamma))),
\end{equation}
where $x_\gamma \in \gamma$ is arbitrary and $\ell(\gamma)$ is the period of $\gamma$.
\end{definition}

This condition could be \emph{a priori} obtained with $\rk(\E_1) \neq \rk(\E_2)$. We shall see that this cannot be the case.
The following result is one of the main theorems of this paper. It seems to improve known results on Liv\v{s}ic cocycle theory (in particular \cite{Parry-99,Schmidt-99}), see \S\ref{ssection:exact} for a more extensive discussion on the literature.

\begin{theorem}
\label{theorem:weak}
Assume $\M$ is endowed with a smooth transitive Anosov flow. Let $\E_1, \E_2 \rightarrow \M$ be two Hermitian vector bundles over $\M$ equipped with respective unitary connections $\nabla^{\mc{E}_1}$ and $\nabla^{\mc{E}_2}$. If the connections have \emph{trace-equivalent holonomies} in the sense of Definition \ref{definition:equivalence}, then there exists $p \in C^\infty(\M,\mathrm{U}(\E_2,\E_1))$ such that: for all $x \in \M, t \in \R$,
\begin{equation}
\label{equation:cohomol}
C_1(x,t) = p(\varphi_t x) C_2(x,t) p(x)^{-1},
\end{equation}
i.e. the cocycles induced by parallel transport are cohomologous. Moreover, $\E_2 \simeq \E_1$ are isomorphic.
\end{theorem}

In order to prove the injectivity Theorem \ref{theorem:injectivity}, we will apply the previous Theorem \ref{theorem:weak} with $\M = SM$, the geodesic flow, pullback bundle $\pi^*\E$ equipped with two pullback connections $\pi^*\nabla_1, \pi^*\nabla_2$. However, as we shall see in \S\ref{section:proofs}, Theorem \ref{theorem:weak} will not directly imply Theorem \ref{theorem:injectivity}: indeed, after differentiating \eqref{equation:cohomol} at $t = 0$, it only gives the existence of a map $p \in C^\infty(SM,\mathrm{U}(\pi^*\E))$ such that 
\[\pi^* \nabla^{\mathrm{Hom}(\nabla_2, \nabla_1)}_X p = 0.\]
We will then have to prove that $p = \pi^* p_0$ for some unitary isomorphism on the base $p_0 \in C^\infty(M,\mathrm{U}(\E))$ such that $\nabla^{\mathrm{Hom}(\nabla_2,\nabla_1)}p_0=0$ (this is equivalent to the connections being gauge-equivalent, as follows directly from Definition \ref{definition:mixed}).

\begin{remark}\rm
	The simplest example in which the automorphism $p$ does not descend to the base can be constructed as follows (originally in \cite{Paternain-09}). If $(\Sigma, g_\Sigma)$ is a Riemannian surface of negative curvature, then along any closed geodesic $\gamma$ the parallel transport with respect to Levi-Civita connection $\nabla^{\mathrm{LC}}$ is the identity (as $P_\gamma$ fixes $\dot{\gamma}$ and hence also its normal $\dot{\gamma}^\perp$). Thus $\pi^*\nabla^{\mathrm{LC}}$ on the pull-back $\pi^*\mc{K}$ of the canonical bundle $\mc{K} = (T_{\mathbb{C}}^*\Sigma)^{0, 1}$ and the trivial connection on $S\Sigma \times \mathbb{C}$ have trace equivalent holonomies. By Theorem \ref{theorem:weak}, there is $p \in C^\infty(\M,\mathrm{U}(\pi^*\mc{K}, \mathbb{C}))$ such that $\pi^*\nabla_X^{\operatorname{Hom}(\nabla^{\operatorname{LC}}, d)} p = 0$, but clearly $p$ is not of degree zero as the bundle $\mc{K}$ is not topologically trivial. In fact $p$ can be chosen to be of degree $1$ and similar examples exist in higher dimensions (see \cite{Cekic-Lefeuvre-23-polynomial}).
\end{remark}

As we shall see in the proof, for any given $L > 0$, it suffices to assume that the trace-equivalent holonomy condition \eqref{equation:trace} holds for all primitive periodic orbits of length $\geq L$ in order to get the conclusion of the theorem. Surprisingly, the rather weak condition \eqref{equation:trace} implies in particular that the bundles are isomorphic as stated in Corollary \ref{corollary:iso} and the trace of the holonomy of unitary connections along closed orbits should allow one in practice to classify vector bundles over manifolds carrying Anosov flows. Even more surprisingly, the rank of $\E_1$ and $\E_2$ might be \emph{a priori} different and Theorem \ref{theorem:weak} actually shows that the ranks have to coincide.

The idea relies on a key notion which we call \emph{Parry's free monoid}, whose introduction goes back to Parry \cite{Parry-99}. This free monoid $\mathbf{G}$ corresponds (at least formally) to the free monoid generated by the set of homoclinic orbits to a given periodic orbit of a point $x_{\star}$ (see \S\ref{sssection:homoclinic} for a definition) and we shall see that a connection induces a unitary representation $\rho : \mathbf{G} \rightarrow \mathrm{U}(\E_{x_{\star}})$ (it is not canonical but we shall see that its important properties are). Geometric properties of the connection can be read off this representation, see Theorem \ref{theorem:iso} below. Moreover, tools from representation theory can be applied and this is eventually how we will prove Theorem \ref{theorem:weak}.

\subsubsection{Opaque and transparent connections}

Theorem \ref{theorem:weak} has an interesting straightforward corollary. Recall that a unitary connection is said to be \emph{transparent} if the holonomy along all closed orbits is trivial.

\begin{corollary}
Assume $\M$ is endowed with a smooth transitive Anosov flow. Let $\E \rightarrow \M$ be a Hermitian vector bundle over $\M$ of rank $r$ equipped with a unitary connection $\nabla^{\mc{E}}$. If the connection is transparent, then $\E$ is trivial and trivialized by a smooth orthonormal family $e_1,...,e_r \in C^\infty(\M,\E)$ such that $\nabla^{\E}_X e_i = 0$.
\end{corollary}

In order to prove the previous corollary, it suffices to apply Theorem \ref{theorem:weak} with $\E_1 =  \E$ equipped with $\nabla^{\E}$ and $\E_2 = \C^r \times M$ equipped with the trivial flat connection. Then $C_2(x,t) = \mathbbm{1}$ and $(e_1,...,e_n)$ is obtained as the image by $p$ of the canonical basis of $\C^n$. This corollary seems to be known in the folklore but nowhere written. It is stated in \cite[Proposition 9.2]{Paternain-12} under the extra-assumption that $\E \oplus \E^*$ is trivial.

The ``opposite" notion of transparent connections is that of \emph{opaque} connections which are connections that do not preserve any non-trivial subbundle $\mc{F} \subset \E$ by parallel transport along the flowlines of $(\varphi_t)_{t \in \R}$. It was shown in \cite[Section 5]{Cekic-Lefeuvre-20} that the opacity of a connection is equivalent to the fact that
\[
\ker (\nabla^{\End}_X|_{C^\infty(\M,\End(\E))}) = \C \cdot \mathbf{1}_{\E}.
\]
Also note that when $X$ is volume-preserving, this corresponds to the Pollicott-Ruelle (co)resonant states at $0$ associated to the operator $\nabla^{\End}_X$. We shall also connect this notion with the representation $\rho : \mathbf{G} \rightarrow \mathrm{U}(\E_{x_{\star}})$ of the free monoid:

\begin{prop}
\label{proposition:opaque}
The following statements are equivalent:
\begin{enumerate}
\item The connection $\nabla^{\E}$ is opaque;
\item $\ker (\nabla^{\End}_X|_{C^\infty(\M,\End(\E))}) = \C \cdot \mathbf{1}_{\E}$;
\item The representation $\rho : \mathbf{G} \rightarrow \mathrm{U}(\E_{x_{\star}})$ is irreducible.
\end{enumerate}
\end{prop}

\subsubsection{Kernel of the endomorphism connection}

The previous proposition actually follows from a more general statement which we now describe. The representation $\rho : \mathbf{G} \rightarrow \mathrm{U}(\E_{x_{\star}})$ gives rise to an orthogonal splitting 
\[
\E_{x_{\star}} = \oplus_{i=1}^K \E_i^{\oplus n_i},
\]
where $\E_i \subset \E_{x_\star}$ and $n_i \geq 1$; each factor $\E_i$ is $\mathbf{G}$-invariant and the induced representation on each factor is irreducible; furthermore, for $i \neq j$, the induced representations on $\E_i$ and $\E_j$ are not isomorphic. Let $\C[\mathbf{G}]$ be the formal algebra generated by $\mathbf{G}$ over $\C$ and let $\mathbf{R} := \rho\left( \C[\mathbf{G}] \right)$. By Burnside's Theorem (see \cite[Corollary 3.3]{Lang-02} for instance), one has that:
\[
\mathbf{R} = \oplus_{i=1}^K \Delta_{n_i} \End(\E_i),
\]
where $\Delta_{n_i} u = u \oplus ... \oplus u$ for $u \in \End(\E_i)$, the sum being repeated $n_i$-times. We introduce the \emph{commutant} $\mathbf{R}'$ of $\mathbf{R}$, defined as:
\[
\mathbf{R}' := \left\{ u \in \End(\E_{x_\star}) ~|~ \forall v \in \mathbf{R}, uv = vu\right\}.
\]
We then have:

\begin{theorem}
\label{theorem:iso}
There exists a natural isomorphism:
\[
\Phi : \mathbf{R}' \rightarrow \ker \nabla^{\End(\E)}_X|_{C^\infty(\M,\End(\E))}.
\]
In particular these spaces have same dimension, that is
\[
\dim\left( \ker \nabla^{\End(\E)}_X|_{C^\infty(\M,\End(\E))} \right) = \dim(\mathbf{R}') = \sum_{i=1}^K n_i^2.
\]
\end{theorem}

\subsubsection{Invariant sections} To conclude this paragraph, we now investigate the existence of smooth \emph{invariant} sections of the bundle $\E \to \M$, namely elements of $\ker \nabla^{\E}_X|_{C^\infty(\M,\E)}$. First of all, observe that if $u \in C^\infty(\M,\E)$ is an invariant section, then $u_\star:= u(x_\star)$ is invariant by the $\mathbf{G}$-action. The converse is also true:

\begin{lemma}
\label{lemma:invariant-section}
Assume that there exists $u_\star \in \E_{x_\star}$ such that $\rho(g)u_\star = u_\star$ for all $g \in \mathbf{G}$. Then, there exists (a unique) $u \in C^\infty(\M,\E)$ such that $u(x_\star)=u_\star$ and $\nabla^{\E}_X u = 0$.
\end{lemma}

Such an approach turns out to be useful when trying to understand a sort of \emph{weak version} of Liv\v{s}ic theory, such as the following: if $\E \rightarrow \M$ is a vector bundle equipped with the unitary connection $\nabla^{\E}$ and for each periodic orbit $\gamma \in \mc{G}$, there exists a section $u_\gamma \in C^\infty(\gamma,\E|_{\gamma})$ such that $\nabla^{\E}_X u_{\gamma} = 0$, then one can wonder if this implies the existence of a global invariant smooth section $u \in C^\infty(\M,\E)$? It turns out that the answer depends on the rank of $\E$:

\begin{lemma}
\label{lemma:rank2}
Assume that $\rk(\E) \leq 2$ and that for all periodic orbits $\gamma \in \mc{G}$, there exists $u_\gamma \in C^\infty(\gamma,\E|_{\gamma})$ such that $\nabla^{\E}_X u_{\gamma} = 0$. Then, there exists $u \in C^\infty(\M,\E)$ such that $\nabla^{\E}_X u = 0$.
\end{lemma}

We shall see that the proof of the previous Lemma is purely representation-theoretic and completely avoids the need to understand dynamics and the distribution of periodic orbits. We leave as an exercise for the reader the fact that Lemma \ref{lemma:rank2} does not hold when $\rk(\E) \geq 3$. A simple counter-example can be built using the following argument: any matrix in $\mathrm{SO}(3)$ preserves an axis; hence, taking any $\mathrm{SO}(3)$-connection on a \emph{real} vector bundle of rank $3$ and then complexifying the bundle, one gets a vector bundle and a connection satisfying the assumptions of Lemma \ref{lemma:rank2}; it then suffices to produce an $\mathrm{SO}(3)$-connection without any invariant sections.

We believe that other links between properties of the representation $\rho$ and the geometry and/or dynamics of the parallel transport along the flowlines could be discovered. To conclude, let us also mention that all the results are presented here for complex vector bundles; most of them could be naturally restated for real vector bundles modulo the obvious modifications in the statements.

\subsection{Dynamical preliminaries on Anosov flows}

\label{section:anosov}

\subsubsection{Shadowing lemma and homoclinic orbits}

Fix an arbitrary Riemannian metric $g$ on $\M$. As usual, we define the \emph{local strong (un)stable manifolds} as:
\[
\begin{array}{l}
W^{s}_{\delta}(x) :=\left\{y \in \M, ~ \forall t \geq 0, d(\varphi_t y, \varphi_t x) < \delta, d(\varphi_t x, \varphi_t y) \rightarrow_{t \rightarrow +\infty} 0 \right\}, \\
W^{u}_{\delta}(x) :=\left\{y \in \M, ~ \forall t \leq 0, d(\varphi_t y, \varphi_t x) < \delta, d(\varphi_t x, \varphi_t y) \rightarrow_{t \rightarrow -\infty} 0 \right\},
\end{array}
\]
where $\delta > 0$ is chosen small enough. For $\delta = \infty$, we obtain the sets $W^{s,u}(x)$ which are the strong stable/unstable manifolds of $x$. We also set $W^{s,u}_{\mathrm{loc}}(x) =:  W^{s,u}_{\delta_0}(x)$ for some fixed $\delta_0 > 0$ small enough. The \emph{local weak (un)stable manifolds} $W^{ws,wu}_{\delta}(x)$ are the set of points $y \in B(x,\delta)$ such that there exists $t \in \R$ with $|t| < \delta$ and $\varphi_t y \in W^{s,u}_{\mathrm{loc}}(x)$.
The following lemma is known as the \emph{local product structure} (see \cite[Theorem 5.1.1]{Fisher-Hasselblatt-19} for more details):

\begin{lemma}
\label{lemma:product-structure}
There exists $\eps_0, \delta_0 > 0$ small enough such that for all $x,y \in \M$ such that $d(x,y) < \eps_0$, the intersection $W^{wu}_{\delta_0}(x) \cap W^{s}_{\delta_0}(y)$ is precisely equal to a unique point $\left\{z\right\}$. We write $z := \llbracket x,y \rrbracket$.
\end{lemma}

The main tool we will use to construct suitable \emph{homoclinic orbits} is the following classical shadowing property of Anosov flows for which we refer to \cite[Corollary 18.1.8]{Hasselblatt-Katok-95}, \cite[Theorem
5.3.2]{Fisher-Hasselblatt-19} and \cite[Proposition 6.2.4]{Fisher-Hasselblatt-19}. For the sake of simplicity, we now write $\gamma=[xy]$ if $\gamma$ is an orbit segment of the flow with endpoints $x$ and $y$.

\begin{theorem}
\label{theorem:shadowing} There exist $\eps_0>0$, $\theta>0$ and $C>0$ with the
following property. Consider $\eps<\eps_0$, and a finite or infinite sequence
of orbit segments $\gamma_i = [x_iy_i]$ of length $T_i$ greater than $1$ such
that for any $n$, $d(y_n,x_{n+1}) \leq \eps$. Then there exists a genuine
orbit $\gamma$ and times $\tau_i$ such that $\gamma$ restricted to $[\tau_i,
\tau_i+T_i]$ \emph{shadows} $\gamma_i$ up to $C\eps$. More precisely, for all $t\in
[0, T_i]$, one has
\begin{equation}
\label{eq:d_hyperbolic}
  d(\gamma(\tau_i+t), \gamma_i(t)) \leq C \eps e^{-\theta \min(t, T_i-t)}.
\end{equation}
Moreover, $|\tau_{i+1} - (\tau_i + T_i)| \leq C \eps$. Finally, if the
sequence of orbit segments $\gamma_i$ is periodic, then the orbit $\gamma$ is
periodic.
\end{theorem}

It is instructive for the reader to have Figure \ref{fig:ASgeometry}(A) in mind, where the upper curve corresponds to the orbit $\gamma$ approximating the segment given by the lower curve. Let us also make the following important comment. In the previous theorem, one can also allow the first orbit segment $\gamma_i$ to be
infinite on the left, and the last orbit segment $\gamma_j$ to be infinite on
the right. In this case,~\eqref{eq:d_hyperbolic} should be replaced by: assuming that $\gamma_i$ is defined on $(-\infty, 0]$
and $\gamma_j$ on $[0,+\infty)$, we would get for some $\tilde\tau_{i+1}$
within $C\eps$ of $\tau_{i+1}$, and all $t\geq 0$
\begin{equation*}
  d(\gamma(\tilde\tau_{i+1}-t), \gamma_i(-t)) \leq C \eps e^{-\theta t}, \quad d(\gamma(\tau_{j}+t), \gamma_j(t)) \leq C \eps e^{-\theta t}.
\end{equation*}

\label{sssection:homoclinic}

Fix an arbitrary periodic point $x_{\star} \in \M$ of period $T_{\star}$ and denote by $\gamma_{\star}$ its primitive orbit.

\begin{definition}[Homoclinic orbits]
\label{definition:homoclinic}
A point $p \in \M$ is said to be \emph{homoclinic} to $x_{\star}$ if $p \in W^{ws}(x_{\star}) \cap W^{wu}(x_{\star})$ (in other words, $d(\varphi_{t+t^\pm_0} p, \varphi_t x_{\star}) \rightarrow_{t \rightarrow \pm \infty} 0$ for some $t^\pm_0 \in \R$). We say that an orbit $\gamma$ is homoclinic to $x_{\star}$ if it contains a point $p \in \gamma$ that is homoclinic to $x_{\star}$ and we denote by $\mc{H} \subset \M$ the set of homoclinic orbits to $x_{\star}$.
\end{definition}

Note that due to the hyperbolicity, the convergence of the point $p$ to $x_{\star}$ is exponentially fast. More precisely, let $\gamma$ be the orbit of $p$ and let $\R \ni t \mapsto \gamma(t)$ be the flow parametrization of $\gamma$. Then, there exists uniform constants $C,\theta > 0$ (independent of $\gamma$) and $A_\pm \in \R$ (depending on $\gamma$) such that the following holds:
\begin{equation}
\label{equation:distance}
d(\gamma(A_\pm \pm n T_{\star}), x_{\star}) \leq C e^{-\theta n}.
\end{equation}
The points $\gamma(A_\pm)$ correspond to an arbitrary choice of points in $W^{s,u}_{\delta_0}(x_{\star})\cap \gamma$ (for some arbitrary $\delta_0 > 0$ small enough). Homoclinic orbits have infinite length (except the orbit of $x_{\star}$ itself) but it will be convenient to introduce a notion of \emph{length} $T_\gamma$ which we define to be equal to $T_\gamma:=A_+-A_-$ (note that this is a highly non-canonical definition). We define the trunk to be equal to the central segment $\gamma([A_-,A_+])$. In other words, the length of $\gamma$ is equal to the length of its trunk. We also define the points: $x_n^\pm := \gamma(A_\pm \pm nT_{\star})$. Note that another choice of values $A'_\pm$ has to differ from $A_\pm$ by $mT_{\star}$ for some $m \in \Z$. Homoclinic orbits will play a key role as we shall see in due course. 

\begin{lemma}
\label{lemma:dense}
Assume that the flow is transitive. Then the set $\mc{W}$ of points belonging to a homoclinic orbit in $\mc{H}$ is dense in $\M$.
\end{lemma}

\begin{proof}
This is a straightforward consequence of the shadowing Theorem \ref{theorem:shadowing}: one concatenates a long segment $S$ of a transitive orbit with $\gamma_{\star}$ i.e. one applies Theorem \ref{theorem:shadowing} with $... \gamma_{\star} \gamma_{\star} S \gamma_{\star} \gamma_{\star} ...$.
\end{proof}

\begin{remark}
\rm
In the particular case of an Anosov geodesic flow on the unit tangent bundle, one can check that $\mc{H}$ is in one-to-one correspondence with $\pi_1(M)/\langle \widetilde{\gamma_{\star}} \rangle$, where $\widetilde{\gamma_{\star}} \in \pi_1(M)$ is any element such that the conjugacy class of $\widetilde{\gamma_{\star}}$ in $\pi_1(M)$ corresponds\footnote{Recall that the set of free homotopy classes $\mc{C}$ is in one-to-one correspondence with conjugacy classes of $\pi_1(M)$, see \cite[Chapter 1]{Hatcher-02}.} to the free homotopy class $c \in \mc{C}$ whose unique geodesic representative is $\gamma_{\star}$.  
\end{remark}

\subsubsection{Applications of the Ambrose-Singer formula}

Consider a Hermitian vector bundle $\E$ over $(\M, g)$ equipped with a unitary connection $\nabla = \nabla^{\E}$. If $x, y \in \M$ are at a distance less than the injectivity radius of $\M$, denote by $C_{x \to y}: \E_x \to \E_y$ the parallel transport with respect to $\nabla^{\E}$ along the shortest geodesic from $x$ to $y$, by $C(x, t): \E_x \to \E_{\varphi_tx}$ the parallel transport along the flow and by $C_\gamma$ the parallel transport along a curve $\gamma$. For $U \in C^\infty(\M,\mathrm{End}(\mc{E}))$, we define, for all $x \in \M$, 
\[
\|U\|_x := \Tr(U^*(x)U(x))^{1/2},
\]
and $\|U\|_{L^\infty} := \sup_{x \in \M} \|U\|_x$. In particular, if $U \in C^\infty(\M,\mathrm{U}(\mc{E}))$ is unitary, then $\|U\|_{L^\infty} = \sqrt{\rk(\E)}$. We record the following consequences of Lemma \ref{lemma:ambrosesinger}:

\begin{lemma}\label{lemma:ASgeometry}
The following consequences of the Ambrose-Singer formula hold:
\begin{enumerate}
	\item Assume we are in the setting of Theorem \ref{theorem:shadowing}: for some $C, \varepsilon, T > 0$, let $x, p \in \M$ satisfy $d(\varphi_t x, \varphi_t p) \leq C \eps e^{-\theta \min(t,T-t)}$ for all $t \in [0, T]$. Then for any $0 \leq T_1 \leq T$:
	\[\|C(\varphi_{T_1}x, -T_1)C_{\varphi_{T_1}p \to \varphi_{T_1}x}C(p, T_1) C_{x \to p} - \mathbbm{1}_{\E_x}\|_x \leq \frac{c_0C \varepsilon}{\theta} \times \|F_\nabla\|_{C^0},\]
	where $c_0 = c_0(X, g) > 0$ depends only on the flow $X$ and the metric.
	
	\item Assume $\gamma \subset B(p, \imath/2)$ is a closed piecewise smooth curve at $p$ of length $L$, where $\imath$ denotes the injectivity radius of $(\mathcal{M}, g)$. Then for some $C = C(g) > 0$ depending on the metric:
	\[\|C_\gamma - \id_p\|_p \leq C L \times \sup_{y \in \gamma} d(p, y) \times \|F_\nabla\|_{C^0}.\]
	
	\item Let $\gamma: [0, L] \to M$ be a unit speed curve based at $p$, and $\nabla'$ be a second unitary connection on $\E$, whose parallel transport along $\gamma$ we denote by $C'_\gamma$. Then:
	\[\|C_{\gamma}^{-1}C'_{\gamma} - \id_p\|_p \leq L \times \|\nabla - \nabla'\|_{C^0}.\]
	\end{enumerate}
\end{lemma}
The geometries appearing in (1), (2) and (3) are depicted in Figure \ref{fig:ASgeometry} (A), (B) and (C), respectively.
 
 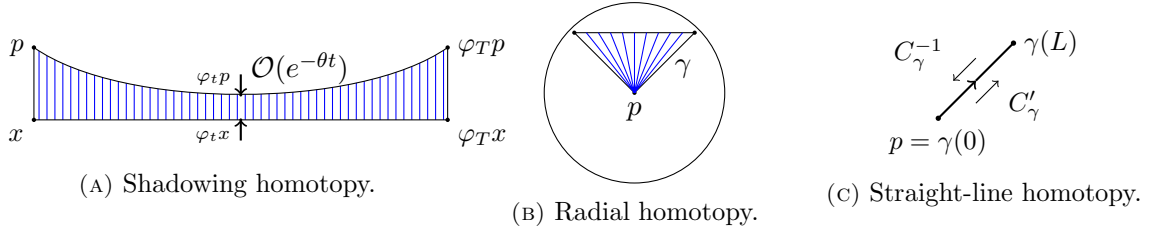
\begin{figure}
 \centering
\begin{subfigure}{0.39\linewidth}
             \centering
\begin{tikzpicture}[scale = 0.55, everynode/.style={scale=0.5}]
\tikzset{cross/.style={cross out, draw=black, minimum size=2*(#1-\pgflinewidth), inner sep=0pt, outer sep=0pt},
cross/.default={1pt}}

		\fill[pattern=vertical lines, pattern color=blue, draw = black] (0, 0) -- (0, 1.75) to[out = -35, in= -145, distance=75] (10, 1.75) -- (10, 0) --  (0,0);

		\fill (0, 0) node[below left] {\small $x$} circle (1.5pt);
		\fill (10, 0) node[below right] {\small $\varphi_Tx$}  circle (1.5pt);
		\fill (0, 1.75) node[left] {\small $p$}  circle (1.5pt);
		\fill (10, 1.75) node[right] {\small $\varphi_T p$}  circle (1.5pt);

		\draw[thick, ->] (5, 1.1) -- (5, 0.6) node[above left] {\tiny $\varphi_t p$};
		\draw (5, 1.25) node[right] {$\mc{O}(e^{-\theta t})$};
		\fill (5, 0.6) circle (1pt);
		\draw[thick, ->] (5, -0.5) -- (5, 0) node[below left] {\tiny $\varphi_t x$};		
		\fill (5, 0) circle (1pt);
\end{tikzpicture}
             \caption{\small Shadowing homotopy.}
             \end{subfigure}
             \begin{subfigure}{0.29\linewidth}
              \centering
             		\begin{tikzpicture}[scale=0.8]
             		\draw (0, 0) circle (1.5);
             			\draw (0, 0) node[below]{$p$}--(-1, 1)--(1,1)--(0,0);
             			\fill (0, 0) circle (1pt) (1, 1) circle (1pt) (-1, 1) circle (1pt);
             			\draw[blue] (0, 0)--(-0.8, 1) (0, 0)--(-0.6,1) (0, 0)--(-0.4,1) (0, 0)--(-0.2,1) (0, 0)--(-0,1) (0, 0)--(0.2,1) (0, 0)--(0.4,1) (0, 0)--(0.6,1) (0, 0)--(0.8,1);
             		\draw (0.5, 0.4) node[right]{$\gamma$};
             		\end{tikzpicture}
             		 \caption{\small Radial homotopy.}
             \end{subfigure}
             \begin{subfigure}{0.29\linewidth}
             \centering
             		\begin{tikzpicture}
             			\draw[thick, ->] (0, 0) node[below]{\small $p = \gamma(0)$}--(0.5, 0.5);
             			\draw[thick] (0.5, 0.5)--(1, 1) node[right]{\small $\gamma(L)$};
             			\fill (0, 0) circle(1pt) (1, 1) circle(1pt);
             			\draw[->] (0.5, 0.2)--(0.8, 0.5) node[below right]{\small $C'_\gamma$};
             			\draw[->] (0.5, 0.8)--(0.2, 0.5) node[above left]{\small $C_\gamma^{-1}$};
             		\end{tikzpicture}
             		             		 \caption{\small Straight-line homotopy.}
             \end{subfigure}
              \caption{\small Presentation of the geometries considered in Lemma \ref{lemma:ASgeometry}.}
             \label{fig:ASgeometry}
\end{figure}
\begin{proof}
	We first prove (1). For $C, \varepsilon$ small enough, for all $t \in [0, T]$ we denote by $\tau_t$ the unit speed shortest geodesic, of length $\ell(t)$, from $\varphi_tx$ to $\varphi_t p$. Define a smooth homotopy $\Gamma: [0, 1]^2 \to \M$ by setting:
	\[\Gamma(s, t) := \tau_{tT_1}(s \ell(t)),\]
	and note that by assumption $\ell(t) \leq C\varepsilon e^{-\theta \min(t, T - t)}$. We apply Lemma \ref{lemma:ambrosesinger} to the homotopy $\Gamma$ to obtain, after a rescaling of parameters $s$ and $t$:
	\begin{multline*}
		C(\varphi_{T_1}x, -T_1)C_{\varphi_{T_1}p \to \varphi_{T_1}x}C(p, T_1) C_{x \to p} - \mathbbm{1}_{\E_x}\\
		= \int_0^{T_1} \int_0^{\ell(t)} C_{\uparrow}^{-1}(s, t) F_\nabla(\partial_t \tau_t(s), \partial_s \tau_t(s)) C_{\rightarrow}(s, t) \, ds \, dt.
	\end{multline*}
	Here we recall $C_{\uparrow}$ and $C_{\rightarrow}$ are parallel transport maps obtained by parallel transport along curves as in Figure \ref{fig:AS1}. Since $C_{\uparrow}$ and $C_{\rightarrow}$ are isometries, and since by compactness $|\partial_t \tau_t(s)| \leq D$ for some $0 < D = D(X, g)$, we have:
	\begin{multline*}
		\|C(\varphi_{T_1}x, -T_1)C_{\varphi_{T_1}p \to \varphi_{T_1}x}C(p, T_1) C_{x \to p} - \mathbbm{1}_{\E_x}\|_x\\
		 \leq CD\varepsilon \|F_{\nabla}\|_{C^0} \int_0^T e^{-\theta \min(t, T - t)} dt \leq \frac{2D C}{\theta} \times \varepsilon \|F_\nabla\|_{C^0}.
	\end{multline*}
	
	For (2), we may assume by approximation that $\gamma$ is smooth. Then taking the homotopy 
	\[\Gamma(s, t) = \exp_x( t \exp_x^{-1} (\gamma(sL))),\]
	and applying Lemma \ref{lemma:ambrosesinger}, we obtain by a rescaling of $s$ and writing $\widetilde{\Gamma}(s, t) = \Gamma(s/L, t)$:
	\[C_\gamma - \id_x = \int_0^L \int_0^1 C_1(s, t)^{-1} F_\nabla(\partial_s \widetilde{\Gamma}, \partial_t \widetilde{\Gamma}) C_2(s, t) \,dt \,ds.\]
	The estimate now follows by using $\|\partial_t \widetilde{\Gamma}\| \leq C d(x, \gamma(s))$, where we introduce the positive constant $C = \sup_{x \in M} \sup_{|y|_{g_x} < \imath/2}\|d\exp_x(y)\|_{T_xM \to T_{\exp_x(y)}}$.
	
	For the final item, denote by $C_t, C_t'$ the parallel transports along $\gamma|_{[0, t]}$ with the connection $\nabla, \nabla'$, respectively. Then it is straightforward that $\partial_t(C_t^{-1}C'_t) = C_t^{-1}(\nabla- \nabla')(\dot{\gamma}(t)) C_t'$, so
	\[C_\gamma^{-1} C_\gamma' - \id_p = \int_0^L C_t^{-1}(\nabla- \nabla')(\dot{\gamma}(t)) C_t' \,dt.\]
	The required estimate follows.
\end{proof}

We also have the following result to which we will refer to as the \emph{spiral Lemma}:

\begin{lemma}
\label{lemma:spiral}
Let $x_{\star} \in \M$ be a periodic point of period $T_{\star}$ and let $x_0 \in W^s_{\mathrm{loc}}(x_{\star})$. Define $x_n := \varphi_{nT_{\star}}x_0$ and write $q_n :=  C(x_{\star},nT_{\star})^{-1}C_{x_n \rightarrow x_{\star}} C(x_0, n T_{\star})C_{x_{\star} \rightarrow x_0}$. Then:
\[
\rho(x_0) := \lim_{n \rightarrow +\infty} q_n \in \mathrm{U}(\E_{x_{\star}})
\]
exists. Moreover, there exist some uniform constants $C,\theta > 0$ such that
\[
|q_n - \rho(x_0)| \leq C e^{-\theta n}.
\]
\end{lemma}

\begin{center}
\begin{figure}[htbp!]
\includegraphics{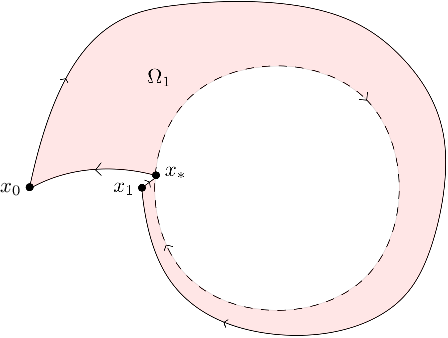}
\caption{The spiral Lemma: the set $\Omega_1$ corresponds to the area over which the integral in the Ambrose-Singer formula is computed for $n=1$.}
\label{fig:spiral}
\end{figure}
\end{center} 

\begin{proof}
Apply the Ambrose-Singer formula as in the first item of the previous Lemma (same notations as in the previous proof):
\[
q_n - \mathbbm{1}_{\E_{x_{\star}}} = \int_0^{nT_{\star}} \int_0^{\ell(t)} C_{\uparrow}^{-1}(s, t) F_\nabla(\partial_t \tau_t(s), \partial_s \tau_t(s)) C_{\rightarrow}(s, t) \, ds \, dt,
\]
where $\tau_t$ is the unit speed shortest geodesic of length $\ell(t)$ from $\varphi_t x_0$ and $\varphi_t x_{\star}$. Observe that this integral converges absolutely as (see \eqref{equation:distance}):
\[
\int_0^{nT_{\star}} \norm{ \int_0^{\ell(t)} C_{\uparrow}^{-1}(s, t) F_\nabla(\partial_t \tau_t(s), \partial_s \tau_t(s)) C_{\rightarrow}(s, t) \, \dd s \, } \dd t \leq \int_0^{n T_{\star}} C \|F_{\nabla}\|_{C^0} e^{-\theta t} \dd t < \infty,
\]
and thus the limit exists. Moreover, it is clear that the convergence is exponential.
\end{proof}

\subsection{Proof of the exact Liv\v{s}ic cocycle Theorem}

\label{section:weak-strong}

\subsubsection{Parry's free monoid}

As we shall see, Parry's free monoid is the key notion to understand the holonomy of unitary connections. Whereas flat connections up to gauge equivalence correspond to representations of the fundamental group up to conjugacy, in the setting of hyperbolic dynamics, we will show that \emph{arbitrary} connections up to cocycle equivalence correspond to representations of Parry's free monoid. 
Recall from \S\ref{sssection:homoclinic} that $x_{\star} \in \M$ is a periodic point of period $T_{\star}$. Let $\mathbf{G}$ be the free monoid generated by $\mc{H}$ (homoclinic orbits to $x_{\star}$), namely the formal set of words
\[
\mathbf{G} := \left\{\gamma_1^{m_1} ... \gamma_k^{m_k} ~|~k \in \N, m_1,...,m_k \in \N_0, \gamma_1,...,\gamma_k \in \mc{H}\right\},
\]
endowed with the obvious monoid structure. The empty word corresponds to the identity element denoted by $\mathbf{1}_{\mathbf{G}}$. Note the periodic orbit corresponding to $x_{\star}$ also belongs to the set of homoclinic orbits. We call $\mathbf{G}$ \emph{Parry's free monoid} as the idea (although not written like this) was first introduced in his work \cite{Parry-99} (see also \cite{Schmidt-99} for a related approach). The main result of this paragraph is the following:

\begin{prop}
\label{proposition:representation0}
Let $\nabla^{\E}$ be a unitary connection on the Hermitian vector bundle $\E \rightarrow \M$. Then $\nabla^{\E}$ induces a representation
\[
\rho : \mathbf{G} \rightarrow \mathrm{U}(\E_{x_{\star}}).
\]
\end{prop}

Formally, this proposition could have also been stated as a definition.

\begin{proof}
Since $\mathbf{G}$ is a free monoid, it suffices to define $\rho$ on the set of generators of $\mathbf{G}$, namely for all homoclinic orbits $\gamma \in \mc{H}$. For the neutral element we  set $\rho(\mathbf{1}_{\mathbf{G}}) = \mathbbm{1}_{\E_{x_{\star}}}$. For the periodic orbit $\gamma_{\star}$ of $x_{\star}$, we set $\rho(\gamma_{\star}) := C(x_{\star},T_{\star})$.

Let $\gamma \in \mc{H}$ (and $\gamma \neq \gamma_{\star}$) and consider a parametrization $\R \ni t \mapsto \gamma(t)$. Following the notations of \S\ref{sssection:homoclinic}, we let $x_n^\pm := \gamma(A_\pm \pm nT_{\star}), x_n^+ = \varphi_{T_n}(x_n^-)$ for some $T_n = A_+-A_-+2n T_{\star}$, where $T_\gamma := A_+ - A_-$ (length of the trunk), and the points $(x_n^\pm)_{n \in \N}$ converge exponentially fast to $x_{\star}$ as $n \rightarrow \infty$. As we shall see, there is a small technical issue coming from the fact that $C(x_{\star},T_{\star})$ is not trivial and this can be overcome by considering a subsequence $k_n \rightarrow \infty$ such that\footnote{For any compact metric group $G$, if $g \in G$, there exists a sequence $k_n \in \N$ such that $g^{k_n} \rightarrow_{n \rightarrow \infty} \mathbf{1}_G$.}
\begin{equation}\label{eq:k_n}
	C(x_{\star},T_{\star})^{k_n} \rightarrow \mathbbm{1}_{\E_{x_{\star}}}, \quad n \to \infty.\end{equation}

For $n,m \in \N$, we define $\rho_{m, n}(\gamma) \in \mathrm{U}(\E_{x_{\star}})$ as follows: 
\begin{equation}\label{eq:rhomn}
\rho_{m,n}(\gamma) :=  C_{x_{k_m}^+ \rightarrow x_{\star}} C(x_0^+, k_m T_{\star}) C(x_0^-, T_\gamma)C(x_{k_n}^-,k_nT_{\star}) C_{x_{\star} \rightarrow x_{k_n}^-},
\end{equation}
and we will write $\rho_{n}(\gamma) := \rho_{n,n}(\gamma)$.

\begin{lemma}
\label{lemma:convergence}
There exists $\rho(\gamma) \in \mathrm{U}(\E_{x_{\star}})$ such that:
\[
\rho_{m,n}(\gamma) \rightarrow_{n,m \rightarrow \infty} \rho(\gamma),
\]
and $\rho(\gamma)$ does not depend in which sense the limit in $n,m$ is taken. 
\end{lemma}

\begin{proof}
We have by construction:
\begin{equation}
\label{equation:cv0}
\begin{split}
\rho_{m,n}(\gamma) &   = C_{x_{k_m}^+ \rightarrow x_{\star}} C(x^+_0, k_m T_{\star})C(x^-_0, T_\gamma)C(x_{k_n}^-,k_n T_{\star}) C_{x_{\star} \rightarrow x_{k_n}^-}\\
& = \left[C_{x_{k_m}^+ \rightarrow x_{\star}} C(x^+_0, k_m T_{\star}) C_{x_{\star} \rightarrow x^+_0} C(x_{\star},k_mT_{\star})^{-1}\right] \\
& \hspace{2cm}  \times C(x_{\star},k_mT_{\star}) C_{x^+_0 \rightarrow x_{\star}} C(x^-_0, T_\gamma)C_{x_{\star} \rightarrow x^-_0} C(x_{\star},k_nT_{\star}) \\
& \hspace{2cm} \times \left[C(x_{\star},k_nT_{\star})^{-1}C_{x^-_0 \rightarrow x_{\star}}C(x_{k_n}^-,k_n T_{\star}) C_{x_{\star} \rightarrow x_{k_n}^-}\right],
\end{split}
\end{equation}
where $T_\gamma$ is independent of $n$ (trunk of $\gamma$). For the middle term we have by \eqref{eq:k_n}:
\[
C(x_{\star},k_mT_{\star}) C_{x^+_0 \rightarrow x_{\star}} C(x^-_0, T_\gamma)C_{x_{\star} \rightarrow x^-_0} C(x_{\star},k_nT_{\star}) = C_{x^+_0 \rightarrow x_{\star}} C(x^-_0, T_\gamma)C_{x_{\star} \rightarrow x^-_0} + o(1),
\]
as $n,m$ go to $+\infty$. Moreover the convergence of the terms between brackets follow from the spiral Lemma \ref{lemma:spiral} (the convergence is exponentially fast).
\end{proof}

This concludes the proof.

\end{proof}

\begin{remark}
\rm
For $\gamma \in \mc{H}$, \eqref{equation:cv0} shows that $\rho(\gamma)$ does not depend on the choice of subsequence $(k_n)_{n \in \N}$ as long as it satisfies $C(x_{\star},T_{\star})^{k_n} \rightarrow \mathbbm{1}$. However, $\rho(\gamma)$ does depend on the choice of trunk $[x_0^-x_0^+]$ for $\gamma$ and another choice of trunk produces a $\rho'(\gamma)$ which differs from $\rho(\gamma)$ by:
\begin{equation}
\label{equation:differs}
\rho'(\gamma) = C(x_{\star},T_{\star})^{m_L(\gamma)} \rho(\gamma) C(x_{\star},T_{\star})^{m_R(\gamma)},
\end{equation}
where $m_L(\gamma), m_R(\gamma) \in \Z$. 
\end{remark}

\subsubsection{Conjugate representations}

We introduce the submonoid $\mathbf{G}^* := \mathbf{G} \setminus \left\{ \gamma_{\star}^k, k \geq 1\right\}$ that is $\mathbf{G}$ minus powers of $\gamma_{\star}$. Recall that the character of a representation $\rho$ is defined by $\chi_\rho(\bullet) := \Tr(\rho(\bullet))$. This paragraph is devoted to proving the following:

\begin{prop}
\label{proposition:representation}
Let $\nabla^{\E_{1,2}}$ be two unitary connections on the Hermitian vector bundles $\E_1,\E_2 \rightarrow \M$. Assume that the connections have trace-equivalent holonomies in the sense of Definition \ref{definition:equivalence}. Then, the induced representations $\rho_{1,2} : \mathbf{G}^* \rightarrow \mathrm{U}({\E_{1,2}}_{x_{\star}})$ have the same character. In particular, this implies that they are isomorphic, i.e. there exists $p_{\star} \in \mathrm{U}({\E_2}_{x_{\star}}, {\E_1}_{x_{\star}})$ such that: 
\begin{equation}\label{eq:repconj}
\forall \gamma \in \mathbf{G}, ~~ \rho_1(\gamma) = p_{\star} \rho_2(\gamma) p_{\star}^{-1}.
\end{equation}
\end{prop}

Following Lemma \ref{lemma:convergence}, we consider a subsequence $(k_n)_{n \in \N}$ such that $C_{1,2}(x_{\star},T_{\star})^{k_n} \rightarrow \mathbbm{1}$. 

\begin{proof}
Once we know that the representations have the same character, the conclusion is a straightforward consequence of a general fact of representation theory, see \cite[Corollary 3.8]{Lang-02}. 
For the sake of simplicity, we take $\gamma = \gamma_1 \cdot \gamma_2$, where $\gamma_{1,2} \in \mc{H}$ (and both $\gamma_{1,2}$ cannot be equal to $\gamma_{\star}$ at the same time since the word $\gamma$ is in $\mathbf{G}^*$) but the generalization to longer words is straightforward as we shall see and words of length $1$ are also handled similarly (one does not even need to concatenate orbits in this case). The empty word (corresponding to the identity element in $\mathbf{G}^*)$ will also be dealt with separately. This proposition is based on the shadowing Theorem \ref{theorem:shadowing} and the fact that one can concatenate orbits. But we will have to be careful to produce periodic orbits which are primitive.

We have by Lemma \ref{lemma:convergence}:
\[
\rho_1(\gamma) = \rho_1(\gamma_1) \rho_1(\gamma_2) = \rho_{1 ; n,N}(\gamma_1)\rho_{1 ; n,n}(\gamma_2) + o(1), \quad n \to \infty,
\]
where we use the convention $\rho_{i;a,b}$ to denote the expression in \eqref{eq:rhomn} with respect to $\nabla^{\E_i}$, for $i=1,2$. The term $N = N(n) \geq n$ will ensure that a certain orbit is primitive as we shall see below. Let $x_{n}^\pm(i)$ be the points on the orbit $\gamma_i$ that are exponentially close to $x_{\star}$, given by \S\ref{sssection:homoclinic}. Consider the concatenation of the orbits $S :=  [x_{k_{N}}^-(1) x_{k_{n}}^+(1)] \cup  [x_{k_n}^-(2) x_{k_n}^+(2)]$. Note that the starting points and endpoints of these segments are at distance at most $\mc{O}(e^{-\theta k_n})$. Thus by the shadowing Theorem \ref{theorem:shadowing}, there exists a genuine periodic orbit $\widetilde{\gamma_n}$ and a point $y_n \in \widetilde{\gamma_n}$ (of period $T'_n$) which $\mc{O}(e^{-\theta k_n})$-shadows the concatenation $S$ (here, if we have a longer word of length $k$, it suffices to apply the shadowing Theorem \ref{theorem:shadowing} with $k$ segments). 

We claim that $\widetilde{\gamma_n}$ is primitive for all $N$ large enough. Indeed, observe that $\widetilde{\gamma_n}$ can be decomposed into the following six subsegments:

\begin{center}
\begin{figure}[htbp!]
\includegraphics{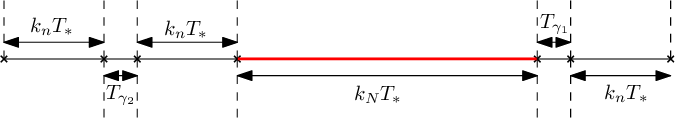}
\caption{The orbit $\widetilde{\gamma_n}$ is made of six segments: in the first segment (of length $k_{n}T_{\star}$), it shadows the first portion $[x_{k_{n}}^-(2) x_{0}^-(2)]$ which wraps around $\gamma_{\star}$; in the second (of length $T_{\gamma_2}$), it shadows the trunk $[x_{0}^-(2) x_{0}^+(2)]$, in the third (of length $k_{n}T_{\star}$), it shadows the last portion of $[x_{0}^+(2) x_{k_{n}}^+(2)]$ which also wraps around $\gamma_{\star}$; then this process is repeated but for the second orbit $\gamma_1$.}
\label{fig:primitive}
\end{figure}
\end{center} 

Moreover, the total length of $\widetilde{\gamma_n}$ is
\[
T'_n = T_{\gamma_2} + 2 k_n T_{\star} + T_{\gamma_1} + (k_{N}+k_n) T_{\star} + \mc{O}(e^{-\theta k_n}).
\]
Take $x \in \gamma_1 \cup \gamma_2$ with $x \not \in \gamma_{\star}$, and consider a small $\varepsilon > 0$ such that $d(x, \gamma_{\star}) > 3\varepsilon$. Let $n$ large enough so that for all $m \geq n$ we have the tail $[x_m^-(1) x_n^-(1)] \subset B(\gamma_{\star}, \varepsilon)$, $\widetilde{\gamma_n}$ satisfies $d(\widetilde{\gamma_n}, x) < \varepsilon$ and finally, so that the shadowing factor of Theorem \ref{theorem:shadowing} satisfies $\mc{O}(e^{-\theta k_n}) < \varepsilon$. Pick $N \geq n$ such that $(k_{N} - k_n) T_{\star} > T'_n/2$. We argue by contradiction and assume that $\widetilde{\gamma_n} = \gamma_0^k$ for some $k \geq 2$ and $\gamma_0 \in \mc{G}^\sharp$, a primitive orbit.

This implies that there is a copy of $\gamma_0$ in the central red segment of Figure \ref{fig:primitive} which $\mc{O}(e^{-\theta k_n})$-shadows the orbit of $x_N^-(1)$ and this forces $\widetilde{\gamma_n} \subset B(\gamma_{\star}, 2\varepsilon)$. Thus $d(\widetilde{\gamma_n}, x) > \varepsilon$, which is a contradiction.

By the first and second items of Lemma \ref{lemma:ASgeometry}, we have:
\[
\rho_{1;n,N}(\gamma_1)\rho_{1;n,n}(\gamma_2) = C_{1,y_n \rightarrow x_{\star}} C_1(y_n, T'_n)C_{1,y_n \rightarrow x_{\star}}^{-1} + \mc{O}(e^{-\theta k_n}).
\]
By assumption, we have $\Tr(C_1(y_n, T'_n)) = \Tr(C_2(y_n,T'_n))$. This yields:
\[
\begin{split}
\Tr(\rho_1(\gamma)) & = \Tr(C_{1,y_n \rightarrow x_{\star}} C_1(y_n, T'_n)C_{1,y_n \rightarrow x_{\star}}^{-1} ) + o(1) \\
& = \Tr(C_1(y_n, T'_n)) + o(1) \\
& = \Tr(C_2(y_n, T'_n)) + o(1)  = \Tr(\rho_2(\gamma)) + o(1).
\end{split}
\]
Taking the limit as $n \rightarrow \infty$, we obtain the claimed result about characters for all non-empty words $\gamma \in \mathbf{G}^*$. 

It remains to deal with the empty word. For that, take any $\gamma \in \mathbf{G}^*$, and consider $n_i \in \N$, a subsequence such that $\rho_1(\gamma)^{n_i} \rightarrow \mathbbm{1}_{{\E_1}_{x_{\star}}}$ and $\rho_2(\gamma)^{n_i} \rightarrow \mathbbm{1}_{{\E_2}_{x_{\star}}}$. Then:
\[
\Tr\left(\rho_1(\gamma)^{n_i}\right) =\Tr\left(\rho_2(\gamma)^{n_i}\right),
\]
and taking the limit as $i \rightarrow \infty$ gives
\[
\Tr(\rho_1(\mathbf{1}_{\mathbf{G}^*}))=\rk(\E_1)=\rk(\E_2) =\Tr(\rho_2(\mathbf{1}_{\mathbf{G}^*})).
\]
By the mentioned \cite[Corollary 3.8]{Lang-02}, there is a $p_{\star}$ satisfying \eqref{eq:repconj} for $\gamma \in \mathbf{G}^*$.

It is now straightforward to show \eqref{eq:repconj} for all $\gamma \in \mathbf{G}$. Applying \eqref{eq:repconj} with $\gamma_{\star} \gamma \in \mathbf{G}^*$, where $\gamma \in \mc{H}\setminus \{\gamma_{\star}\}$ is arbitrary, we get that:
\[
\begin{split}
\rho_1(\gamma_{\star} \gamma)&  = \rho_1(\gamma_{\star}) \rho_1(\gamma) \\
& = p_{\star} \rho_2(\gamma_{\star} \gamma) p_{\star}^{-1} = p_{\star} \rho_2(\gamma_{\star}) p_{\star}^{-1} p_{\star} \rho_2(\gamma) p_{\star}^{-1}.
\end{split}
\]
Since $\rho_1(\gamma) = p_{\star} \rho_2(\gamma) p_{\star}^{-1}$ (because $\gamma \in \mathbf{G}^*$), we get that $\rho_1(\gamma_{\star}) = p_{\star} \rho_2(\gamma_{\star}) p_{\star}^{-1}$, that is $C_1(x_{\star},T_{\star}) = p_{\star} C_2(x_{\star},T_{\star}) p_{\star}^{-1}$ or equivalently $P(x_{\star},T_{\star})p_{\star} = p_{\star}$ (where $P$ denotes the parallel transport along the flowlines of $(\varphi_t)_{t \in \R}$ with respect to the mixed connection $\nabla^{\mathrm{Hom}(\nabla^{\E_2},\nabla^{\E_1})}_X$, as in \eqref{equation:tp-mixed}), concluding the proof.
\end{proof}

\begin{remark}
\rm
Although the representations $\rho_{1,2}$ depend on choices (namely on a choice of trunk for each homoclinic orbit $\gamma \in \mc{H}$), the map $p_{\star} \in \mathrm{U}(\E_{x_{\star}})$ does not. Indeed, taking $\rho_{1,2}'$ two other representations (for some other choices of trunks), one gets by \eqref{equation:differs}:
\[
\begin{split}
\rho'_1(\gamma) & = C_1(x_{\star},T_{\star})^{m_L(\gamma)} \rho_1(\gamma) C_1(x_{\star},T_{\star})^{m_R(\gamma)} \\
&  = C_1(x_{\star},T_{\star})^{m_L(\gamma)} p_{\star} \rho_2(\gamma) p_{\star}^{-1} C_1(x_{\star},T_{\star})^{m_R(\gamma)} \\
& = C_1(x_{\star},T_{\star})^{m_L(\gamma)} p_{\star} C_2(x_{\star},T_{\star})^{-m_L(\gamma)} \rho'_2(\gamma) C_2(x_{\star},T_{\star})^{-m_R(\gamma)} p_{\star}^{-1} C_1(x_{\star},T_{\star})^{m_R(\gamma)} \\
& = \left(P(x_{\star},m_L(\gamma)T_{\star})p_{\star}\right) \rho'_2(\gamma) \left(P(x_{\star},m_R(\gamma)T_{\star})p_{\star}\right)^{-1}  = p_{\star} \rho'_2(\gamma) p_{\star}^{-1},
\end{split}
\]
since $P(x_{\star},T_{\star})p_{\star}=p_{\star}$, that is $p_{\star}$ also conjugates the representations $\rho'_{1,2}$. Note that the map $p_{\star}$ given by \cite[Corollary 3.8]{Lang-02} is generally not unique. Nevertheless, if the representation is irreducible, it is unique modulo the trivial $\Ss^{1}$-action.
\end{remark}

\subsubsection{Proof of Theorem \ref{theorem:weak}}

We can now complete the proof of Theorem \ref{theorem:weak}.

\begin{proof}[Proof of Theorem \ref{theorem:weak}]
Let $\mc{W}$ be the set of all points belonging to homoclinic orbits in $\mc{H}$. By Lemma \ref{lemma:dense}, $\mc{W}$ is dense in $\M$ and we are going to define the map $p$ (which will conjugate the cocycles) on $\mc{W}$ and then show that $p$ is Lipschitz-continuous on $\mc{W}$ so that it extends naturally to $\M$. The map $p$ is defined as the parallel transport of $p_{\star}$ with respect to the mixed connection.

By assumptions, we have $C_i(x_{\star},T_{\star})^{k_n} \rightarrow \mathbbm{1}_{\E_*}$, and thus $P(x_{\star},T_{\star})^{k_n} \rightarrow \mathbbm{1}_{\mathrm{Hom}({\E_2}_{x_{\star}},{\E_1}_{x_{\star}} )}$ (where we use the notation $\mathbbm{1}_{\mathrm{Hom}({\E_2}_{x_{\star}},{\E_1}_{x_{\star}} )}(q) = q$ for $q \in \mathrm{Hom}({\E_2}_{x_{\star}},{\E_1}_{x_{\star}})$). Consider a point $x \in \gamma$, where $\gamma \in \mc{H}$ is a homoclinic orbit and also consider a parametrization of $\gamma$ as in \S\ref{sssection:homoclinic}. For $n \in \N$ large enough, consider the point $x_n^- \in \gamma$ (which is exponentially close to $x_{\star}$) and write $x = \varphi_{T^-_n}(x_n^-)$ for some $T^-_n > 0$. Define:
\[
p^-_n(x) := P(x_{k_n}^-,T^-_{k_n})P_{x_{\star} \rightarrow x_{k_n}^-} p_{\star} \in \mathrm{U}({\E_2}_{x},{\E_1}_{x}).
\]

\begin{lemma}
\label{lemma:p-minus}
Fix $\gamma \in \mc{H}$. Then for all $x \in \gamma$, there exists $p_-(x) \in \mathrm{U}({\E_2}_{x},{\E_1}_{x} )$ such that $p^-_n(x) \rightarrow_{n \rightarrow \infty} p_-(x)$. There exists $C > 0$ such that: $|p^-_n(x)-p_-(x)| \leq C/n$. Moreover, $\nabla^{\mathrm{Hom}(\nabla^{\E_2},\nabla^{\E_1})}_X p_- = 0$ on $\gamma$. 
\end{lemma}

In particular, this shows that $p_-$ is smooth in restriction to $\gamma$ as $\nabla^{\mathrm{Hom}(\nabla^{\E_2},\nabla^{\E_1})}_X$ is elliptic on $\gamma$.

\begin{proof}
By construction, the differential equation is clearly satisfied if the limit exists. Moreover, we have for some time $T_0$ (independent of $n$, $T^-_{k_n} = T_0 + k_n T_{\star}$):
\[
\begin{split}
p^-_n(x) & = P(x_{k_n}^-,T^-_{k_n})P_{x_{\star} \rightarrow x_{k_n}^-} p_{\star}\\
&  = P(x_0^-, T_0)P(x_{k_n}^-,k_n T_{\star}) P_{x_{\star} \rightarrow x_{k_n}^-} p_{\star} \\
& = P(x_0^-, T_0) P_{x_{\star} \rightarrow x_0^-} P(x_{\star},T_{\star})^{k_n} \left[P(x_{\star},T_{\star})^{-k_n}P_{x_0^- \rightarrow x_{\star}} P(x_{k_n}^-,k_n T_{\star}) P_{x_{\star} \rightarrow x_{k_n}^-} p_{\star}\right].
\end{split}
\]
By assumption, the term outside the bracket converges as $n \rightarrow \infty$ and the term between brackets converges by the spiral Lemma \ref{lemma:spiral}.
\end{proof}

We now claim the following:

\begin{lemma}
\label{lemma:stable}
There exists a uniform constant $C > 0$ such that the following holds. Assume that $x$ and $z$ belong to two homoclinic orbits in $\mc{H}$ and $z \in W^u_{\mathrm{loc}}(x)$. Then:
\[
\|P_{x \rightarrow z}p_-(x) - p_-(z)\| \leq C d(x,z).
\]
\end{lemma}

By the previous proofs, the point $x$ is associated to points $x^-_n$ on the homoclinic orbit and we will use the same notations for the point $z$ associated to the points $z^-_n$.

\begin{proof}
There is here a slight subtlety coming from the fact that the parametrizations of the homoclinic orbits $\gamma$ were chosen in a non-canonical way (via a choice of $A_\pm$). In particular, it is not true that the flowlines of $z_{k_n}^-$ and $x_{k_n}^-$ shadow each other; in other words, we might not have $T^-_{k_n}(z) = T^-_{k_n}(x)$ but we rather have $T^-_{k_n}(z) = T^-_{k_n}(x) + mT_{\star}$ for some integer $m \in \Z$ depending on both $x$ and $z$.

We have:
\[
\begin{split}
& \|P_{x \rightarrow z}p_-(x) - p_-(z)\| \\
&  = \|P_{x \rightarrow z} p^-_n(x) - p^-_n(z)\| + o(1) \\
& = \|P_{x \rightarrow z} P(x_{k_n}^-,T^-_{k_n}(x))P_{x_{\star} \rightarrow x_{k_n}^-} p_{\star} - P(z_{k_n}^-,T^-_{k_n}(z))P_{x_{\star} \rightarrow z_{k_n}^-} p_{\star}\| + o(1) \\
& \leq C \| P_{z_{k_n}^- \rightarrow x_{\star}}P(z_{k_n}^-,T^-_{k_n}(z))^{-1} P_{x \rightarrow z} P(x_{k_n}^-,T^-_{k_n}(x))P_{x_{\star} \rightarrow x_{k_n}^-}p_{\star} - p_{\star}\| + o(1) \\
& \leq C \| P_{z_{k_n}^- \rightarrow x_{\star}}P(z_{k_n},mT_{\star})^{-1}P_{x_{\star} \rightarrow z_{k_n-m}^- } \\
& \hspace{1.2cm} \times \left[P_{z_{k_n-m}^- \rightarrow x_{\star}}P(z_{k_n-m}^-,T^-_{k_n}(z)-mT_{\star})^{-1} P_{x \rightarrow z} P(x_{k_n}^-,T^-_{k_n}(x))P_{x_{\star} \rightarrow x_{k_n}^-}\right]p_{\star} - p_{\star}\| + o(1).
\end{split}
\]
Applying the first item of Lemma \ref{lemma:ASgeometry}, we have that:
\[
\|P_{z_{k_n-m}^- \rightarrow x_{\star}}P(z_{k_n-m}^-,T^-_{k_n}(z)-mT_{\star})^{-1} P_{x \rightarrow z} P(x_{k_n}^-,T^-_{k_n}(x))P_{x_{\star} \rightarrow x_{k_n}^-} - \mathbbm{1}_{\End(\E_{x_{\star}})}\| \leq Cd(x,z),
\]
where the constant $C > 0$ is uniform in $n$. Moreover, observe that
\[
\lim_{n \rightarrow \infty} P_{z_{k_n}^- \rightarrow x_{\star}}P(z_{k_n},mT_{\star})^{-1}P_{x_{\star} \rightarrow z_{k_n-m}^- } = P(x_{\star},mT_{\star})^{-1}.
\]

Hence:
\[
 \|P_{x \rightarrow z}p_-(x) - p_-(z)\| \leq C\left(\| P(x_{\star},mT_{\star})^{-1}p_{\star} - p_{\star}\|  + d(x,z) + o(1)\right).
 \]
Using that $P(x_{\star},T_{\star})p_{\star}=p_{\star}$, we get that the first term on the right-hand side vanishes. Taking the limit as $n \rightarrow +\infty$, we obtain the announced result.
\end{proof}

Note that we could have done the same construction ``in the future" by considering instead:
\[
p_+(x) = \lim_{n \rightarrow \infty} P(x,T^+_{k_n})^{-1} P_{x_{\star} \rightarrow x_{k_n}^+} p_{\star} \in \mathrm{U}({\E_2}_{x},{\E_1}_{x}),
\]
where $x_n^+ := \varphi_{T^+_n}(x)$ is exponentially closed to $x_{\star}$ as in \S\ref{sssection:homoclinic}. A similar statement as Lemma \ref{lemma:stable} holds with the unstable manifold being replaced by the stable one. We have:

\begin{lemma}
\label{lemma:equality}
For all $x \in \mc{W}$, $p_-(x) = p_+(x)$.
\end{lemma}

\begin{proof}
This follows from Proposition \ref{proposition:representation}. Indeed, we have:
\[
\begin{split}
\|p_-(x)-p_+(x)\| & = \|P(x_{k_n}^-,T^-_{k_n})P_{x_{\star} \rightarrow x_{k_n}^-} p_{\star} - P(x,T^+_{k_n})^{-1} P_{x_{\star} \rightarrow x_{k_n}^+} p_{\star}\| + o(1) \\
& \leq C \| P_{x_{k_n}^+ \rightarrow x_{\star} }P(x,T^+_{k_n})P(x_{k_n}^-,T^-_{k_n})P_{x_{\star} \rightarrow x_{k_n}^-}p_{\star} - p_{\star}\| + o(1) \\
& \leq C \| P_{x_{k_n}^+ \rightarrow x_{\star} }P(x_{k_n}^-,T_{k_n})P_{x_{\star} \rightarrow x_{k_n}^-}p_{\star} - p_{\star}\| + o(1),\\
\end{split}
\]
where $T_n := T_n^- + T_n^+$. Observe that:
\[
\begin{split}
& P_{x_{k_n}^+ \rightarrow x_{\star} }P(x_{k_n}^-,T_{k_n})P_{x_{\star} \rightarrow x_{k_n}^-}p_{\star} \\
&  \hspace{2cm}= C_{1,x_{k_n}^+ \rightarrow x_{\star}} C_1(x_{k_n}^-,T_{k_n}) C_{1,x_{\star} \rightarrow x_{k_n}^-} p_{\star} \left(C_{2,x_{k_n}^+ \rightarrow x_{\star}} C_2(x_{k_n}^-,T_{k_n}) C_{2,x_{\star} \rightarrow x_{k_n}^-} \right)^{-1} \\
&  \hspace{2cm} = \rho_{1,n}(\gamma) p_{\star} \rho_{2,n}(\gamma)^{-1}  = \rho_1(\gamma)p_{\star} \rho_2(\gamma)^{-1} + o(1) = p_{\star} + o(1),
\end{split}
\]
by Proposition \ref{proposition:representation}. Hence $\|p_-(x)-p_+(x)\| = o(1)$, that is $p_-(x)=p_+(x)$.
\end{proof}

We can now prove the following lemma:

\begin{lemma}
\label{lemma:lipschitz}
The map $p_-$ is Lipschitz-continuous.
\end{lemma}

\begin{proof}
Consider $x, y \in \mc{W}$ which are close enough. Let $z := \llbracket x,y \rrbracket \in W^{wu}_{\mathrm{loc}}(x) \cap W^{s}_{\mathrm{loc}}(y)$ and define $\tau$ such that $\varphi_\tau(z) \in W^u_{\mathrm{loc}}(x)$. Note that $|\tau| \leq Cd(x,y)$ for some uniform constant $C > 0$; also observe that the point $z$ is homoclinic to the periodic orbit $x_{\star}$. We have:
\[
\begin{split}
& \|p_-(x)-P_{y \rightarrow x}p_-(y)\| \\
& \leq \|p_-(x)-P_{z \rightarrow x}p_-(z)\| + \|p_-(z)-p_+(z)\| + \|P_{z \rightarrow x}p_+(z)-P_{y \rightarrow x}p_+(y)\| + \|p_+(y)-p_-(y)\| \\
& \leq \|p_-(x)-P_{z \rightarrow x}p_-(z)\| + \|P_{x \rightarrow y}P_{z \rightarrow x}p_+(z)-p_+(y)\| \\
& \leq  \|p_-(x)-P_{\varphi_{\tau}(z) \rightarrow x} p_-(\varphi_{\tau}(z))\| + \|P_{\varphi_{\tau}(z) \rightarrow x} p_-(\varphi_{\tau}(z))- P_{z \rightarrow x}p_-(z)\| + \|P_{x \rightarrow y}P_{z \rightarrow x}p_+(z)-p_+(y)\| 
\end{split}
\]
where the terms disappear between the second and third line by Lemma \ref{lemma:equality}. By Lemma \ref{lemma:stable}, the first term is controlled by:
\[
\|p_-(x)-P_{\varphi_{\tau}(z) \rightarrow x} p_-(\varphi_{\tau}(z))\| \leq Cd(x,\varphi_{\tau}(z)) \leq Cd(x,y).
\]
As to the second term, using the second item of Lemma \ref{lemma:ASgeometry}, we have:
\[
\|P_{\varphi_{\tau}(z) \rightarrow x} p_-(\varphi_{\tau}(z))- P_{z \rightarrow x}p_-(z)\| = \|P_{x \rightarrow z} P_{\varphi_{\tau}(z) \rightarrow x} P(z,\tau)p_-(z)- p_-(z)\| \leq Cd(x,y).
\]
Eventually, the last term $\|P_{x \rightarrow y}P_{z \rightarrow x}p_+(z)-p_+(y)\|$ is controlled similarly to the first term by applying Lemma \ref{lemma:stable} (but with the stable manifold instead of unstable).
\end{proof}

As $\mc{W}$ is dense, $p_-$ extends to a Lipschitz-continuous map on $\M$ which satisfies the equation $\nabla^{\mathrm{Hom}(\nabla^{\E_2},\nabla^{\E_1})}_X p_- = 0$ and by \cite[Theorem 4.1]{Bonthonneau-Lefeuvre-20}, this implies that $p_-$ is smooth. This concludes the proof of the Theorem.

\end{proof}

\subsubsection{Proof of the geometric properties}

We now prove Proposition \ref{proposition:opaque}.

\begin{proof}[Proof of Proposition \ref{proposition:opaque}]
The equivalence between (1) and (2) can be found in \cite[Section 5]{Cekic-Lefeuvre-20}. If $\mc{F} \subset \mc{E}$ is a non-trivial subbundle that is invariant by parallel transport along the flowlines of $(\varphi_t)_{t \in \R}$, it is clear that $\rho$ will leave the space $\mc{F}_{x_{\star}}$ invariant and thus is not irreducible. Conversely, if $\rho$ is not irreducible, then there exists a non-trivial $\mc{F}_{x_{\star}} \subset \E_{x_{\star}}$ preserved by $\rho$. Let $\pi_{\star} : \mc{E}_{x_{\star}} \rightarrow \mc{F}_{x_{\star}}$ be the orthogonal projection. For $x$ on a homoclinic orbit, define $\pi(x) : \E_x \rightarrow \E_x$ similarly to $p_-$ in Lemma \ref{lemma:p-minus} by parallel transport of the section $\pi_{\star}$ with respect to the connection $\nabla^{\End(\E)}$. Following the previous proofs (we only use $\rho \pi_\star = \pi_\star \rho$), one shows that $\pi$ extends to a Lipschitz-continuous section on homoclinic orbits which satisfies $\pi^2 = \pi$ and $\nabla^{\End}_X \pi = 0$. By \cite[Theorem 4.1]{Bonthonneau-Lefeuvre-20}, $\pi$ extends to a smooth section i.e. $\pi \in C^\infty(\M,\End(\E))$. Moreover, $\pi(x_{\star})=\pi_{\star}$, hence $\pi$ is the projection onto a non-trivial subbundle $\mc{F} \subset \mc{E}$.
\end{proof}

We now prove Theorem \ref{theorem:iso}.

\begin{proof}[Proof of Theorem \ref{theorem:iso}]
The linear map $\Phi : \mathbf{R}' \rightarrow \ker \nabla^{\End(\E)}_X|_{C^\infty(\M,\End(\E))}$ is defined in the following way. Consider $u_{\star} \in \mathbf{R}'$ and define, as in Lemma \ref{lemma:p-minus}, for $x$ on a homoclinic orbit, $u_-(x)$ as the parallel transport of $u_{\star}$ from $x_{\star}$ to $x$ along the orbit (with respect to the endomorphism connection $\nabla^{\End(\E)}$). Similarly, one can define $u_+(x)$ by parallel transport from the future. The fact that $u_{\star} \in \mathbf{R}'$ is then used in the following observation (see Lemma \ref{lemma:equality}):
\[
\|u_-(x)-u_+(x)\| = \|\rho(\gamma)u_{\star} \rho(\gamma)^{-1} - u_\star\| = 0.
\]
(Note that $\rho(\gamma)u_{\star} \rho(\gamma)^{-1}$ corresponds formally to the parallel transport of $u_{\star}$ with respect to $\nabla^{\End(\E)}$ from $x_{\star}$ to $x_{\star}$ along the homoclinic orbit $\gamma$.) Hence, following Lemma \ref{lemma:lipschitz}, we get that $u_-$ is Lipschitz-continuous and satisfies $\nabla^{\End(\E)}_X u_- = 0$. By \cite[Theorem 4.1]{Bonthonneau-Lefeuvre-20}, it is smooth and we set $u_- := \Phi(u_{\star}) \in \ker \nabla^{\End(\E)}_X|_{C^\infty(\M,\End(\E))}$.

Also observe that this construction is done by using parallel transport with respect to the unitary connection $\nabla^{\End(\E)}$. As a consequence, if $u_\star, u_\star' \in \mathbf{R}$ are orthogonal (i.e. $\Tr(u_\star^* u_\star') = 0$), then $\Phi(u_{\star})$ and $\Phi(u'_\star)$ are also pointwise orthogonal. This proves that $\Phi$ is injective.

It now remains to show the surjectivity of $\Phi$. Let $u \in\ker \nabla^{\End(\E)}_X|_{C^\infty(\M,\End(\E))}$. Following \cite[Section 5]{Cekic-Lefeuvre-20}, we can write $u = u_R + i u_I$, where $u_R^*=u_R,u_I^*=u_I$ and $\nabla^{\End(\E)}_X u_R = \nabla^{\End(\E)}_X u_I = 0$. By \cite[Lemma 5.6]{Cekic-Lefeuvre-20}, we can then further decompose $u_R = \sum_{i=1}^p \lambda_i \pi_{\mc{F}_i}$ (and same for $u_I$), where $\lambda_i \in \R$, $p \in \mathbb{N}$ and $\mc{F}_i \subset \E$ is a maximally invariant subbundle of $\E$ (i.e. it does not contain any non-trivial subbundle that is invariant under parallel transport along the flowlines of $(\varphi_t)_{t \in \R}$ with respect to $\nabla^{\E}$), and $\pi_{\mc{F}_i}$ is the orthogonal projection onto $\mc{F}_i$. Setting $(\pi_{\mc{F}_i})_\star := \pi_{\mc{F}_i}(x_\star)$, invariance of $\mc{F}_i$ by parallel transport implies that $\rho(\gamma)(\pi_{\mc{F}_i})_\star = (\pi_{\mc{F}_i})_\star\rho(\gamma)$, for all $\gamma \in \mathbf{G}$, that is $(\pi_{\mc{F}_i})_\star \in \mathbf{R}'$. Moreover, we have $\Phi((\pi_{\mc{F}_i})_\star) = \pi_{\mc{F}_i}$. This proves that both $u_R$ and $u_I$ are in $\mathrm{ran}(\Phi)$. This concludes the proof.
\end{proof}

It remains to prove the results concerning invariant sections:

\begin{proof}[Proof of Lemma \ref{lemma:invariant-section}]
Uniqueness is immediate since $\nabla^{\E}_X u = 0$ implies that
\[
X |u|^2 = \langle \nabla^{\E}_X u, u \rangle = \langle u , \nabla^{\E}_X u \rangle = 0,
\]
that is $|u|$ is constant. Now, given $u_\star$ which is $\mathbf{G}$-invariant, we can define $u_-(x)$ for $x$ on a homoclinic orbit $\gamma$ by parallel transport of $u_\star$ from $x_\star$ to $x$ along $\gamma$ with respect to $\nabla^{\E}$, similarly to Lemma \ref{lemma:p-minus} and to the proof of Theorem \ref{theorem:iso}. We can also define $u_+(x)$ in the same fashion (by parallel transport in the other direction). Then one gets that $\|u_-(x)-u_+(x)\| = \|u_\star - \rho(\gamma)u_\star\|=0$ and the same arguments as before show that $u_-$ extends to a smooth function in the kernel of $\nabla^{\E}_X$.
\end{proof}

\begin{proof}[Proof of Lemma \ref{lemma:rank2}]
This is based on the following:

\begin{lemma}
Assume that for all periodic orbits $\gamma \in \mc{G}$, there exists $u_\gamma \in C^\infty(\gamma,\E|_{\gamma})$ such that $\nabla^{\E}_X u_{\gamma} = 0$. Then for all $g \in \mathbf{G}$, there exists $u_g \in \E_{x_\star}$ such that $\rho(g)u_g = u_g$. 
\end{lemma}

\begin{proof}
Recall that by the construction of Proposition \ref{proposition:representation}, each element $\rho(g) \in \mathrm{U}(\E_{x_\star})$ can be approximated by the holonomy $C_{y_n \to x_\star}C(y_n,T'_n)C_{x_\star \to y_n}$ along a sequence of periodic orbits of points $y_n$ converging to $x_\star$. Now, each $C(y_n,T'_n)$ has $1$ as eigenvalue by assumption and taking the limit as $n \rightarrow \infty$, we deduce that $1$ is an eigenvalue of $\rho(g)$.
\end{proof}

As a consequence, we can write for all $g \in \mathbf{G}$, in a fixed orthonormal basis of $\E_{x_\star}$	:
\[
\rho(g) = \alpha_g \begin{pmatrix} 1 & 0 \\ 0 & s(g) \end{pmatrix} \alpha_g^{-1},
\]
for some $\alpha_g \in \mathrm{U}(\E_{x_\star})$ and $s(g)$ is an $(r-1) \times (r-1)$ matrix. For $\rk(\E)=1$, the Lemma is then a straightforward consequence of Lemma \ref{lemma:invariant-section} since the conjugacy $\alpha_g$ does not appear. For $\rk(\E)=2$, one has the remarkable property that $s(g)$ is \emph{still} a representation of $ \mathbf{G}$ since $\det \rho(g) = s(g) \in \mathrm{U}(1)$. As a consequence, $\rho :  \mathbf{G} \rightarrow \mathrm{U}(\E_{x_\star})$ has the same character as $\rho' :  \mathbf{G} \rightarrow \mathrm{U}(\E_{x_\star})$ defined by:
\[
\rho'(g) :=  \begin{pmatrix} 1 & 0 \\ 0 & s(g) \end{pmatrix}.
\]
By \cite[Corollary 3.8]{Lang-02}, we then conclude that these representations are isomorphic, that is there exists $p_\star \in \mathrm{U}(\E_{x_\star})$ such that $\rho(g) = p_\star \rho'(g) p_\star^{-1}$. If $u_\star' \in \E_{x_\star}$ denotes the vector fixed by $\rho'(\mathbf{G})$, then $u_\star := p_\star u'_\star$ is fixed by $\rho(\mathbf{G})$. We then conclude by Lemma \ref{lemma:invariant-section}.
\end{proof}

\section{Pollicott-Ruelle resonances and local geometry on the moduli space of connections}

\label{section:geometry}

This section is devoted to the study of the moduli space of connections, with the point of view of Pollicott-Ruelle resonances. We will first deal with the opaque case and then outline the main distinctions with the non-opaque case. We consider a Hermitian vector bundle $(\mc{E},\nabla^{\mc{E}})$ endowed with a unitary connection over the Anosov Riemannian manifold $(M,g)$. Recall the notation of \S \ref{section:twisted}: we write $\X = (\pi^*\nabla^{\E})_X$, $\RR_\pm(z) = (\pm \X + z)^{-1}$ for its resolvent and $\RR^\pm_0$, $\Pi_0^\pm$ for the holomorphic parts and the spectral projector at zero, respectively.

\subsection{The Coulomb gauge}

We study the geometry of the space of connections (and of the moduli space of gauge-equivalent connections) in a neighborhood of a given unitary connection $\nabla^{\mc{E}}$ of regularity $C^s_*$ (for $1 < s < \infty$\footnote{It is very likely that the case $s=\infty$ still works. This would require to use the Nash-Moser Theorem.}) such that $\ker(\nabla^{\End(\E)}) = \C \cdot \mathbbm{1}_{\E}$. For the standard differential topology of Banach manifolds, we refer the reader to \cite{Lang-99}. We denote by
\[
\mc{O}_s(\nabla^{\mc{E}}) := \left\{ \nabla^{\mc{E}}  + p^{-1} \nabla^{\mathrm{End}(\E)}p ~|~ p \in C_*^{s+1}(M,\mathrm{U}(\E)),\,\, \|p - \mathbbm{1}\|_{C^{s+1}_*} < 
\delta\right\}
\]
the orbit of gauge-equivalent connections of $C^{s}_*$ regularity, where $\delta > 0$ is small enough so that $\mc{O}_s(\nabla^{\E})$ is a smooth Banach submanifold. We also define the slice at $\nabla^{\E}$ by
\[
\mathcal{S}_s(\nabla^{\E}) := \nabla^{\E} + \ker (\nabla^{\End(\E)})^* \cap \big\{A \in C^{s}(M,T^*M \otimes \mathrm{End}_{\mathrm{sk}}(\E)): \|A\|_{C_*^s} < \delta\big\}.
\]
Note that $\mathbb{S}^1$ acts by multiplication freely and properly on $C_*^s(M, \mathrm{U}(\E))$ and hence we may form the quotient Banach manifold, denoted by $C_*^s(M, \mathrm{U}(\E))/\mathbb{S}^1$, which in particular satisfies
\begin{equation}
	T_{\mathbbm{1}_{\E}} \Big(C_*^s(M, \mathrm{U}(\E))/\mathbb{S}^1\Big) = C_*^s(M, \End_{\mathrm{sk}}(\E))/\big(\mathbb{R}\cdot (i\mathbbm{1}_{\E})\big), 
\end{equation}
where we used the identification of tangent spaces given by the exponential map.
	Next, observe that the map $O: p \mapsto p^*\nabla^{\E}$ is injective modulo the multiplication action of $\mathbb{S}^1$ on $C^{s + 1}_*(M, \mathrm{U}(\E))$ and that it is an immersion at $p = \mathbbm{1}$ with $\dd_{\mathbbm{1}} O (\Gamma) = \nabla^{\End(\E)} \Gamma$. Therefore by equation \eqref{eq:decomposition-tt}, $\mc{O}_s(\nabla^{\E})$ and $\mc{S}_s$ are smooth transverse Banach manifolds with:
	\[T_{\nabla^{\E}} \mc{O}(\nabla^{\End(\E)}) = \ran(\nabla^{\End(\E)}), \quad T_{\nabla^{\End(\E)}} \mc{S}_s = \ker(\nabla^{\End(\E)})^*.\]

We will say that a connection $\nabla_2^{\E}$ is in the \emph{Coulomb gauge} with respect to $\nabla_1^{\E}$ if $(\nabla_1^{\End(\E)})^*(\nabla_2^{\E} - \nabla_1^{\E}) =0$. The following lemma shows that, near $\nabla^{\E}$, we may always put a pair of connections in the Coulomb gauge with respect to each other. It is a slight generalisation of the usual claim (see \cite[Proposition 2.3.4]{Donaldson-Kronheimer-90}).

\begin{lemma}[Coulomb gauge]
\label{lemma:coulomb}
Let $s > 1$. There exists $\varepsilon = \varepsilon(s, \nabla^{\E}) > 0$ and a neighbourhood $\mathbbm{1}_{\E} \in \mc{U} \subset C_*^{s + 1}(M, \mathrm{U}(\E))/\mathbb{S}^1$ such that for any $A_i \in C^{s}(M,T^*M \otimes \mathrm{End}_{\mathrm{sk}}(\E))$ with $\|A_i\|_{C^s_*} < \varepsilon$, after setting $\nabla_i^{\E} = \nabla^{\E} + A_i$  for $i = 1, 2$, there exists a unique $p_{A_1, A_2} \in \mc{U}$ such that $p_{A_1, A_2}^*\nabla_2^{\E} - \nabla_1^{\E} \in \ker (\nabla_1^{\End(\E)})^*$. Furthermore, if $A_i$ are smooth, then $p_{A_1, A_2}$ is smooth.
Moreover, the map 
\[
\big(C^{s}(M,T^*M \otimes \mathrm{End}_{\mathrm{sk}}(\E))\big)^2 \ni (A_1, A_2) \mapsto \phi(A_1, A_2) := p_{A_1, A_2}^*\nabla_2^{\E} \in \mc{S}_s(\nabla_1^{\E}),
\]
is smooth. Setting $\phi(A) := \phi(0, A)$, we have:
 \[
 \dd \phi|_{A = 0} = \pi_{\ker (\nabla^{\End(\E)})^*}.
 \]
\end{lemma}

\begin{center}
\begin{figure}[htbp!]
\includegraphics[scale=0.9]{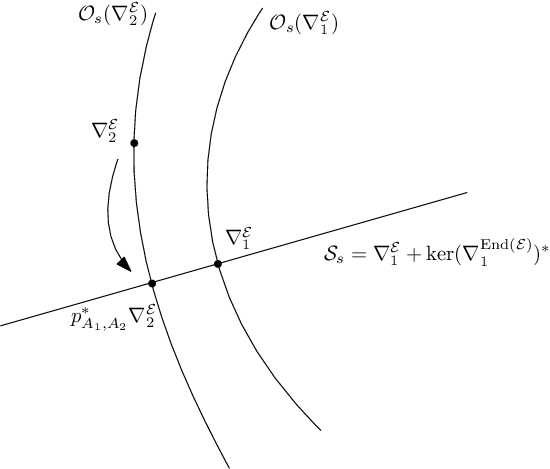}
\caption{A schematic representation of Lemma \ref{lemma:coulomb}.}
\label{fig:coulomb}
\end{figure}
\end{center}

\begin{proof}
Note that the exponential map $\exp: C^{s + 1}_*(M, \End_{\mathrm{sk}}(\E)) \cap \{i \mathbbm{1}_{\E}\}^{\perp_{L^2}} \to C^{s + 1}_*(M, \mathrm{U}(\E))/\mathbb{S}^1$ is well-defined and a local diffeomorphism at zero, so we reduce the claim to finding a neighbourhood $0 \in \mathcal{V} \subset C^{s + 1}_*(M, \End_{\mathrm{sk}}(\E)) \cap \{i\mathbbm{1}_{\E}\}^{\perp_{L^2}}$ and setting $p = p_{A_1, A_2} = \exp(\chi_{A_1, A_2})$ for $\chi = \chi_{A_1, A_2} \in \mc{V}$, that is $\mc{U} = \exp(\mc{V})$. Define the functional
	\begin{align*}
		&F: \big(C^s_*(M, T^*M \otimes \End_{\mathrm{sk}}(\E))\big)^2 \times C^{s + 1}_*(M, \End_{\mathrm{sk}}(\E)) \cap \{i\mathbbm{1}_{\E}\}^{\perp_{L^2}}\to C^{s - 1}_*(M, \End_{\mathrm{sk}}(\E))\\
		&\cap \{i\mathbbm{1}_{\E}\}^{\perp_{L^2}}, F(A_1, A_2, \chi) :=
		(\nabla_1^{\End(\E)})^*\Big(\exp(-\chi) \nabla^{\End(\E)} \exp(\chi) + \exp(-\chi) A_2 \exp(\chi) - A_1\Big).
	\end{align*}
	We have $F$ well-defined, i.e. with values in skew-Hermitian endomorphisms, since $\nabla^{\E}$ is unitary, and integrating by parts $\langle{F(A_1, A_2, \chi), \mathbbm{1}_{\E}}\rangle_{L^2} = 0$; note that $F$ is smooth in its entries. Next, we compute the partial derivative with respect to the $\chi$ variable at $A_1 = A_2 = 0$ and $\chi = 0$:
	\[\dd_{\chi}F(0, 0, 0) (\Gamma) = \partial_t|_{t = 0} F(0, 0, t\Gamma) = (\nabla^{\End(\E)})^* \nabla^{\End(\E)} \Gamma.\]
	This derivative is an isomorphism on 
	\[(\nabla^{\End(\E)})^* \nabla^{\End(\E)}: C^{s + 1}_*(M, \End_{\mathrm{sk}}(\E)) \cap \{i \mathbbm{1}_{\E}\}^{\perp_{L^2}} \to C^{s - 1}_*(M, \End_{\mathrm{sk}}(\E)) \cap \{i\mathbbm{1}_{\E}\}^{\perp_{L^2}},\]
	by the Fredholm property of $(\nabla^{\End(\E)})^* \nabla^{\End(\E)}$ and since $\ker (\nabla^{\End(\E)}) = \mathbb{C} \cdot \mathbbm{1}_{\E}$ by assumption. The first claim then follows by an application of the implicit function theorem for Banach spaces. 
	
	The fact that $p$ is smooth if $(A_1, A_2)$ is, is a consequence of elliptic regularity and the fact that $C_*^s$ is an algebra, along with the Coulomb property:
	\[(\nabla_1^{\End(\E)})^* \nabla_1^{\End(\E)} p = (\nabla_1^{\End(\E)}p) p^{-1} \bullet \nabla_1^{\End(\E)} p + p(\nabla_1^{\End(\E)})^* \big(p^{-1} (A_1 - A_2) p\big) \in C_*^{s}\]
	implies $p \in C_*^{s+2}$. Bootstrapping we obtain $p_{A_1, A_2} \in C^\infty$. Here $\bullet$ denotes the operation of taking the inner product on the differential form side and multiplication on the endomorphism side.
	
	Eventually, we compute the derivative of $\phi(A)$. Write $p_A := p_{0, A}$ and $\chi_A := \chi_{0, A}$, where $\chi_A$ is orthogonal to $i \mathbbm{1}_{\E}$ with respect to the $L^2$-scalar product, so that by definition
	\begin{equation}
	\label{equation:phi}
	\phi(A) = \nabla^{\E} + p_A^{-1} \nabla^{\End(\E)} p_A + p_A^{-1} A p_A.
	\end{equation}
	By differentiating the relation $F(A, \chi_A) := F(0, A,\chi_A)=0$ at $A=0$, we obtain for every $\Gamma \in C^s_*(M,T^*M \otimes \End_{\mathrm{sk}}(\E))$:
	\begin{align*}
	0 &= \dd_A F|_{A=0,\chi=0}(\Gamma) + \dd_\chi F|_{A=0,\chi=0}(\dd \chi_A|_{A=0}(\Gamma))\\
	&= (\nabla^{\End(\E)})^*\Gamma + (\nabla^{\End(\E)})^*\nabla^{\End(\E)} \dd \chi_A|_{A=0}(\Gamma),
	\end{align*}
	that is $\dd \chi_A|_{A=0}(\Gamma) = - [(\nabla^{\End(\E)})^*\nabla^{\End(\E)}]^{-1} (\nabla^{\End(\E)})^*\Gamma$. 	
	 Observe that $\dd p_A|_{A=0} = \dd \chi_A|_{A=0}$ via the exponential map and by \eqref{equation:phi}, we obtain:
	\[
	\dd \phi|_{A=0}(\Gamma) = \nabla^{\End(\E)} \dd \chi_A|_{A=0}(\Gamma) + \Gamma = \Gamma -  \nabla^{\End(\E)}[(\nabla^{\End(\E)})^*\nabla^{\End(\E)}]^{-1} (\nabla^{\End(\E)})^*\Gamma.
	\]
	We then conclude by \eqref{equation:projection}.
\end{proof}

In particular, the proof also gives that the map
\[
C^{s}_*(M,T^*M \otimes \mathrm{End}_{\mathrm{sk}}(\E)) \ni A \mapsto \phi(A) \in \mc{S}_s := \mc{S}_{s}(\nabla^{\E}),
\]
is constant along orbits of gauge-equivalent connections (by construction).

\subsection{Resonances at $z=0$: finer remarks}
\label{ssection:finer}

We recall that $\nabla^{\E}$ is an arbitrary smooth unitary connection on a Hermitian vector bundle $\E \rightarrow M$ and that the differential operator $\X := (\pi^*\nabla^{\E})_X$ is defined over $SM$.

In the subsequent lemma, we will use the following characterization of Pollicott-Ruelle resonances: $z_0 \in \C$ is a Pollicott-Ruelle resonance of $-\X$ if and only if there exists a non-zero distribution $u \in \mathcal{D}'_{E_u^*}(\M, \E) $ such that $-\X u = z_0 u$. Here for a closed conic set $\Gamma \subset T^*\mc{M}$, we denote by $\mathcal{D}'_{\Gamma}(\M, \E)$ the set of distributional sections $u$, such that the \emph{wavefront set} satisfies $\operatorname{WF}(u) \subset \Gamma$, see \cite[Chapter 8]{Hormander-90} for the background on the wavefront set. This characterization follows by the flexibility in the choice of anisotropic spaces (see e.g. \cite[Theorem 13]{Faure-Sjostrand-11} for details).

\begin{lemma}\label{lemma:symmetricspectrum}
The Pollicott-Ruelle resonance spectrum of $\X$ is symmetric with respect to the real axis.
\end{lemma}

\begin{proof}
If $z_0$ is a resonance associated to $-\X$, i.e. a pole of $z \mapsto \RR_+(z)$, then by \eqref{equation:adjoint} $\overline{z}_0$ is a resonance associated to $+\X$, i.e. a pole of $z \mapsto \RR_-(z)$. Let $u \in \mathcal{D}'_{E_u^*}(\M, \E) $ be a non-zero resonant state such that $-\X u = z_0 u$. Let $R : (x,v) \mapsto (x,-v)$ be the antipodal map on $SM$; note that the pullback $R^*$ acts on sections of $\pi^*\mc{E}$ and that $R^* \pi^* \nabla^{\mc{E}} = \pi^* \nabla^{\mc{E}}$ since $\pi \circ R = \pi$. Observe that $R^*: \mathcal{D}'_{E_u^*}(\M, \E) \to \mathcal{D}'_{E_s^*}(\M, \E)$ is an isomorphism, since $R^*X=-X$ so $R^*$ will swap the stable and the unstable bundles. Then $z_0 R^* u = -R^* \X u = \X R^* u$ and $R^*u \in \mathcal{D}'_{E_s^*}(\M, \E) $. Thus $R^*u$ is a resonant state associated to the resonance $z_0$. So both $z_0$ and $\overline{z}_0$ are resonances for $+\X$, which completes the proof.\footnote{Alternatively, by inspecting the construction of the anisotropic Sobolev space in \cite{Faure-Sjostrand-11}, we see that we may assume $R^* : \mc{H}_\pm^s \rightarrow \mc{H}_\mp^s$ is an isomorphism, simply by replacing the degree function $m$ in the construction by $\frac{m - R^*m}{2}$, which then implies that $R^*m = -m$.}
\end{proof}

We remark that the preceding lemma also holds in sufficiently high finite regularity by a density argument and the continuity of resonances established in Lemma \ref{lemma:perturbation-pr}.

Consider a contour $\gamma \subset \mathbb{C}$ such that $-\X$ has no resonances other than zero inside or on $\gamma$. By continuity of resonances (see Lemma \ref{lemma:perturbation-pr}), there is an $\varepsilon > 0$ such that for all skew-Hermitian $1$-forms $A$ with $\|A\|_{C_*^s} < \varepsilon$ the operator $-\X_A := -(\pi^*(\nabla^{\E} + A))_X$ has no resonances on $\gamma$. Here we need to take $s$ large enough (depending on the dimension), so that the framework of microlocal analysis applies.

In the specific case where $\dim \ker(\X|_{\mc{H}_+}) = 1$, we denote by $\lambda_A$ the unique resonance of $-\X_A$ enclosed by $\gamma$. Note that the map $A \mapsto \lambda_A$ is $C^3$-regular for $\varepsilon > 0$ small enough (see Lemma \ref{lemma:perturbation-pr}).

\begin{lemma}
\label{lemma:diff-2}
Assume that $\dim \ker(\X|_{\mc{H}_+}) = 1$. Then $\lambda_A \in \mathbb{R}$ and for $\Gamma \in C^\infty(M,T^*M \otimes \mathrm{End}_{\mathrm{sk}}(\E))$:
\[
\dd \lambda_A|_{A=0} = 0, ~~ \dd^2 \lambda_A|_{A=0}(\Gamma,\Gamma) = - \langle \Pi \pi_1^*\Gamma u_0, \pi_1^*\Gamma u_0 \rangle_{L^2},
\]
where $u_0$ is a resonant state associated to $A=0$ and $\|u_0\|_{L^2}=1$.
\end{lemma}

\begin{proof}
By the symmetry property of Lemma \ref{lemma:symmetricspectrum} and continuity of resonances we know $\lambda_A \in \mathbb{R}$. Also, observe that $u_0$ is either pure odd or pure even with respect to $v$ (i.e. $R^*u_0 = u_0$ or $R^*u_0 = -u_0$) because $R^*$ keeps $\ker (\X)$ fixed and $\ker (\X)$ is assumed to be one dimensional.

For the second claim, it is sufficient to start with the equality $-\X_{\tau \Gamma} u_{\tau \Gamma} = \lambda_{\tau \Gamma} u_{\tau \Gamma}$, where $\Gamma \in C^\infty(M,T^*M \otimes \mathrm{End}_{\mathrm{sk}}(\E))$, $\tau \in (-\delta, \delta)$ is small enough so that $\tau \mapsto \lambda_{\tau \Gamma}$ and $\tau \mapsto u_{\tau \Gamma} \in \mc{H}_+$ are $C^3$, and to compute the derivatives at $\tau=0$. Observe that $\dot{\X}_0 = \pi_1^*\Gamma,\, \ddot{\X}_0 = 0$. We obtain: $-\dot{\X}_0 u_0 -\X_0 \dot{u}_0 = \dot{\lambda}_0 u_0$ and taking the $L^2$ scalar product with $u_0$, we find $\dot{\lambda}_0 = 0$, using that $u_0$ is either pure odd or pure even. Thus $\dot{u}_0- \Pi_0^+ \dot{u}_0 = -\mathbf{R}_0^+ \pi_1^*\Gamma u_0$. Then, taking the second derivative at $\tau = 0$, we get: $-2 \pi_1^*\Gamma \dot{u}_0 - \X_0 \ddot{u}_0 = \ddot{\lambda}_0 u_0$ and taking once again the scalar product with $u_0$, we find $\ddot{\lambda}_0 = - 2 \langle \mathbf{R}_0^+ \pi_1^*\Gamma u_0,  \pi_1^*\Gamma u_0 \rangle_{L^2}$. It is then sufficient to observe that by symmetry (using $(\mathbf{R}_0^+)^* = \mathbf{R}_0^-$ and $\ddot{\lambda}_0 \in \mathbb{R}$)
\[
\langle \mathbf{R}_0^+ \pi_1^*\Gamma u_0,  \pi_1^*\Gamma u_0 \rangle_{L^2} = \langle \mathbf{R}_0^- \pi_1^*\Gamma u_0,  \pi_1^*\Gamma u_0 \rangle_{L^2}.
\]
This proves the result.
\end{proof}

\subsection{P-R resonance at $0$ of the mixed connection: opaque case}

\label{ssection:pollicott-ruelle-dea}

We now further assume that  $\X := (\pi^* \nabla^{\mathrm{End}(\E)})_X$ has the resonant space at $0$ spanned by $\mathbbm{1}_{\E}$. This condition is known as the \emph{opacity} of the connection $\pi^*\nabla^{\E}$. When $(M,g)$ is Anosov, this is known to be a generic condition, see \cite[Theorem 1.6]{Cekic-Lefeuvre-20}.

As in \S\ref{sssection:connection-induced}, we assume that $s \gg 1$ (so that standard microlocal analysis and the perturbation Lemma \ref{lemma:perturbation-pr} apply) and we introduce the mixed connection induced by $\nabla_1^{\E} = \nabla^{\E} + A_1$ and $\nabla_2^{\E} = \nabla^{\E} + A_2$, namely
\[
\nabla^{\Hom(\nabla_1^{\E}, \nabla_2^{\E})} u = \nabla^{\End(\E)} u + A_2u - u A_1,
\]
and we set $\X_{A_1, A_2} :=  (\pi^* \nabla^{\Hom(\nabla_1^{\E}, \nabla_2^{\E})})_X$ and $z \mapsto \RR_{\pm}(z, A_i)$ for the resolvents. The operator $\X := \X_{0, 0}$ has the resonant space at $z=0$ spanned by $\mathbbm{1}_{\E}$. For $\|A_1\|_{C_*^s}, \|A_2\|_{C_*^s}$ small enough by Lemma \ref{lemma:perturbation-pr} the map $(A_1, A_2) \mapsto \lambda_{A_1, A_2}$ is $C^3$-regular, where we denote by $\lambda_{A_1, A_2}$ the unique resonance close to $0$, namely the unique pole of $\RR_{\pm}(z,A_i)$ inside the small contour $\gamma$ around zero (see \S \ref{ssection:finer}). Since $\nabla^{\Hom(\nabla_1^{\E}, \nabla_2^{\E})}$ is unitary, by \eqref{eq:spectrum-left-half-plane} we have $\re(\lambda_{A_1, A_2}) \leq 0$, and by Lemmas \ref{lemma:diff-2} and \ref{lemma:perturbation-pr} that $\lambda_{A_1, A_2} \in \mathbb{R}$ (as otherwise the rank of the projector in Lemma \ref{lemma:perturbation-pr} would be at least $2$, contradicting the fact that its locally constant). In fact $\lambda_{A_1, A_2}$ descends to the moduli space: if $p_i^*(\nabla^{\E} + A_i') = \nabla^{\E} + A_i$ for $\|A_i'\|_{C_*^s}$ small enough, then using \eqref{equation:lien} we get
\begin{equation}\label{eq:mixedunitaryequiv}
	\X_{A_1', A_2'}u = (p_2)^{-1} \cdot \X_{A_1, A_2}(p_2 u (p_1)^{-1}) \cdot p_1, \quad u \in \mc{H}_+.
\end{equation}
Here we used that $\mc{H}_+$ is stable under multiplication by $C_*^s$ for $s$ large enough; hence $\X_{A_1, A_2}$ and $\X_{A_1', A_2'}$ have equal P-R spectra and so $\lambda_{A_1, A_2} = \lambda_{A_1', A_2'}$.

Next, we need a uniform estimate for the generalized X-ray transform operator $\Pi_1^{\End(\E)}$, introduced in Definition \ref{definition:generalized-xray} with $m=1$ and with respect to the endomorphism connection $\operatorname{End}(\nabla^{\mc{E}})$ (see Definition \ref{definition:mixed}), in a neighbourhood of $\nabla^{\E}$:

\begin{lemma}\label{lemma:Pi_1uniform}
	Assume $\Pi_1^{\End(\E)}$ is $s$-injective. There are constants $\varepsilon, C > 0$ depending only on $\nabla^{\E}$ such that for all skew-Hermitian $1$-forms with $\|A\|_{C_*^s} < \varepsilon$:
	\[
\forall f \in H^{-1/2}(M, T^*M \otimes \End(\mc{E})), ~~~ \langle \Pi_1^{\End(\nabla^{\E} + A)} f, f \rangle_{L^2} \geq C \|\pi_{\ker (\nabla^{\End(\nabla^{\E} + A)})^*}f\|^2_{H^{-1/2}}.
\] 
\end{lemma}
\begin{proof}
	Observe firstly that the left hand side of the inequality vanishes on potential tensors by \eqref{eq:Pi_mproperties} and hence it suffices to consider $f \in \ker (\nabla^{\End(\nabla^{\E} + A)})^*$. Then we obtain:
	\begin{align*}
			\langle \Pi_1^{\End(\nabla^{\E} + A)} &f, f \rangle_{L^2} = \langle \Pi_1^{\End(\nabla^{\E})} f, f \rangle_{L^2} + \langle (\Pi_1^{\End(\nabla^{\E} + A)} - \Pi_1^{\End(\nabla^{\E})}) f, f \rangle_{L^2}\\
			&\geq C_0 \|\pi_{\ker(\nabla^{\End(\E)})^*}f\|_{H^{-1/2}}^2 - \|\Pi_1^{\End(\nabla^{\E} + A)} - \Pi_1^{\End(\nabla^{\E})}\|_{H^{-1/2} \to H^{1/2}} \|f\|_{H^{-1/2}}^2\\
			&\geq \frac{1}{4} C_0 \|f\|_{H^{-1/2}}^2 - \frac{1}{2} C_0 \|\pi_{\ker(\nabla^{\End(\nabla^{\E})})^*} - \pi_{\ker (\nabla^{\End(\nabla^{\E} + A)})^*}\|_{H^{-1/2} \to H^{-1/2}} \|f\|_{H^{-1/2}}^2\\
			&\geq \frac{1}{8}C_0 \|f\|_{H^{-1/2}}^2.
	\end{align*}
	In the second line, we used Lemma \ref{lemma:generalized-xray} (item 3) with a constant $C_0$. In the next line we used that the map $A \mapsto \Pi_1^{\End(\nabla^{\E} + A)} \in \Psi^{-1}$ is continuous; the proof of this fact is analogous to the proof of \cite[Proposition 4.1]{Guillarmou-Knieper-Lefeuvre-19} and we omit it. Thus for $\varepsilon > 0$ small enough $\|\Pi_1^{\End(\nabla^{\E} + A)} - \Pi_1^{\End(\nabla^{\E})}\|_{H^{-1/2} \to H^{1/2}} \leq C_0/4$. Similarly in the last line we used that $A \mapsto \pi_{\ker(\nabla^{\End(\nabla^{\E} + A)})^*} \in \mc{L}(H^{-1/2})$ is continuous which follows from standard microlocal analysis from \eqref{equation:projection}: more precisely, this is a consequence of the fact that $(\nabla^{\End(\nabla^{\E})})^*\nabla^{\End(\nabla^{\E})}$ is an isomorphism on $C^\infty(M,\End(\E)) \cap (\C \mathbbm{1}_{\E})^\bot$ by the opacity assumption (where the orthogonal is taken with respect to the $L^2$-scalar product) and by continuity, this holds for all operators $(\nabla^{\End(\nabla^{\E}+A)})^*\nabla^{\End(\nabla^{\E}+A)}$ as long as $\|A\|_{C^s_*}$ is small enough (for $s$ large enough); hence the operator $[(\nabla^{\End(\nabla^{\E}+A)})^*\nabla^{\End(\nabla^{\E}+A)}]^{-1}$ in \eqref{equation:projection} is locally uniformly bounded and so by the resolvent formula, it varies continuously with respect to the connection.
	So again we choose $\varepsilon > 0$ small enough so that $\|\pi_{\ker(\nabla^{\End(\nabla^{\E})})^*} - \pi_{\ker (\nabla^{\End(\nabla^{\E} + A)})^*}\|_{H^{-1/2} \to H^{-1/2}} \leq C_0/4$. The claim follows by setting $C = C_0/8$.
\end{proof}

Recall that $\lambda_{A_1, A_2} \leq 0$ in the following:

\begin{lemma}
\label{lemma:borne-lambda-a-1}
Assume that the generalized X-ray transform $\Pi^{\End(\E)}_1$ defined with respect to the connection $\nabla^{\End(\E)}$ is $s$-injective. For $s \gg 1$ large enough, there exist constants $\eps, C > 0$ such that for all $A_i \in C^s(M,T^*M \otimes  \End_{\mathrm{sk}}(\E))$ with $\|A_i\|_{C^s_*} < \eps$ for $i = 1, 2$, we have:
\[
0 \leq  \|\phi(A_1, A_2) - \nabla_1^{\E}\|^2_{H^{-1/2}(M,T^*M\otimes \End_{\mathrm{sk}}(\E))} \leq C |\lambda_{A_1, A_2}|.
\]
\end{lemma}

\begin{proof}
We introduce two functionals in the vicinity of $\nabla^{\E}$, for small enough $\varepsilon > 0$: 
\begin{align*}
	F_1, F_2 &: \big(C^{s_0}_*(M, \End_{\mathrm{sk}}(\E)) \cap \{\|A\|_{C_*^{s_0}} < \varepsilon\}\big)^2 \to \mathbb{R},\\
	 F_1(A_1, A_2) &:=\lambda_{A_1, A_2}, \quad F_2(A_1, A_2) := -\|\phi(A_1, A_2)-\nabla_1^{\E}\|^2_{H^{-1/2}}.
\end{align*}
They are well-defined and restrict as $C^{3}$-regular maps on $\mc{S}_{s_0}$ for some $s_0 \gg 1$ large enough by Lemma \ref{lemma:coulomb} and the discussion above. Moreover, using Lemma \ref{lemma:diff-2}, we have for all $A$:
\[F_1(A, A) = F_2(A, A) = 0 \quad \mathrm{and}\quad \dd F_1|_{(A, A)} = \dd F_2|_{(A, A)} = 0.\] 
We will compare the second partial derivatives in the variable $A_2$ at a point $(A, A)$. Given $\Gamma \in T_{\nabla^{\E}} \mc{S}_{s_0} \simeq \ker (\nabla^{\End(\E)})^*$, we have by Lemma \ref{lemma:coulomb}:
\begin{equation}\label{eq:F_2}
\dd^2_{A_2} F_2|_{(A, A)}(\Gamma,\Gamma) = -2\|\pi_{\ker (\nabla^{\mathrm{End}(\nabla^{\E} + A)})^*} \Gamma\|^2_{H^{-1/2}}.
\end{equation}
By Lemma \ref{lemma:diff-2}, we have
\[
\dd^2_{A_2} F_1|_{(A, A)}(\Gamma,\Gamma) = -c^2 \langle \Pi^{\End(\nabla^{\E} + A)} \pi_1^* \Gamma \mathbbm{1}_{\E}, \pi_1^* \Gamma \mathbbm{1}_{\E} \rangle_{L^2} = - c^2 \langle \Pi^{\End(\nabla^{\E} + A)}_1 \Gamma, \Gamma \rangle_{L^2},
\]
for some constant $c > 0$, where $\Pi^{\End(\nabla^{\E} +A)}$ denotes the $\Pi$ operator with respect to the endomorphism connection induced by $\nabla^{\E} + A$. We used here that the orthogonal projection to the resonant space $\mathbb{C} \mathbbm{1}_{\E}$ of $\X_{A, A}$ at zero vanishes, because $\langle{\pi_1^*\Gamma, \mathbbm{1}_{\E}}\rangle_{L^2} = 0$ as $\pi_1^*\Gamma$ is odd. For $\varepsilon > 0$ small enough, by Lemma \ref{lemma:Pi_1uniform} we know $\Pi^{\End(\nabla^{\E} + A)}_1$ is $s$-injective and there is a constant $C' = C'(\nabla^{\E}) > 0$ such that: 
\begin{equation}\label{eq:F_1}
\dd^2_{A_2} F_1|_{(A, A)}(\Gamma,\Gamma) \leq - C' \|\pi_{\ker(\nabla^{\mathrm{End}(\nabla^{\E} + A)})^*}\Gamma\|^2_{H^{-1/2}} = C'/2  \dd^2_{A_2} F_2|_{(A, A)}(\Gamma,\Gamma).
\end{equation}
As a consequence, writing $G(A_2) := F_1(A, A_2) - C'/4 \times F_2(A, A_2)$, we have $G(A)=0, \dd G|_{A_2 = A} = 0$ and by \eqref{eq:F_1}, \eqref{eq:F_2}
\[\dd^2 G|_{A_2 = A}(\Gamma,\Gamma) \leq C'/4  \dd^2_{A_2} F_2|_{(A, A)}(\Gamma,\Gamma) = -C'/2 \|\pi_{\ker(\nabla^{\mathrm{End}(\nabla^{\E} + A)})^*} \Gamma\|^2_{H^{-1/2}}.\]
If we now Taylor expand the $C^3$-map $\mc{S}_{s_0} \ni A_2 \mapsto G(A_2)$ at $A_2 = A$, we obtain:
\begin{align*}
G(A + \Gamma) &= \dfrac{1}{2} \dd^2 G|_{A_2=A}(\Gamma, \Gamma) + \mc{O}(\|\Gamma\|_{C_*^{s_0}}^3)\\
 &\leq -C'/4 \|\Gamma\|_{H^{-1/2}}^2 + C'/4 \|(\pi_{\ker(\nabla^{\End(\E)})^*} - \pi_{\ker(\nabla^{\End(\nabla^{\E} + A)})^*})\Gamma\|_{H^{-1/2}}^2 + C''\|\Gamma\|^3_{C_*^{s_0}}\\
 &\leq -C'/8\|\Gamma\|_{H^{-1/2}}^2 + C''\|\Gamma\|_{C_*^{s_0}}^3.
\end{align*}
In the second line we introduced a uniform constant $C'' = C''(\nabla^{\E}) > 0$ using the $C^3$-regular property and $\pi_{\ker(\nabla^{\End(\E)})^*} \Gamma = \Gamma$. For the last line, we observed that $A \mapsto \pi_{\ker(\nabla^{\End(\nabla^{\E} + A)})^*} \in \mc{L}(H^{-1/2})$ is a continuous map by equation \eqref{equation:projection} (see the last paragraph of Lemma \ref{lemma:Pi_1uniform}) and hence the $H^{-1/2} \to H^{-1/2}$ estimate is arbitrarily small for $\varepsilon$ small enough. This estimate holds for all $\|A\|_{C_*^{s_0}}, \|\Gamma\|_{C_*^{s_0}} < \varepsilon/2$.

Choosing $s \gg s_0$ and assuming that $A \in C_*^s(M,T^*M \otimes  \End_{\mathrm{sk}}(\E))$ with $\|A\|_{C^s_*} < \eps$, there is a $C'''(\nabla^{\E}) > 0$, such that for $\eps > 0$ with $C''' \varepsilon \leq C'/16$, we then obtain by interpolation:
\[
G(A + \Gamma) \leq -C'/8 \|\Gamma\|^2_{H^{-1/2}} + \underbrace{C''' \|\Gamma\|_{C^s_*}}_{\leq C'/16} \|\Gamma\|_{H^{-1/2}}^2 \leq - C'/16 \|\Gamma\|_{H^{-1/2}}^2 \leq 0.
\]
After taking $\varepsilon > 0$ small enough, the statement holds with $C=C'/2$.
\end{proof}

We note that the preceding lemma shows that $\lambda_{A_1, A_2}$ controls the distance in the moduli space between $\nabla^{\E} + A_1$ and $\nabla^{\E} + A_2$.

\begin{remark}
\rm
It was proved in \cite{Guillarmou-Knieper-Lefeuvre-19} that there exists a metric $G$ on the moduli space of isometry classes (of metrics with negative sectional curvature) which generalizes the usual Weil-Petersson metric on Teichm\"uller space in the sense that, in the case of a surface, the restriction of $G$ to Teichm\"uller space is equal to the Weil-Petersson metric. We point out that the operator $\Pi_1$ also allows to define a metric $G$ at a generic point $\mathfrak{a}_0 \in \mathbb{A}_{\E}$, similarly to \cite{Guillarmou-Knieper-Lefeuvre-19}. Indeed, if $\mathfrak{a}_0 \in \mathbb{A}_{\E}$, taking a representative $\nabla^{\E} \in \mathfrak{a}_0$, one has $T_{\mathfrak{a}_0}\mathbb{A}_{\E} \simeq \ker (\nabla^{\End})^*$ and thus, given $\Gamma \in \ker (\nabla^{\End})^*$, one can consider:
\[
G_{\mathfrak{a}_0}(\Gamma,\Gamma) := \langle \Pi_1 \Gamma, \Gamma \rangle_{L^2(M,T^*M\otimes\End(\E))} \geq c \|\Gamma\|^2_{H^{-1/2}},
\]
for some constant $c>0$. Lemma \ref{lemma:Pi_1uniform} shows that the constant $c$ is locally uniform with respect to $\mathfrak{a}_0$.
\end{remark}

\subsection{P-R resonance at $0$ of the mixed connection: non-opaque case}
\label{ssection:non-opaque}

The aim of this paragraph is to deal with neighbourhoods of connections that are not necessarily opaque, and only assume $\Pi_1^{\End(\E)}$ is injective. In other words, we do not want to assume the resonant space of $-(\pi^*\nabla^{\End(\E)})_X$ at zero is spanned by $\mathbbm{1}_{\E}$ necessarily.

Next, as in \S \ref{ssection:pollicott-ruelle-dea}, we introduce the mixed connection with respect to $\nabla^{\E} + A$ and $\nabla^{\E}$, denoted by $\nabla^{\Hom(\nabla^{\E} + A, \nabla^{\E})}$, and set $\X_A :=(\pi^*\nabla^{\Hom(\nabla^{\E} + A, \nabla^{\E})})_X$. We assume $s \gg 1$. As before, consider a contour $\gamma \subset \mathbb{C}$ around zero such that $\X := \X_0$ has only the resonance zero enclosed by $\gamma$ and $\varepsilon > 0$ such that $-\X_A$ has no resonances on $\gamma$ for all $\|A\|_{C_*^s} < \varepsilon$. We introduce
\[\Pi_A^+ := \frac{1}{2\pi i} \int_\gamma (z + \X_A)^{-1} dz, \quad \lambda_A := \Tr(-\X_A\Pi_A^+).\]
This generalises the quantity studied in \S \ref{ssection:finer}, where it was assumed that the multiplicity of $\X$ at zero is equal to one. By Lemma \ref{lemma:perturbation-pr}, we have $A \mapsto \Pi_A^+$ and $A \mapsto \lambda_A$ are $C^3$-regular. As in \eqref{eq:mixedunitaryequiv}, we observe that the operators $-\X_A$ and $-\X_{A'}$ are unitarily equivalent on $\mc{H}_+$ whenever $\nabla^{\E} + A$ and $\nabla^{\E} + A'$ are gauge equivalent; hence $\lambda_A = \lambda_{A'}$ and so $\lambda_A$ descends to the local moduli space.

Note also that $\re(\lambda_A) \leq 0$, since by \eqref{eq:spectrum-left-half-plane} all resonances of $-\X_A$ lie in the half-plane $\{\re z \leq 0\}$ and this gives us hope that $\re(\lambda_A)$ controls the distance between the connections. Assume that the resonant space of $-\X$ at zero is spanned by smooth $L^2$-orthonormal resonant states $\{u_i\}_{i = 1}^p$. We have the following generalisation of Lemma \ref{lemma:diff-2}: 

\begin{lemma}
\label{lemma:variations}
For $A \in C^\infty(M,T^*M \otimes \mathrm{End}_{\mathrm{sk}}(\E))$ with $\|A\|_{C_*^s} < \varepsilon$, we have $\lambda_A \in \mathbb{R}$ and the following perturbation formulas hold:
\[
\dd \lambda_A|_{A=0} = 0, ~~ \dd^2 \lambda_A|_{A=0}(\Gamma,\Gamma) = - \sum_{i = 1}^p \langle \Pi u_i \pi_1^*\Gamma, u_i \pi_1^*\Gamma\rangle_{L^2}.
\]
\end{lemma}
\begin{proof}
	The fact that $\lambda_A$ is real follows from the symmetry of the spectrum of $-\X_A$ shown in Lemma \ref{lemma:symmetricspectrum}. The first derivative formula is obvious as $\lambda_A \leq 0$; the second one follows from minor adaptations of \cite[Lemma 5.9]{Cekic-Lefeuvre-20}, where the analogous case of endomorphisms was considered. 
\end{proof}

Next, by a straightforward adaptation of the Lemma \ref{lemma:coulomb}, we obtain the existence of $\varepsilon > 0$, such that for all $A$ with $\|A\|_{C_*^s} < \varepsilon$, there is a smooth map $A \mapsto \phi(A) \in \mc{S}_s$ that sends $\nabla^{\E} + A$ to Coulomb gauge with respect to $\nabla^{\E}$, that is it satisfies $\phi(A) - \nabla^{\E} \in \ker (\nabla^{\End(\E)})^*$.
\begin{remark}\rm
	We cannot get the analogous statement to Lemma \ref{lemma:coulomb} for parameters $(A_1, A_2)$, because the range of $F(A_1, A_2, \chi)$ equals $\ker(\nabla_1^{\End(\E)})^\perp$ and this is not uniform in $A_1$ (i.e. $\ker \nabla_1^{\End(\E)}$ changes as we move $A_1$); the space $\mathbb{A}_{\E}$ is not a smooth manifold at reducible connections.
\end{remark}

In the following lemma, we will need to assume that $\pi_{1*} \Pi_0^+ = 0$. Equivalently, this means that the resonant states $u_i \in \ker(\X|_{\mc{H}_+})$ satisfy $\pi_{1*} u_i = 0$, for $i = 1, \dotso, p$, i.e. the degree $1$ Fourier mode of all the $u_i$ vanishes.

\begin{lemma}
\label{lemma:maininequalitygeneral}
Assume that the generalized X-ray transform $\Pi^{\End(\E)}_1$ defined with respect to the connection $\nabla^{\End(\E)}$ is $s$-injective and additionally that $\pi_{1*} \Pi_0^+ = 0$. For $s \gg 1$ large enough, there exist constants $\eps, C > 0$ such that for all $A \in C^s(M,T^*M \otimes  \End_{\mathrm{sk}}(\E))$ with $\|A\|_{C^s_*} < \eps$:
\[
0 \leq  \|\phi(A)-\nabla^{\E}\|^2_{H^{-1/2}(M,T^*M\otimes \End_{\mathrm{sk}}(\E))} \leq C |\lambda_A|.
\]
\end{lemma}
\begin{proof}
This is straightforward from the proof of Lemma \ref{lemma:borne-lambda-a-1}. With the same functionals $F_1(A) = \lambda_A$ and $F_2(A) = -\|\phi(A) - \nabla^{\E}\|^2_{H^{-1/2}}$, the only slight difference is the computation of $\dd^2 F_1$. Pick an $L^2$-orthonormal basis $u_1, \dotso, u_p$ of the resonant space of $-\X$ at zero such that $u_1 = c\mathbbm{1}_{\E}$, where $c$ is a fixed constant. By Lemmas \ref{lemma:variations} and \ref{lemma:relations-resolvent}, we have
\[
\dd^2 F_1|_{A=0}(\Gamma,\Gamma) = -\sum_{i = 1}^p \langle \Pi u_i \pi_1^* \Gamma, u_i \pi_1^* \Gamma \rangle_{L^2} \leq -c^2 \langle \Pi_1^{\End(\E)} \Gamma, \Gamma \rangle_{L^2}.
\]
Note importantly that we have used $\Pi_0^+ \pi_1^*\Gamma = 0$. This follows from the expression for the projector $\Pi_0^+ = \sum_{i = 1}^p \langle{\bullet, u_i}\rangle_{L^2} u_i$ and $\pi_{1*}u_i = 0$ for all $i$. This suffices to run the proof in the same manner. 
\end{proof}

\section{Injectivity of the primitive trace map}

We can now prove the main results stated in the introduction.

\label{section:proofs}

\subsection{The local injectivity result}

\label{ssection:proof-injectivity}

We now prove the injectivity result of Theorem \ref{theorem:injectivity}.

\begin{proof}[Proof of Theorem \ref{theorem:injectivity}]
We fix a regularity exponent $N \gg 1$ large enough so that the results of \S\ref{section:geometry} apply. We fix $\nabla^{\mc{E}}$, a smooth unitary connection on $\E$ and assume that it is \emph{generic} (see page \pageref{def:generic}). 
By mere continuity, the same properties hold for every connection $\nabla^{\E} + A$ such that $\|A\|_{C^N_*} < \eps$, where $\eps > 0$ is small enough depending on $\nabla^{\E}$.

Consider two smooth unitary connections $\nabla_i^{\E} = \nabla^{\E} + A_i$ such that $\|A_i\|_{C^N_*} <  \eps$ for $i = 1, 2$. Assume that $\mc{T}^\sharp(\nabla_1^{\E})=\mc{T}^\sharp(\nabla_2^{\E})$. By differentiating with respect to time and taking $t=0$ in \eqref{equation:cohomol}, the exact Liv\v{s}ic cocycle Theorem \ref{theorem:weak} yields the existence of a smooth map $p \in C^\infty(SM,\mathrm{U}(\E))$ such that:
\[
	\pi^* \nabla^{\Hom(\nabla_1^{\mc{E}}, \nabla_2^{\mc{E}})}_X p = 0,
\]
that is $p$ is a resonant state for the operator $\X_{A_1, A_2}$ associated to the eigenvalue $0$. Assumptions \textbf{(A)} and  \textbf{(B)} allow us to apply Lemma \ref{lemma:borne-lambda-a-1}. We therefore obtain:
\[
	\lambda_{A_1, A_2} = 0 \leq - C \|\phi(A_1, A_2) - \nabla_1^{\E}\|^2_{H^{-1/2}(M,T^*M \otimes \End(\E))} \leq 0,
\]
where $C = C(\nabla^{\E}) > 0$ only depends on $\nabla^{\E}$. Hence $\phi(A_1, A_2) = p_{A_1, A_2}^*\nabla_2^{\E} = \nabla_1^{\E}$. In other words, the connections are gauge-equivalent.
\end{proof}	

Next, we discuss a version of local injectivity in a neighbourhood of a connection which is non-opaque. We will say a map $f: X \to Y$ of topological spaces is \emph{weakly locally injective at $x_0 \in X$} if there exists a neighbourhood $U \ni x$ such that $f(x) = f(x_0)$ for $x \in U$ implies $x = x_0$. This notion appears in non-linear inverse problems where the linearisation is not continuous, see \cite[Section 2]{Stefanov-Uhlmann-08}. We have:

\begin{theorem}\label{thm:weaklocal}
	If $N \gg 1$ and $[\nabla^{\E}] \in \mathbb{A}_{\E}$ is such that the generalised $X$-ray transform $\Pi_1^{\End(\E)}$ is $s$-injective, as well as $\pi_{1*} \ker(\pi^*\nabla^{\End(\E)}_X|_{C^\infty}) = 0$, then the primitive trace map $\mc{T}^\sharp$ is weakly locally injective at $[\nabla^{\E}]$ in the $C^N$-quotient topology.
\end{theorem}
\begin{proof}
	The proof is analogous to the proof of Theorem \ref{theorem:injectivity}, by using the results of \S \ref{ssection:non-opaque}. We omit the details.
\end{proof}

We shall see below (see Lemma \ref{lemma:Pi_1inj}) that flat connections have an injective generalised $X$-ray transform $\Pi_1^{\End(\E)}$ and satisfy the additional condition that $\ker (\pi^*\nabla^{\E})_X|_{C^\infty}$ consists of elements of degree zero (but might not be opaque). The previous Theorem therefore shows that the primitive trace map is weakly locally injective near such connections. Let us state this as a corollary for the trivial connection, as it partially answers an open question of Paternain \cite[p33, Question (3)]{Paternain-13}.

\begin{corollary}
	Let $\E = M \times \mathbb{C}^r$ be the trivial Hermitian vector bundle equipped with the trivial flat connection $d$. Then there exists a neighbourhood $\mc{U} \ni [d]$ in $\mathbb{A}_{\E}$ with $C^N$-quotient topology such that $[d]$ is the unique gauge class of transparent connections in $\mc{U}$.
\end{corollary}

\subsection{Global injectivity results}

We now detail some cases in which Theorem \ref{theorem:injectivity} can be upgraded.

\subsubsection{Line bundles}

\label{sssection:line}

We let $\mc{T}^\sharp_1$ be the restriction of the total primitive trace map \eqref{equation:trace-total} to line bundles. The moduli space of all connections on line bundles $\mathbb{A}_1$ carries a natural Abelian group structure using the tensor product.
When restricted to line bundles, the primitive trace map $\mc{T}^\sharp_1$ takes value in $\ell(\mc{C}^\sharp,\mathrm{U}(1))$, namely the set of sequences indexed by primitive free homotopy classes. We have:

\begin{lemma}
The map $\mc{T}^\sharp_1 : \mathbb{A}_1 \rightarrow \ell^\infty(\mc{C}^\sharp,\mathrm{U}(1))$ is a multiplicative group homomorphism.
\end{lemma}
\begin{proof}
	Left as an exercise to the reader.
\end{proof}

\begin{remark}
\rm
There also exists a group homomorphism for higher rank vector bundles by taking the determinant instead of the trace. More precisely, writing $\mathbb{A}_r$ for the set of all unitary connections on all possible Hermitian vector bundles of rank $r$ (up to isomorphisms), one can set
\begin{equation}
\label{equation:det}
\det{}^\sharp : \mathbb{A} \rightarrow \ell^\infty(\mc{C}^\sharp,\mathrm{U}(1)),
\end{equation}
by taking the determinant of the holonomy along each closed primitive geodesic. This map is also a group homomorphism (where the group structure on $\bigsqcup_{r \geq 0} \mathbb{A}_r$) is also obtained by tensor product. Nevertheless, the determinant map \eqref{equation:det} cannot be injective as all trivial bundles (of different ranks) have same image.
\end{remark}

We have the following result, mainly due to Paternain \cite{Paternain-09}:

\begin{prop}[Paternain]
\label{proposition:line}
Let $(M,g)$ be a smooth Anosov $n$-manifold. If $n \geq 3$, then the restriction of the primitive trace map to line bundles
\begin{equation}
\label{equation:trace-line}
\mc{T}_1^\sharp : \mathbb{A}_1 \longrightarrow \ell^\infty(\mc{C}^\sharp),
\end{equation}
is globally injective. Moreover, if $n = 2$ then:
\[
\ker \mc{T}^\sharp_1 = \left\{ ([\kappa^{\otimes n}], [{\nabla^{\mathrm{LC}}}^{\otimes n}]), n \in \Z\right\},
\]
where $\kappa \rightarrow M$ denotes the canonical line bundle and $\nabla^{\mathrm{LC}}$ is connection induced on $\kappa$ by the Levi-Civita connection.
\end{prop}

Observe that on surfaces, the trivial line bundle $\C \times M \rightarrow M$ (with trivial connection) and the canonical line bundle $\kappa \rightarrow M$ (with the Levi-Civita connection) both have trivial holonomy but are not isomorphic. This explains the existence of a non-trivial kernel for $n=2$. We will need this preliminary lemma:

\begin{lemma}
\label{lemma:iso}
Let $(M,g)$ be a smooth closed Riemannian manifold of dimension $\geq 3$ and let $\pi : SM \rightarrow M$ be the projection. Let $\mc{L}_1 \rightarrow M$ and $\mc{L}_2 \rightarrow M$ be two Hermitian line bundles. If $\pi^*\mc{L}_1 \simeq \pi^*\mc{L}_2$ are isomorphic, then $\mc{L}_1 \simeq \mc{L}_2$ are isomorphic.
\end{lemma}

\begin{proof}
The topology of line bundles is determined by their first Chern class. As a consequence, it suffices to show that $c_1(\mc{L}_1) = c_1(\mc{L}_2)$. By assumption, we have $c_1(\pi^*\mc{L}_1) = \pi^* c_1(\mc{L}_1) =  c_1(\pi^*\mc{L}_2) = \pi^* c_1(\mc{L}_2)$ and thus it suffices to show that $\pi^* : H^2(M,\Z) \rightarrow H^2(SM,\Z)$ is injective when $\dim(M) \geq 3$. But this is then a mere consequence of the Gysin exact sequence \cite[Proposition 14.33]{Bott-Tu-82}. 
\end{proof}

We can now prove Proposition \ref{proposition:line}:

\begin{proof}[Proof of Proposition \ref{proposition:line}]
Assume that $\mc{T}^\sharp_1(\mathfrak{a}_1) = \mc{T}^\sharp_1(\mathfrak{a}_2)$, where $\mathfrak{a}_1 \in \mathbb{A}_{\mc{L}_1}$ and $\mathfrak{a}_2 \in \mathbb{A}_{\mc{L}_2}$ are two classes of connections defined on two (classes of) line bundles. By Theorem \ref{theorem:weak}, we obtain that the pullback bundles $\pi^*\mc{L}_1$ and $\pi^*\mc{L}_2$ are isomorphic, hence $\mc{L}_1 \simeq \mc{L}_2$ are isomorphic by Lemma \ref{lemma:iso}. Up to composing by a first bundle (unitary) isomorphism, we can therefore assume that $\mc{L}_1 = \mc{L}_2 =: \mc{L}$. Let $\nabla^\mc{L}_1 \in \mathfrak{a}_1$ and $\nabla^\mc{L}_2 \in \mathfrak{a}_2$ be two representatives of these classes. They satisfy $\mc{T}^\sharp(\nabla^{\mc{L}}_1) = \mc{T}^\sharp(\nabla^{\mc{L}}_2)$. Combing Theorem \ref{theorem:weak} with \cite[Theorem 3.2]{Paternain-09}, the primitive trace map $\mc{T}^\sharp_{\mc{L}}$ is known to be globally injective for connections on the same fixed bundle. Hence $\nabla^{\mc{L}}_1$ and $\nabla^{\mc{L}}_2$ are gauge-equivalent.

For the second claim, $\mathrm{x} = ([\mc{L}],\mathfrak{a})$. If $\mc{T}^\sharp_1(\mathrm{x}) = (1,1,...)$ (i.e. the connection is transparent), then by Theorem \ref{theorem:weak}, one has that $\pi^*\mc{L} \rightarrow SM$ is trivial. By the Gysin sequence \cite[Proposition 14.33]{Bott-Tu-82}, this implies that $c_1(\mc{L})$ is divisible by $2g-2$, where $g$ is the genus of $M$ (see \cite[Theorem 3.1]{Paternain-09}), hence $[\mc{L}] = [\kappa^{\otimes n}]$ for some $n \in \Z$. Moreover, the Levi-Civita connection on $\kappa^{\otimes n}$ is transparent and by uniqueness (see \cite[Theorem 3.2]{Paternain-09}), this implies that $\mathfrak{a} = [{\nabla^{\mathrm{LC}}}^{\otimes n}]$.
\end{proof}

\begin{remark}
\rm
The target space in \eqref{equation:trace-line} is actually $\ell^\infty(\mc{C}^\sharp,\mathrm{U}(1))$ (sequences indexed by $\mc{C}^\sharp$ and taking values in $\mathrm{U}(1)$) which can be seen as a subset of $\mathrm{U}(\ell^\infty(\mc{C}^\sharp))$, the group of unitary operators of the Banach space $\ell^\infty(\mc{C}^\sharp)$ (equipped with the sup norm). Then $\mc{T}_1^{\sharp}$ is a group homomorphism and Proposition \ref{proposition:line} asserts that
\[
\mc{T}_1^{\sharp} : \mathbb{A}_1 \rightarrow \mathrm{U}(\ell^\infty(\mc{C}^\sharp))
\]
is a faithful unitary representation of the Abelian group $\mathbb{A}_1$.
\end{remark}

We end this paragraph with a generalization of Proposition \ref{proposition:line}. There is a natural submonoid $\mathbb{A}' \subset \mathbb{A}$ which is obtained by considering sums of lines bundles equipped with unitary connections, that is:
\[
\mathbb{A}' := \left\{ \mathrm{x}_1 \oplus ... \oplus \mathrm{x}_k ~\mid~ k \in \N, \mathrm{x}_i \in \mathbb{A}_1\right\}.
\]
We then have the following:

\begin{theorem}
\label{theorem:sum}
Let $(M,g)$ be a smooth Anosov Riemannian manifold of dimension $\geq 3$. Then the restriction of the primitive trace map to $\mathbb{A}'$:
\[
\mc{T}^{\sharp} : \mathbb{A}' \longrightarrow \ell^\infty(\mc{C}^\sharp)
\]
is globally injective.
\end{theorem}

\begin{proof}
We consider $\mc{L} := \mc{L}_1 \oplus ... \oplus \mc{L}_k$ and $\mc{J} = \mc{J}_1 \oplus ... \oplus \mc{J}_{k'}$, two Hermitian vector bundles over $M$, equipped with the respective connections $\nabla^{\mc{L}_1} \oplus ... \oplus \nabla^{\mc{L}_k}$ and $\nabla^{\mc{J}_1} \oplus ... \oplus \nabla^{\mc{J}_{k'}}$ and we assume that they have same image by the primitive trace map. Fixing a periodic point $(x_\star,v_\star)$ and applying Proposition \ref{proposition:representation}, we obtain that $k=k'$ and the existence of two isomorphic representations $\rho_{\mc{L}} : \mathbf{G} \to \mathrm{U}(\pi^*\mc{L}_{(x_\star,v_\star)})$ and $\rho_{\mc{J}} : \mathbf{G} \to \mathrm{U}(\pi^*\mc{J}_{(x_\star,v_\star)})$, where $\mathbf{G}$ denotes Parry's free monoid at $(x_\star, v_\star)$. Since these representations are sums of $1$-dimensional representations, there is a unitary isomorphism $p_\star : \pi^*\mc{L}_{(x_\star,v_\star)} \to \pi^*\mc{J}_{(x_\star,v_\star)}$ such that for each $i \in \left\{1,...,k\right\}$, there exists $\sigma(i) \in \left\{1,...,k\right\}$ with $p^{(i)}_\star := p_\star|_{\pi^*\mc{L}_{i, (x_\star, v_\star)}}$ is a representation isomorphism $p^{(i)}_\star: \pi^* \mc{L}_{i, (x_\star,v_\star)} \to \pi^* \mc{J}_{\sigma(i), (x_\star,v_\star)}$.

Now, following the arguments of Lemma \ref{lemma:p-minus}, we parallel-transport $p^{(i)}_\star$ along the homoclinic orbits with respect to the pullback of the mixed connection $\pi^*\nabla^{\mathrm{Hom}(\nabla^{\mc{L}_i}, \nabla^{\mc{J}_{\sigma(i)}})}_X$ (induced by the connections $\nabla^{\mc{L}_i}$ on $\mc{L}_i$ and $\nabla^{\mc{J}_{\sigma(i)}}$ on $\mc{J}_{\sigma(i)}$); the Lipschitz-regularity of the obtained section follows, as in Lemma \ref{lemma:lipschitz}, from the fact that $p_\star^{(i)} \rho_{\mc{L}_i}(g) = \rho_{\mc{J}_{\sigma(i)}}(g) p_\star^{(i)}$ for all $g \in \mathbf{G}$.
Using the regularity result of \cite{Bonthonneau-Lefeuvre-20}, we thus obtain a unitary section
\[
p^{(i)} \in C^\infty(SM,\pi^*\mathrm{Hom}(\mc{L}_i, \mc{J}_{\sigma(i)}))
\]
conjugating the parallel transports along geodesic flowlines with respect to the connections $\pi^* \nabla^{\mc{L}_i}$ and $\pi^* \nabla^{\mc{J}_{\sigma(i)}}$. In particular, the existence of such $p^{(i)}$ ensures that $\mc{T}^{\sharp}_1(\mc{L}_i,\nabla^{\mc{L}_i}) = \mc{T}^{\sharp}_1(\mc{J}_{\sigma(i)},\nabla^{\mc{J}_{\sigma(i)}})$. We then conclude by Proposition \ref{proposition:line}, showing that each pair $(\mc{L}_i, \nabla^{\mc{L}_i})$ is isomorphic to $(\mc{J}_{\sigma(i)}, \nabla^{\mc{J}_{\sigma(i)}})$, for $i = 1, \dotso, k$.
\end{proof}

\subsubsection{Flat bundles}

\label{sssection:flat}

We discuss the particular case of flat vector bundles. It is well-known that the data of a vector bundle equipped with a unitary connection (modulo isomorphism) is equivalent to a unitary representation of the fundamental group (modulo inner automorphisms of the unitary group). More precisely, given $\rho \in \Hom(\pi_1(M),\mathrm{U}(r))$, one can associate a Hermitian bundle $\E \rightarrow M$ equipped with a flat unitary connection $\nabla^{\E}$ by the following process: let $\widetilde{M}$ be the universal cover of $M$; consider the trivial bundle $\C^r \times \widetilde{M}$ equipped with the flat connection $d$ and define the relation $(x,v) \sim (x',v')$ if and only if $x' = \gamma(x), v' = \rho(\gamma)v$, for some $\gamma \in \pi_1(M)$; then $(\E,\nabla^{\E})$ is obtained by taking the quotient $\C^r \times \widetilde{M}/\sim$. Changing $\rho$ by an isomorphic representation $\rho' = p \cdot \rho \cdot p^{-1}$ (for $p \in \mathrm{U}(r)$) changes the connection by a gauge-equivalent connection and this process gives a one-to-one correspondence between the moduli spaces.

For $r \geq 0$, we let
\[
\mc{M}_r := \mathrm{Hom}(\pi_1(M),\mathrm{U}(r))/\sim,
\]
be the moduli space of unitary representations of the fundamental group, where two representations are equivalent $\sim$ whenever they are isomorphic. The space $\mc{M}_r$ is called the \emph{character variety}, see \cite{Labourie-13} for instance. For $r=0$, it is reduced to a point; for $r=1$, it is given by $\mc{M}_1 = \mathrm{U}(1)^{b_1(M)}$, where $b_1(M)$ denotes the first Betti number of $M$. Given $\mathrm{x} \in \mc{M}_r$, we let $\Psi(\mathrm{x}) = (\E_{\mathrm{x}}, \nabla^{\E_{\mathrm{x}}})$ be the data of a Hermitian vector bundle equipped with a unitary connection (up to gauge-equivalence) described by the above process. The primitive trace map $\mc{T}^\sharp$ can then be seen as a map:
\[
\mc{T}^\sharp : \bigsqcup_{r \geq 0} \mc{M}_r \rightarrow \ell^\infty(\mc{C}^\sharp), \quad \mc{T}^\sharp(\mathrm{x}) := \mc{T}^\sharp(\nabla^{\E_{\mathrm{x}}}),
\]
where the right-hand side is understood by \eqref{equation:trace}. We then have the following:

\begin{prop}
\label{proposition:flat}
Let $(M, g)$ be an Anosov manifold of dimension $\geq 2$. Then the primitive trace map
\[
\mc{T}^\sharp : \bigsqcup_{r \geq 0} \mc{M}_r \rightarrow \ell^\infty(\mc{C}^\sharp), 
\]
is globally injective. Moreover, given $\mathrm{x}_0 = ([\E_0],[\nabla^{\E}_0]) \in \mc{M}_r$, the primitive trace map is weakly locally injective (in the sense of Theorem \ref{thm:weaklocal}) near $\mathrm{x}_0$ in the space $\mathbb{A}_{[\E_0]}$ of all unitary connections on $[\E_0]$.
\end{prop}

The previous Proposition \ref{proposition:flat} will be strengthened below when further assuming that $(M,g)$ has negative curvature (see Lemma \ref{lemma:small-curvature}): we will show that the primitive trace map is globally injective on connections with \emph{small} curvature.
The first part of Proposition \ref{proposition:flat} could be proved by purely algebraic arguments; nevertheless, we provide a proof with dynamical flavour, which is more in the spirit of the present article. We need a preliminary result:

\begin{lemma}\label{lemma:Pi_1inj}
Assume $(M,g)$ is Anosov and $\nabla^{\E}$ is a flat and unitary connection on the Hermitian vector bundle $\E \rightarrow M$. If $\X := (\pi^*\nabla^{\E})_X$, then:
\begin{itemize}
\item If $\X u = f$ with $f=f_0 + f_1 \in C^\infty(M,(\Omega_0 \oplus \Omega_1) \otimes \E)$ and $u \in C^\infty(SM,\pi^*\E)$, then $f_0 = 0$ and $u$ is of degree $0$. 
\item In particular, smooth invariant sections $u \in \ker \X|_{C^\infty(SM,\pi^*\E)}$ are of degree $0$.
\item The operator $\Pi_1^{\E}$ is s-injective.
\end{itemize}
\end{lemma}

\begin{proof}
The proof is based on the twisted Pestov identity for flat connections. 

\begin{lemma}[Twisted Pestov identity]
Let $u\in H^2(SM,\pi^*\E)$. Then
\[
\|\nabla^{\E} _{\V} \X u \|_{L^2}^2 = \|\X \nabla^{\E} _{\V} u  \|_{L^2}^2 - \langle R \nabla^{\E} _{\V} u, \nabla^{\E} _{\V} u \rangle_{L^2} + (n -1) \| \X u \|_{L^2}^2.
\]
\end{lemma}

For the notation, see \S \ref{sssection:twisted-fourier}; for a proof, we refer to \cite[Proposition 3.3]{Guillarmou-Paternain-Salo-Uhlmann-16}. An important point is that the following inequality holds for Anosov manifolds:
\[
\|\X \nabla^{\E} _{\V} u  \|_{L^2}^2 - \langle R \nabla^{\E} _{\V} u, \nabla^{\E} _{\V} u \rangle_{L^2} \geq C\|\nabla^{\E}_{\V} u\|^2_{L^2},
\]
where $C > 0$ is independent of $u$, see \cite[Theorem 7.2]{Paternain-Salo-Uhlmann-15} for the case of the trivial line bundle (the generalization to the twisted case is straightforward). We thus obtain:
\begin{equation}
\label{equation:inegalite}
\|\nabla^{\E} _{\V} \X u \|_{L^2}^2 \geq C\|\nabla^{\E}_{\V} u\|^2_{L^2} + (n -1) \| \X u \|_{L^2}^2.
\end{equation}

By assumption, $\X u = f_0 + f_1  \in C^\infty(M, (\Omega_0 \oplus \Omega_1) \otimes \E)$. Observe that this equation can be split into odd/even parts, namely: $\X u_{\mathrm{even}} = f_1, \X u_{\mathrm{odd}} = f_0$, and $u_{\mathrm{even},\mathrm{odd}} \in C^\infty(SM,\pi^*\E)$ have respective even/odd Fourier components. Applying \eqref{equation:inegalite} with $u_{\mathrm{odd}}$, we obtain $f_0 = 0$, $\X u_{\mathrm{odd}} = 0$ and $\nabla^{\E}_{\V} u_{\mathrm{odd}} = 0$, that is $u_{\mathrm{odd}}$ is of degree $0$ but $0$ is even so $u_{\mathrm{odd}} = 0$. As far as $u_{\mathrm{even}}$ is concerned, observe that $\nabla^{\E} _{\V} \X u_{\mathrm{even}} = \nabla^{\E} _{\V} f_1$ and:
\[
\|\nabla^{\E} _{\V} f_1\|_{L^2}^2 = \langle -\Delta_{\V}^{\E} f_1, f_1 \rangle_{L^2} = (n-1) \|f_1\|^2_{L^2}.
\]
Hence, applying the twisted Pestov identity with $u_{\mathrm{even}}$, we obtain:
\[
0 = \|\X \nabla^{\E} _{\V} u_{\mathrm{even}}  \|_{L^2}^2 - \langle R \nabla^{\E} _{\V} u_{\mathrm{even}}, \nabla^{\E} _{\V} u_{\mathrm{even}} \rangle_{L^2} \geq C\|\nabla^{\E}_{\V} u_{\mathrm{even}}\|^2_{L^2},
\]
that is $u_{\mathrm{even}}$ is of degree $0$. This proves the first point and the second point is a direct consequence of the first point.

For the last point, consider the equation $\X u = \pi_1^*f$. By the first point, $u$ is of degree zero, so $u = \pi_0^*u'$ for some $u' \in C^\infty(M, \E)$. Hence by \eqref{eq:pullback} we get $f = \nabla^{\E}u'$ and the conclusion follows by Lemma \ref{lemma:x-ray}.
\end{proof}

We can now prove Proposition \ref{proposition:flat}.

\begin{proof}
We prove the first part of Proposition \ref{proposition:flat}. We assume that $\mc{T}^\sharp(\mathrm{x}_1) =\mc{T}^\sharp(\mathrm{x}_2)$ or, equivalently, that $\mc{T}^\sharp(\nabla^{\E_{\mathrm{x}_1}}) = \mc{T}^\sharp(\nabla^{\E_{\mathrm{x}_2}})$. The exact Liv\v{s}ic cocycle Theorem \ref{theorem:weak} implies that the bundles $\pi^*\E_{\mathrm{x}_1}$ and $\pi^*\E_{\mathrm{x}_2}$ are isomorphic and yields the existence of a section $p \in C^\infty(SM,\mathrm{U}(\pi^*\E_{\mathrm{x}_2}, \pi^*\E_{\mathrm{x}_1}))$ such that: $C_{\mathrm{x}_1}(x,t) = p(\varphi_t x)C_{\mathrm{x}_2}(x,t)p(x)^{-1}$ for all $x \in \M, t \in \R$, which is equivalent to
\[
\pi^* \nabla^{\mathrm{Hom}(\nabla^{\E_{\mathrm{x}_2}}, \nabla^{\E_{\mathrm{x}_1}})}_X p = 0,
\]
where $\nabla^{\mathrm{Hom}(\nabla^{\E_{\mathrm{x}_2}}, \nabla^{\E_{\mathrm{x}_1}})}$ is the mixed connection induced by $\nabla^{\E_{\mathrm{x}_2}}$ and $\nabla^{\E_{\mathrm{x}_1}}$ on $\mathrm{Hom}(\E_{\mathrm{x}_2}, \E_{\mathrm{x}_1})$. Observe that by \eqref{equation:induced-curvature}, the curvature of $\nabla^{\mathrm{Hom}(\nabla^{\E_{\mathrm{x}_2}}, \nabla^{\E_{\mathrm{x}_1}})}$ vanishes as both curvatures $F_{\nabla^{\E_{\mathrm{x}_{1,2}}}}$ vanish. Applying Lemma \ref{lemma:Pi_1inj} with $\X := \pi^* \nabla^{\mathrm{Hom}(\nabla^{\E_{\mathrm{x}_2}}, \nabla^{\E_{\mathrm{x}_1}})}_X$ acting on the pullback bundle $\pi^*\mathrm{Hom}(\E_{\mathrm{x}_2}, \E_{\mathrm{x}_1})$, we get that $p$ is of degree $0$, which is equivalent to the fact that the connections are gauge-equivalent.

As to the second part of Proposition \ref{proposition:flat}, by Theorem \ref{thm:weaklocal} it is a straightforward consequence of the s-injectivity of $\Pi_1^{\End(\E_0)}$ and the fact that elements of $\ker(\pi^*\nabla^{\End(\E_0)})_X|_{C^\infty})$ are of degree zero, which follows from Lemma \ref{lemma:Pi_1inj} (items 2 and 3).
\end{proof}

\subsubsection{Negative sectional curvature}

\label{sssection:negative}

We now assume further that the Riemannian manifold $(M,g)$ has negative sectional curvature. We introduce the following condition:

\begin{definition}
We say that the pair of connections $(\nabla^{\E_1},\nabla^{\E_2})$ satisfies the \emph{spectral condition} if the mixed connection $\nabla^{\mathrm{Hom}(\nabla^{\E_1},\nabla^{\E_2})}$ has no non-trivial twisted CKTs. 
\end{definition}

This condition is symmetric in the pair $(\nabla^{\E_1},\nabla^{\E_2})$. Observe that by \eqref{equation:lien}, the previous condition is invariant by changing one of the two connections by $p^* \nabla^{\E_i}$, for some vector bundle isomorphism $p$, and thus this condition descends to the moduli space. We then define
\begin{equation}\label{eq:setS}
\mathbf{S} \subset \mathbb{A} \times \mathbb{A},
\end{equation}
the subspace of all pairs of equivalence classes of connections satisfying the spectral condition. The set $\mathbf{S}$ is open and dense (for the $C^N_*$-topology, $N \gg 1$) as shown in Appendix \ref{appendix:ckts}. Moreover, it also contains all pairs of connections with \emph{small curvature}, that is if
\[
\Omega_\eps := \left\{ \mathrm{x} = ([\E],[\nabla^{\E}]) \in \mathbb{A}, ~~~ \|F_{\nabla^{\E}}\|_{L^\infty(M,\Lambda^2 T^*M \otimes \End(\E))} < \eps \right\} \subset \mathbb{A},
\]
then we have the following:

\begin{lemma}
\label{lemma:small-curvature}
Let $(M,g)$ be a negatively-curved Riemannian manifold of dimension $\geq 2$ and let $-\kappa < 0$ be an upper bound for the sectional curvature. There exists $\eps(n,\kappa) > 0$ such that:
\[
\Omega_{\eps(n,\kappa)} \times \Omega_{\eps(n,\kappa)} \subset \mathbf{S}.
\]
One can take $\eps(n,\kappa) =\dfrac{\kappa \sqrt{n-1}}{4}$.
\end{lemma}

\begin{proof}
We start by a preliminary discussion. Given a Hermitian vector bundle $\E \to M$ with metric $\langle \bullet, \bullet \rangle$, a unitary connection $\nabla^{\E}$, we introduce, following \cite[Section 3]{Guillarmou-Paternain-Salo-Uhlmann-16}, an operator $\mc{F}^{\E} \in C^\infty(SM,\mc{N} \otimes \End_{\mathrm{sk}}(\E))$ (recall that $\mc{N}$ is the normal bundle, see \S\ref{sssection:line-bundle}) defined by the equality:
\begin{equation}
\label{equation:twisted-curvature}
 \langle \mc{F}^{\E}(x,v) e, w \otimes e' \rangle := \langle F_{\nabla^{\E}}(v,w)e,e' \rangle,
\end{equation}
where $F_{\nabla^{\E}}$ is the connection of $\nabla^{\E}$, and $(x,v) \in SM, e, e' \in \E_x, w \in \mc{N}(x,v)$, and the metric on the left-hand side is the natural extension of the metric $\langle \bullet, \bullet \rangle$ on $\E$ to $\mc{N} \otimes \E$ by tensoring with the metric $g$. A straightforward computation shows that:
\begin{equation}
\label{equation:bad-bound}
\|\mc{F}^{\E}\|_{L^\infty(SM, \mc{N} \otimes \End_{\mathrm{sk}}(\E))} \leq \|F_{\nabla^{\E}}\|_{L^\infty(M,\Lambda^2 T^*M \otimes \End(\E))}. 
\end{equation}
Now, let $\nabla^{\E_1}$ and $\nabla^{\E_2}$ be two unitary connections, $\nabla^{\mathrm{Hom}(\nabla^{\E_1}, \nabla^{\E_2})}$ be the mixed connection and $\mc{F}^{\mathrm{Hom(\E_1,\E_2)}}$ be the operator induced by the mixed connection as in \eqref{equation:twisted-curvature}. Observe that by \eqref{equation:induced-curvature} and \eqref{equation:bad-bound}, we get:
\begin{equation}
\label{equation:bad-bound-2}
\begin{split}
\|\mc{F}^{\mathrm{Hom(\E_1,\E_2)}}\|_{L^\infty}  \leq \|F_{\nabla^{\mathrm{Hom}(\nabla^{\E_1}, \nabla^{\E_2})}}\|_{L^\infty} \leq \|F_{\nabla^{\E_1}}\|_{L^\infty} + \|F_{\nabla^{\E_2}}\|_{L^\infty} < 2\eps(n,\kappa).
\end{split}
\end{equation}
By \cite[Theorem 4.5]{Guillarmou-Paternain-Salo-Uhlmann-16}, if $m \geq 1$ satisfies:
\begin{equation}
\label{equation:ckts}
m(m+n-2) \geq 4 \dfrac{\|\mc{F}^{\mathrm{Hom(\E_1,\E_2)}}\|^2_{L^\infty}}{\kappa^2},
\end{equation}
then there are no twisted CKTs of degree $m$ (for the connection $\nabla^{\mathrm{Hom}(\nabla^{\E_1}, \nabla^{\E_2})}$). Now, the choice of $\eps(n,\kappa) > 0$ combined with \eqref{equation:bad-bound-2} guarantees that \eqref{equation:ckts} is satisfied for any $m \geq 1$.
\end{proof}

We then have the following statement:

\begin{prop}
\label{proposition:negative}
Let $(M,g)$ be a negatively-curved Riemannian manifold of dimension $\geq 2$. Let $(\mathfrak{a}, \mathfrak{a}') \in \mathbf{S}$ such that $\mc{T}^\sharp(\mathfrak{a}) = \mc{T}^\sharp(\mathfrak{a}')$. Then $\mathfrak{a} = \mathfrak{a}'$.
\end{prop}

In other words, two connections satisfying the spectral condition and whose images by the primitive trace map are equal, are actually gauge-equivalent.

\begin{proof}
Consider two representatives $\nabla^{\E_1} \in \mathfrak{a}$ and $\nabla^{\E_2} \in \mathfrak{a}'$. The exact Liv\v{s}ic cocycle Theorem \ref{theorem:weak} provides a section $p \in C^\infty(SM,\mathrm{U}(\pi^*\E_2,\pi^*\E_1))$ such that:
\[
\pi^* \nabla^{\mathrm{Hom}(\nabla^{\E_{2}}, \nabla^{\E_{1}})}_X p = 0.
\]
By assumption, $(M,g)$ has negative curvature and thus $p$ has finite Fourier degree by \cite[Theorem 4.1]{Guillarmou-Paternain-Salo-Uhlmann-16}. Moreover, since $ \nabla^{\mathrm{Hom}(\nabla^{\E_{2}}, \nabla^{\E_{1}})}$ has no non-trivial twisted CKTs, $p$ is of degree $0$ (see \cite[Theorem 5.1]{Guillarmou-Paternain-Salo-Uhlmann-16}). This shows that the connections are gauge-equivalent.
\end{proof}

\subsubsection{Topological results}

\label{sssection:topology}

In this section we prove a global \emph{topological} uniqueness result for the primitive trace map. 

\begin{prop}\label{proposition:topology}
	Let $(M, g)$ be an orientable Anosov manifold. If $\mathrm{x}_i = ([\E_i], [\nabla^{\E_i}]) \in \mathbb{A}$ for $i = 1, 2$, then $\mc{T}^\sharp(\mathrm{x}_1) = \mc{T}^\sharp(\mathrm{x}_2)$ implies:
	
	\begin{itemize}
			\item If $\dim M$ is odd or more generally the Euler characteristic $\chi(M) = 0$ vanishes, then $\E_1 \simeq \E_2$ are isomorphic as vector bundles.
			\item If $\dim M = 2d$ for some $d \in \mathbb{N}$ and $\chi(M) \neq 0$, then
				\begin{itemize}
					\item The Chern classes satisfy $c_i(\E_1) = c_i(\E_2)$ for $i = 1, \dotso, d - 1$; also $c_d(\E_1) - c_d(\E_2) \in H^{2d}(M; \mathbb{Z}) \cong \mathbb{Z}$ is a multiple of $\chi(M)$.
					\item If the even cohomology ring $H^{\mathrm{even}}(M; \mathbb{Z})$ is torsion-free, and the rank of the bundles is less than $d$ or more generally $c_d(\E_1) = c_d(\E_2)$, then $\E_1$ and $\E_2$ are stably isomorphic, i.e. there is an $m \geq 0$ such that $\E_1 \oplus \mathbb{C}^m \simeq \E_2 \oplus \mathbb{C}^m$.
				\end{itemize}
	\end{itemize}	
\end{prop}
\begin{proof}
	As a direct consequence of Theorem \ref{theorem:weak}, from $\mc{T}^\sharp(\mathrm{x}_1) = \mc{T}^\sharp(\mathrm{x}_2)$ we obtain that $\pi^*\E_1 \simeq \pi^*\E_2$ are isomorphic.
	
	If $(M, g)$ has a vanishing Euler characteristic, there is a non-vanishing vector field $V \in C^\infty(M, TM)$ (see \cite[Chapter 11]{Bott-Tu-82}), that we normalise to unit norm using the metric $g$ and hence see as a section of $SM$. Then since $\pi \circ V = \id_M$, we get
	\[\E_1 \simeq V^*\pi^*\E_1 \simeq V^*\pi^* \E_2 \simeq \E_2,\]
	completing the proof of the first item. 
	
	The first point of the second item is immediate after an application of the Gysin exact sequence \cite[Proposition 14.33]{Bott-Tu-82} for the sphere bundle $SM$. The second point follows from the first one and the fact that the Chern character gives an isomorphism between the rational $K$-theory and even rational cohomology, see \cite[Proposition 4.5]{Hatcher-17}.
\end{proof}

It is not known to the authors if further results hold as to the injectivity of $\pi^* : \mathrm{Vect}(M) \rightarrow \mathrm{Vect}(SM)$ in even dimensions ($\dim M \geq 4$).

\appendix

\section{Generic absence of CKTs for the mixed connection}

\label{appendix:ckts}

We assume that $(M,g)$ has negative curvature in this section. For $i=1,2$, we let $\nabla^{\E_i}$ be a (smooth) unitary connection on the Hermitian vector bundle $\E_i \to M$. By \cite[Theorem 4.5]{Guillarmou-Paternain-Salo-Uhlmann-16}, there exists $m_0 \gg 1$ (depending on the dimension and the sup norm of the curvatures of $\nabla^{\E_{1,2}}$) such that the mixed connection $\nabla^{\mathrm{Hom}(\nabla^{\E_{1}},\nabla^{\E_{2}})}$ has no non-trivial twisted CKTs of degree $m \geq m_0$. This property is stable by any small perturbation of $\nabla^{\E_i}$ in the $C^1$-topology (so that the curvature is well-defined). From this, we deduce by standard elliptic theory that if the pair $(\nabla^{\E_1},\nabla^{\E_2})$ satisfies the spectral condition, then any small perturbation (in the $C^1$-topology) will also satisfy the spectral condition: indeed, absence of twisted CKTs of degree $m$ is equivalent to the invertibility of a natural Laplacian operator acting on $\Omega_m \otimes \mathrm{Hom}(\E_1,\E_2)$ and this is an open property as there is only \emph{a priori} a finite number of integers $m \leq m_0$ to check, see \cite{Cekic-Lefeuvre-20} for further details. This shows that the set $\mathbf{S}$ defined in \eqref{eq:setS} is open.

We now show that it is dense. More precisely, we show the following:

\begin{lemma}
Let $\nabla^{\E_{1,2}}$ be a smooth unitary connection on $\E_{1,2} \to M$ and assume that the mixed connection $\nabla^{\mathrm{Hom}(\nabla^{\E_{1}},\nabla^{\E_{2}})}$ admits non-trivial twisted CKTs of degree $m \geq 1$. Then, for any $\eps > 0$ and $k_0 \gg 1$ large enough, there exists a small perturbation $\nabla^{\E_2}+\Gamma_2$, where $\Gamma_2 \in C^\infty(M,T^*M \otimes \End_{\mathrm{sk}}(\E_2))$ and $\|\Gamma_2\|_{C^{k_0}} \leq \eps$, such that the mixed connection induced by the pair $(\nabla^{\E_1},\nabla^{\E_2}+\Gamma_2)$ has no non-trivial twisted CKTs of degree $m$.
\end{lemma}

\begin{proof}
We let $\X^{\Gamma_2} := \pi^* \nabla^{\mathrm{Hom}(\nabla^{\E_{1}},\nabla^{\E_{2}}+\Gamma_2)}_X$, where $\Gamma_2$ is small and, as in \S\ref{sssection:twisted-fourier}, we define:
\[
\X^{\Gamma_2}_\pm : C^\infty(M,\Omega_m \otimes \mathrm{Hom}(\E_1,\E_2)) \to  C^\infty(M,\Omega_{m \pm 1} \otimes \mathrm{Hom}(\E_1,\E_2)),
\]
and $\X_\pm := \X^{\Gamma_2 = 0}$. We also let $\Delta^{\Gamma_2} := (\X^{\Gamma_2}_+)^* \X^{\Gamma_2}_+$, the Laplacian-type operator acting on sections of $\Omega_m \otimes \mathrm{Hom}(\E_1,\E_2)$. The existence of twisted CKTs of degree $m$ for the mixed connection $\nabla^{\mathrm{Hom}(\nabla^{\E_{1}},\nabla^{\E_{2}})}$ is equivalent to the existence of a non-trivial kernel for $\Delta^{\Gamma_2=0}$, see \cite[Section 4]{Cekic-Lefeuvre-20} for further details.

Given $\gamma \subset \C$, a small contour in $\C$ around $0$ (containing only the eigenvalue $0$ of $\Delta^{\Gamma_2=0}$), we let $\lambda^{\Gamma_2}$ be the sum of the eigenvalues of $\Delta^{\Gamma_2}$ inside $\gamma$. We have that $C^{k_0} \ni \Gamma_2 \mapsto \lambda^{\Gamma_2}$ is at least $C^3$, when $k_0 \gg 1$ is chosen large enough. We have $\lambda^{\Gamma_2=0} = 0$ and $\dd \lambda^{\Gamma_2} = 0$, see \cite[Section 4]{Cekic-Lefeuvre-20}. Moreover, it was shown in \cite[Lemma 4.2]{Cekic-Lefeuvre-20} that the second derivative is obtained as
\[
\dd^2\lambda^{\Gamma_2=0}(A_2,A_2) = \sum_{i=1}^d \|\pi_{\ker \X_-} \left[\dd \X^{\Gamma_2=0} (A_2)\right]_+ u_i\|^2_{L^2} =  \sum_{i=1}^d \|\pi_{\ker \X_-} (A_2)_+ u_i\|^2_{L^2},
\]
where $\left\{u_1, ..., u_d\right\}$ is an $L^2$-orthonormal basis of $\ker \Delta^{\Gamma_2=0}$ (each $u_i$ is a smooth section of $\Omega_m \otimes \mathrm{Hom}(\E_1,\E_2) \to M$), $A_2 \in C^\infty(M,T^*M \otimes \End_{\mathrm{sk}}(\E_2))$ and $(A_2)_+$ is the positive part of the operator (see \cite[Section 2.2]{Cekic-Lefeuvre-20}), and $\pi_{\ker \X_-}$ is the $L^2$-orthogonal projection onto $\ker \X_-|_{L^2(M,\Omega_{m+1} \otimes \mathrm{Hom}(\E_1,\E_2))}$. The formula for $\dd \X^{\Gamma_2=0}(A_2) = A_2$ can be directly read off \eqref{eq:localhom}.

It thus suffices to produce a small perturbation such that this second derivative is positive. We can argue by contradiction and assume that for \emph{any} perturbation, this second derivative vanishes. Following \emph{verbatim} the arguments of \cite[Section 4.3]{Cekic-Lefeuvre-20}, and using that the operator $\X_-$ is of \emph{uniform divergence type} (see \cite[Section 3]{Cekic-Lefeuvre-20}), this would imply that:
\[
\langle u_1(x), (A_2)_- w \rangle_x = 0,
\]
for all $x \in M$, $A_2 \in C^\infty(M,T^*M \otimes \End_{\mathrm{sk}}(\E_2)), w \in \Omega_{m+1}(x) \otimes \mathrm{Hom}(\E_1(x),\E_2(x))$. It thus suffices to show, as in \cite[Lemma 4.8]{Cekic-Lefeuvre-20} that this forces $u_1$ to be equal to zero, which is a contradiction since $\|u_1\|_{L^2}=1$.

We now fix an arbitrary point $x_0$. We can write $u_1(x_0)= \sum_{i=1}^k p_i \otimes s_i$, where $k = \rk(\mathrm{Hom}(\E_1,\E_2))$, $\left\{s_1,...,s_k\right\}$ is an orthonormal basis of $\mathrm{Hom}(\E_1,\E_2)$ at $x_0$ and $p_i \in \Omega_{m}(x_0)$. We write $(\mathbf{e}_1, ... \mathbf{e}_n)$ for an orthonormal basis of $T_{x_0}M$. Taking $A_2 = i \mathbbm{1}_{\E_2} \otimes \e_1^*$, $w = f \otimes s_{i_0}$ where $f \in \Omega_{m + 1}(x_0)$ is such that $(v_1)_- f = p_{i_0}$ (here $v_1 = \mathbf{e}_1^*(v)$ and $(v_1)_-$ is the minus operator associated; this operator is surjective by \cite[Lemma 2.4]{Cekic-Lefeuvre-20}), we get as in \cite[Lemma 4.8]{Cekic-Lefeuvre-20}:
\[
\langle u_1(x_0), (A_2)_- w \rangle_{x_0} = \sum_{i=1}^k \langle p_i \otimes s_i, (v_1)_- f \otimes i \mathbbm{1}_{\E_2}  \cdot s_{i_0} \rangle  = i \|p_{i_0}\|^2 \|s_{i_0}\|^2 = i \|p_{i_0}\|^2_{L^2} = 0.
\]
Hence $u_1 \equiv 0$. This concludes the proof.
\end{proof}

\bibliographystyle{alpha}
\bibliography{Biblio}
\end{document}